\documentclass[12pt]{amsart}
\usepackage{graphicx}
\usepackage{hyperref}
\usepackage{setspace}
\usepackage{enumerate}
\usepackage{amsmath,amsfonts,amssymb,amscd}
\usepackage{bbm}
\usepackage{amsthm}
\newtheorem{thm}{Theorem}[section]

\newtheorem{lem}[thm]{Lemma}
\newtheorem{prop}[thm]{Proposition}

\theoremstyle{definition}
\newtheorem{defn}[thm]{Definition}

\theoremstyle{remark}
\newtheorem{rem}[thm]{Remark}
\numberwithin{equation}{section}

\newtheorem*{lem*}{Lemma}

\let\emptyset\varnothing

\DeclareMathOperator{\supp}{supp}

\newcommand{\ZZ}{\mathbb{Z}}

\newcommand{\NN}{\mathbb{N}}

\newcommand{\RR}{\mathbb{R}}
\newcommand{\CC}{\mathbb{C}}

\newcommand{\s}{\mathbf{s}}
\renewcommand{\t}{\mathbf{t}}

\renewcommand{\b}{\mathbf{b}}
\newcommand{\x}{\mathbf{x}}
\renewcommand{\u}{\mathbf{u}}
\renewcommand{\v}{\mathbf{v}}
\newcommand{\w}{\mathbf{w}}
\newcommand{\z}{\mathbf{z}}
\newcommand{\y}{\mathbf{y}}

\newcommand{\bpsi}{\boldsymbol{\psi}}

\newcommand*{\medcap}{\mathbin{\scalebox{1.5}{\ensuremath{\cap}}}}

\newcommand{\tsp}{\thinspace}

\begin{document}

\title[Diophantine equations in primes]{Diophantine equations in primes:
density of prime points on affine hypersurfaces II}

\author{Shuntaro Yamagishi}
\address{IST Austria\\
Am Campus 1\\
3400 Klosterneuburg\\
Austria}
\email{shuntaro.yamagishi@ist.ac.at}
\indent

\date{Revised on \today}

\begin{abstract}
Let $F \in \mathbb{Z}[x_1, \ldots, x_n]$ be a homogeneous form of degree $d \geq 2$, and let $V_F^*$ denote the singular locus of the affine variety
$V(F) = \{ \mathbf{z} \in {\mathbb{A}}^n_{\CC}: F(\mathbf{z}) = 0 \}$. 
In this paper, we prove the existence of integer solutions with prime coordinates to the equation $F(x_1, \ldots, x_n) = 0$ provided $F$ satisfies suitable local conditions
and $n -  \dim V_F^* > d^2 (2d-1) 4^{d} \left( 33 + \frac{37}{2d-1 } \right) + \frac{3 (32 d) (33d) (6d + 1)}{d-1} 2^{d-1} + (12 d + 1)   (4 (d-1) 2^d + 1)$.
The result is obtained by using the Heath-Brown identity and optimizing various parts of the argument in the author's previous work, which
made use of the Vaughan identity and required $n - \dim V_F^* \geq 2^8 3^4 5^2 d^3 (2d-1)^2 4^{d}$.
\end{abstract}



\maketitle

\section{Introduction}
Let $F \in \mathbb{Z}[x_1, \ldots, x_n]$ be a homogeneous form of degree $d \geq 2$.
In this paper, we are interested in the problem of finding prime solutions, which are integer solutions with prime coordinates, to the equation
\begin{equation}
\label{eqn main}
F(x_1, \ldots, x_n) = 0.
\end{equation}
First we introduce some notation in order to state our result. 
Let $\wp$ denote the set of prime numbers. Let $\mathbb{Z}_p^{\times}$ be the units of $p$-adic integers. We consider the following condition.

Local conditions ($\star$): The equation (\ref{eqn main})
has a non-singular real solution in $(0,1)^{n}$, and also
has a non-singular solution in $( \mathbb{Z}_p^{\times})^n$ for every $p \in \wp$.
\newline
\newline
Let $V^*_{F}$ denote the singular locus of $V(F) = \{  \mathbf{z} \in \mathbb{A}^n_{\CC}:  F(\mathbf{z}) = 0  \}$, i.e. it is the affine variety
\begin{equation}
\label{sing loc}
V_{F}^* =  \{ \mathbf{z} \in \mathbb{A}_{ \mathbb{C} }^{n}:  \nabla F(\mathbf{z}) = \mathbf{0} \},
\end{equation}
where
$\nabla F = \left( \frac{\partial F}{\partial x_1}, \ldots,  \frac{\partial F}{\partial x_n} \right)$. We also define the following class of smooth weights.
\begin{defn}
Let $\delta, \mathfrak{c} > 0$ and $M_0 \in \mathbb{Z}_{\geq 0}$.
We define $\mathcal{S}^+( \delta ;  M_0; \mathfrak{c})$ to be the set of smooth functions $\omega :\mathbb{R} \rightarrow [0, \infty)$ satisfying
\begin{enumerate}[(i)]
\item $\supp \omega = [-\delta, \delta]$,
\item for any $k \in \{ 0, \ldots, M_0 \}$ we have
$\| \partial^{k} \omega / \partial x^{k} \|_{L^\infty(\mathbb{R})} \leq \mathfrak{c}$.
\end{enumerate}
\end{defn}
Let $X > 1$. Let $\mathbf{x}_0 = (x_{0,1}, \ldots, x_{0,n}) \in (0,1)^n$ and $\omega \in \mathcal{S}^+( \delta ;  n; \mathfrak{c})$,
and we define
\begin{eqnarray}
\label{defn of psi}
\psi_i(x_i) = \omega \left( \frac{x_i}{X} - x_{0,i}  \right) \quad (1 \leq i \leq n).
\end{eqnarray}
Let $\Lambda$ denote the von Mangoldt function, where $\Lambda(x)$ is $\log p$ if $x$ is a power of $p \in \wp$ and $0$ otherwise.
Given $\mathcal{X} \subseteq \mathbb{C}^n$,  we let $\mathbbm{1}_{ \mathcal{X}}$ be the characteristic function of $\mathcal{X}$, i.e.
$\mathbbm{1}_{ \mathcal{X}} (\x) = 1$ if $\x \in \mathcal{X}$ and $0$ otherwise.


In this paper, we prove the following result by the Hardy-Littlewood circle method.
\begin{thm}
\label{mainthm}
Let $F \in \mathbb{Z}[x_1, \ldots, x_n]$ be a homogeneous form of degree $d \geq 2$ satisfying
\begin{eqnarray}
\label{codim of F}
n -  \dim V_F^* &>& d^2 (2d-1) 4^{d} \left( 33 + \frac{37}{2d-1 } \right)
\\
\notag
&& + \frac{3 (32 d) (33d) (6d + 1)}{d-1} 2^{d-1}  + (12 d + 1)   (4 (d-1) 2^d + 1)
\end{eqnarray}
and the local conditions \textnormal{($\star$)}. Let $\mathbf{x}_0 \in (0,1)^n$ be a non-singular real solution to the equation (\ref{eqn main}).
Let $\delta, \mathfrak{c} > 0$ and $\omega \in \mathcal{S}^+( \delta ; n; \mathfrak{c})$, where $\delta$ is sufficiently small with respect to $F$ and $\mathbf{x}_0$, and let
$\psi_1, \dots, \psi_n$ be as in (\ref{defn of psi}).
Then for any $A > 0$ we have
\begin{equation}
\label{asympt}
\sum_{\mathbf{x} \in [0,X]^n } \prod_{1 \leq i \leq n} \psi_i(x_i)  \Lambda(x_i) \cdot \mathbbm{1}_{ V(F) } (\mathbf{x}) = c(F; \omega, \mathbf{x}_0) \, X^{n-d} + O \left(
\frac{X^{n-d}}{(\log X)^A} \right),
\end{equation}
where $c(F; \omega, \mathbf{x}_0) > 0$ is a constant depending only on $F$, $\omega$ and $\mathbf{x}_0$.
\end{thm}
The result improves on the author's previous work \cite{Y2} in which the above asymptotic formula (\ref{asympt}) was established with
$$
n -  \dim V_F^* \geq 2^8 3^4 5^2 d^3 (2d-1)^2 4^{d}
$$
instead of (\ref{codim of F}). The first result in this direction for a general homogeneous form of degree $d$ was obtained by Cook and Magyar \cite{CM}, where they required
$n -  \dim V_F^*$ to be an exponential tower in $d$. For results when $d=2$, we refer the reader to see
\cite{Coo}, \cite{Gree}, \cite{Li} and \cite{Zh}.
One day after the author's work \cite{Y2} appeared on the arXiv in May 2021
(we note that \cite{Y2} was submitted in January 2019 and accepted in March 2021), a preprint by
Liu and Zhao \cite{LZ} also  appeared on the arXiv, in which  they developed an intricate, and different, method establishing the result with
$$
n \geq 16 d^2 4^d
$$
when $\dim V_F^* = 0$. In particular,  their result saves at least a factor of $d$ compared to (\ref{codim of F}). We remark that the results in \cite{CM} and \cite{LZ} are obtained for systems of forms of equal degree and without smooth weights.
As mentioned in \cite{Y2}, we expect that these differences can be overcome with additional technical effort. Also we may deduce the existence of prime solutions from Theorem
\ref{mainthm}; we refer the reader to see \cite{Y2} for the details.

The main ideas of this paper go back to the author's previous work \cite{Y2}.
The result is proved by incorporating several technical simplifications and improvements
to optimize various parts of the argument in \cite{Y2}, one of the changes being the use of
the Heath-Brown identity (introduced in (\ref{HB iden})), instead of the Vaughan identity.
In Sections \ref{prem major} and \ref{prem minor}, we prepare for the major and minor arcs analyses respectively, and treat them in the subsequent sections.
We complete the proof of Theorem \ref{mainthm} in Section \ref{conc}. The following exponential sum
\begin{equation}
\label{def S}
S(\alpha) = \sum_{\mathbf{x} \in [0, X]^n } \prod_{1 \leq i \leq n} \psi_i(x_i)  \Lambda(x_i) \cdot e( \alpha F(\mathbf{x}) ),
\end{equation}
where $e(z) = e^{2 \pi i z}$ for $z \in \mathbb{R}$,
is central to our study. In particular, it will be relevant in Sections \ref{major}--\ref{conc}. 

\textit{Notation.} Throughout the paper, we use $\ll$ and $\gg$ to denote Vinogradov's well-known notation, i.e.
the statement $f \ll g$ means there exists a positive constant $C$ (it may depend on parameters
which are regarded as fixed) such
that $|f| \leq  C g$ for all values under consideration, and  the statement $g \gg f$ is equivalent to $f \ll g$.
The statement $f \asymp g$ means $f \ll g$ and $f \gg g$.
We also make use of the $O$-notation; the statement $f = O(g)$ is equivalent to $f \ll g$.
All the implicit constants in $\ll, \gg$, $\asymp$ and the $O$-notation are independent of $X$. Given a homogeneous form
$G \in \mathbb{C}[x_1, \ldots, x_m]$, we denote
$$
\textnormal{codim} \tsp V_G^* =  m - \dim V_G^*.
$$
Also for $f \in \mathbb{C}[x_1, \ldots, x_m]$, we let $f^{[j]}$ denote the degree $j$ homogeneous portion of $f$.
By $\mathbb{A}^m_{\CC}$ we mean the affine $m$-space over $\CC$, i.e. the set $\CC^m$ with the Zariski topolgy, and for any
$f \in \mathbb{C}[x_1, \ldots, x_m]$
we denote
$$
V(f) = \{ \x \in \mathbb{A}^m_{\CC}: f(\x) = 0  \}.
$$
Given a function $\varrho: \RR \to \RR$, we let $\supp \varrho = \overline{\{ z \in \RR:  \varrho(z) \neq 0 \}}$, i.e.
the closure of the set $\{ z \in \RR:  \varrho(z) \neq 0 \}$. We shall refer to $\mathfrak{B} \subseteq \mathbb{R}^m$ as a box if $\mathfrak{B}$ is of the form
$
\mathfrak{B} = I_1 \times \cdots \times I_m,
$
where each $I_j$ is a closed, open or half open/closed interval. Finally, $\varepsilon$ will always be a sufficiently small (with respect to relevant parameters) positive number, even when it is not explicitly stated so, and we do not assume it to be fixed line by line. In other words, we allow it to vary from one exression/inequality to the next.
For example, we may write $X^{2^n \varepsilon} \ll  X^{\varepsilon} \ll X^{1 - \varepsilon}$ and the three $\varepsilon$ that appear
should not necessarily be interpreted as the same real number.

\textit{Acknowledgements.} The author was supported by the NWO Veni Grant \texttt{016.Veni.192.047}
and a FWF grant (DOI 10.55776/P32428) while working on this paper.

\section{Preliminaries for the major arcs analysis}
\label{prem major}
Let $q \in \mathbb{N}$ and $\chi^0$ be the principal character modulo $q$. We consider $\chi^0$ as a primitive character modulo $q$ only when $q=1$, and not otherwise. We
will use the following zero-free region estimate of the Dirichlet $L$-functions. The proof is given in \cite[Theorem 7.1]{Y2} with the main input being \cite{I}.
\begin{thm}
\label{thm zero free}
Let $s = \sigma + it$, $M \geq 3$, $T \geq 0$ and $\mathfrak{L} = \log M(T + 3)$.
Then there exists an absolute constant $c_1 > 0$ such that $L(s, \chi) \not = 0$ whenever
\begin{eqnarray}
\label{region}
\sigma \geq 1 - \frac{c_1}{\log M + ( \mathfrak{L} \log 2 \mathfrak{L}  )^{3/4}} \ \ \textnormal{  and  } \ \  |t| \leq T
\end{eqnarray}
for all primitive characters $\chi$ of modulus $q \leq M$, with the possible exception of at most one primitive character
$\widetilde{\chi}$ modulo $\widetilde{r}$.
If such $\widetilde{\chi}$ exists, then $L(s, \widetilde{\chi})$ has at most one zero in (\ref{region}) and the exceptional zero $\widetilde{\beta}$ is real and simple, and
$\widetilde{r}$ satisfies $M \geq \widetilde{r} \gg_A (\log M)^A$ for any $A>0$.
\end{thm}

\begin{rem} \label{remZFR}
Let $M = X^{\vartheta_0}$ for a fixed $\vartheta_0 > 0$ and $X$ sufficiently large. Then we have $\widetilde{r} \gg_A (\log X)^{A}$ for any $A > 0$
(the subscript in $\gg_A$ is to indicate that the implicit constant depends on $A$).
In particular, the exceptional zero will not occur for the Dirichlet $L$-functions associated to primitive characters of modulus $1 \leq q \leq (\log X)^D$ with $D > 0$. It
also follows from Siegel's theorem (see \cite[Remark 7.2]{Y2}) that
\begin{equation}
\label{excepZZ}
0 < 1 - \widetilde{\beta} < \frac12.
\end{equation}
\end{rem}
Let
$$
B_T = \{ s = \sigma + it : 0 \leq \sigma \leq 1 ,  |t| \leq T \}.
$$ Let
$N_{\chi}(\alpha, T)$ denote the number of zeros, with multiplicity, of $L(s, \chi)$ in the rectangle
$\alpha \leq \sigma \leq 1$ and $|t| \leq T$.
For the remainder of this paper, we let $\sum'_{\rho}$ denote the sum over the non-exceptional zeros
(with respect to Theorem \ref{thm zero free}), with multiplicity, of $L(s, \chi)$ in $B_T$ and let $\textnormal{Re }\rho = \beta$.
Let $\sum^*_{\chi (\textnormal{mod} \tsp r)}$  
denote the sum over the primitive characters $\chi$ modulo $r$.
The proof of the following lemma is given in \cite[Lemma 7.3]{Y2} with the main inputs being \cite{Hux} and \cite{J}.
\begin{lem}
\label{Gbdd}
Let $\xi, D, A > 0$. Let $M = X^{\vartheta_0}$ and $T = X^{\gamma}$, where $\vartheta_0, \gamma > 0$ satisfy
$2\vartheta_0 + \gamma < \frac{5}{12}$. Then we have
\begin{align}
\textnormal{i)}& \ \ \ \ \  \ \ \ \ \ \ \  \notag \sum_{  1 \leq r \leq M} \  \sideset{}{^*}\sum_{\chi (\textnormal{mod} \tsp r)} \sideset{}{'}\sum_{\rho} (\xi X)^{\beta - 1}
\ll 1 ,
\\
\notag
\\
\notag
\textnormal{ii)}&  \ \  \ \ \ \
\sum_{  1 \leq r \leq (\log X)^D } \  \sideset{}{^*}\sum_{\chi (\textnormal{mod} \tsp r)} \sum_{ \substack{ \rho \in B_T \\ L(\rho, \chi) = 0  } } (\xi X)^{\beta - 1} \ll
(\log X)^{-A},
\end{align}
where the sum $\sum_{ \substack{ \rho \in B_{T}  \\  L(\rho, \chi) = 0  }}$ is over all the zeros, with multiplicity, of $L(s, \chi)$ in $B_{T}$.
Here the implicit constants may depend on $\xi, D$ and $A$.
\end{lem}

Let $\phi$ be Euler's totient function.
Let $\mathbb{U}_q = (\mathbb{Z}/q \mathbb{Z})^*$, the group of units in $\mathbb{Z}/q \mathbb{Z}$.
The following lemma is proved in \cite[Lemma 7.4]{Y2}.
\begin{lem}
\label{exp sum modp}
Let $F \in \mathbb{Z}[x_1, \ldots, x_n]$ be a homogeneous form of degree $d \geq 2$.
Let $q \in \mathbb{N}$ and $a \in \mathbb{Z}$ be such that $\gcd(a,q)=1$.
Let $\chi_1, \ldots, \chi_n$ be any Dirichlet characters modulo $q$, and let
\begin{eqnarray}
\label{calS}
\mathcal{S} (\chi_1, \ldots, \chi_n) = \sum_{\mathbf{h} \in (\mathbb{Z}/q \mathbb{Z})^n  } \chi_1(h_1) \cdots \chi_n(h_n)
\ e \left( \frac{a}{q} F(h_1, \ldots, h_n)  \right).
\end{eqnarray}
Then for any $\varepsilon > 0$ we have
$$
|\mathcal{S} (\chi_1, \ldots, \chi_n) | \ll q^{n - \frac{1}{2(2d-1) 4^d} \textnormal{codim} \tsp V_{F}^* + \varepsilon},
$$
where the implicit constant is independent of the choice of $\chi_1, \ldots, \chi_n$.
\end{lem}

The following proposition on oscillatory integrals is proved as the main result in \cite{Y}.
\begin{prop}
\label{prop osc int}
Let $F \in \mathbb{Z}[x_1, \ldots, x_n]$ be a homogeneous form of degree $d \geq 2$ satisfying $n - \dim V_{F}^* > 4$.
Let $r_j \in [-1, 0]$ and $t_j \in \mathbb{R}$ $(1 \leq j \leq n)$.
Suppose $\mathbf{x}_0 = (x_{0,1}, \ldots, x_{0,n}) \in (0,1)^n$ is a non-singular real solution to the equation $F(\mathbf{x}) = 0$.
Let $\omega \in \mathcal{S}^+(\delta; n; \mathfrak{c})$.
Then provided $\delta$ is sufficiently small, we have
$$
\Big{|} \int_{0}^{\infty} \cdots \int_{0}^{\infty}
\prod_{1 \leq j \leq n} \omega(x_{j} - x_{0, j}) \cdot x_1^{r_1 + i t_1} \cdots x_n^{r_n + i t_n} \tsp  e( \tau F(\mathbf{x}))  \tsp  d \mathbf{x}  \Big{|} \ll \min \{ 1,
|\tau|^{-1} \},
$$
where the implicit constant is independent of $r_1, \ldots, r_n$,
$t_1, \ldots, t_n$ and $\tau$.
\end{prop}
The key feature of the result is that the bound is uniform in $\mathbf{t}$; the estimate can be deduced easily for a fixed $\mathbf{t} \in \mathbb{R}^n$.

\section{Major arcs}
\label{major}
Let $F \in \mathbb{Z}[x_1, \ldots, x_n]$ be a homogeneous form of degree $d \geq 2$. 
Let $\mathbf{x}_0$, $\delta$, $\omega$ and $\psi_1, \ldots, \psi_n$ be as in the statement of Theorem \ref{mainthm}.
Let
\begin{eqnarray}
\label{defnoflamdastar}
\Lambda^*(x) =
\left\{
    \begin{array}{ll}
         \log x
         &\mbox{if } x \in \wp, \\
         0
         &\mbox{otherwise. }
    \end{array}
\right.
\end{eqnarray}
We define
\begin{eqnarray}
S^*(\alpha) = \sum_{\mathbf{x} \in \mathbb{N}^{n} }  \varpi(\x) \prod_{1 \leq i \leq n} \Lambda^*(x_i) \cdot  e ( \alpha  F(\mathbf{x})  ),
\end{eqnarray}
where
\begin{equation}
\label{defn varp}
\varpi(\x) = \prod_{1 \leq i \leq n} \psi_i(x_i) = \prod_{1 \leq i \leq n} \omega \left( \frac{x_i}{X} - x_{0,i}  \right).
\end{equation}
Clearly we have
\begin{eqnarray}
\label{S and S1}
S(\alpha) = S^*(\alpha) + O(X^{n - \frac12}).
\end{eqnarray}

We let $\kappa > 0$ be a real number such that
\begin{equation}
\label{defn kap}
|\mathcal{S} (\chi_1, \ldots, \chi_n)| \ll q^{n - \kappa - 1 - \varepsilon},
\end{equation}
where $\mathcal{S}(\chi_1, \ldots, \chi_n)$ is defined in (\ref{calS}); therefore, by Lemma \ref{exp sum modp} we may choose
\begin{equation}
\notag
\kappa = \frac{\textnormal{codim} \tsp V_F^*}{2(2d-1)4^d} - 1 - \kappa_0,
\end{equation}
where $\kappa_0 > 0$ is sufficiently small.
In particular,  under the hypotheses of Theorem \ref{mainthm}, this choice of $\kappa$ satisfies  $\kappa > 2$.
Let $\vartheta_0, \gamma, \lambda > 0$ be such that
\begin{eqnarray}
\label{theta0 major}
2 \vartheta_0 + \gamma < \frac{5}{12}
\quad
\textnormal{ and }
\quad
\gamma > 2 \vartheta_0 + 2 \lambda.
\end{eqnarray}
We set
\begin{eqnarray}
\label{defnT}
M = X^{\vartheta_0} \quad  \textnormal{ and } \quad T = X^{\gamma},
\end{eqnarray}
and let $\widetilde{\chi}$, $\widetilde{r}$ and $\widetilde{\beta}$ be as in Theorem \ref{thm zero free}.
We define the following extended major arcs
\begin{equation}
\label{defn major}
\mathfrak{M}^+(\vartheta_0; \lambda)  = \bigcup_{1 \leq q \leq X^{\vartheta_0}} \bigcup_{\substack{ 0 \leq a \leq q  \\ \gcd(a, q) = 1}} \mathfrak{M}^+_{q,a}(\vartheta_0),
\end{equation}
where
$$
\mathfrak{M}^+_{q, a} (\vartheta_0) = \left\{ {\alpha} \in [0,1) :  \Big{|}  \alpha - \frac{a}{q}  \Big{|} \leq  X^{\vartheta_0 + \lambda  - d} \right\}.
$$
It can be verified that the arcs $\mathfrak{M}^+_{q, a}(\vartheta_0)$ are disjoint for $X$ sufficiently large (see for example \cite[Lemma 4.1]{Bir} for the argument).
In this section, we prove the following which we will make use of with $\vartheta_0 = 1/16 - \varepsilon$ and $\lambda = 1/12$.
\begin{prop}
\label{prop major}
Let $F \in \mathbb{Z}[x_1, \ldots, x_n]$ be a homogeneous form of degree $d \geq 2$ satisfying
$$
\textnormal{codim} \tsp V_F^* >  2 \max \left(  (d-1) 2^{d-1}  + 2  (2 d - 1)    ,   \frac{ (d-1) d 2^{d+1} }{  2d-1   }  \right)
$$
and the local conditions \textnormal{($\star$)}. Suppose $\kappa > 2$, where $\kappa$ is as in (\ref{defn kap}). Let $\vartheta_0, \gamma, \lambda > 0$ be such that (\ref{theta0 major}) is satisfied.
Then for any $A > 0$ we have
$$
\int_{\mathfrak{M}^+(\vartheta_0; \lambda )} S(\alpha) \tsp  d \alpha = c(F; \omega, \mathbf{x}_0) \, X^{n-d} + O \left(  \frac{X^{n-d}}{(\log X)^A} \right),
$$
where $c(F; \omega, \mathbf{x}_0) > 0$ is a constant depending only on $F$, $\omega$ and $\x_0$.
\end{prop}

Let $m, \ell \in \mathbb{Z}_{\geq 0}$ be such that $m + \ell \leq n$. We denote\footnote{Though parentheses are used to denote $\mathbf{i}, \mathbf{j}$ and  $\mathbf{k}$,
these are meant to be sets, not ordered tuples.} $\mathbf{j} = (j_1, \ldots, j_m)$, $\mathbf{k} = (k_1, \ldots, k_{\ell})$ and $\mathbf{i} = (i_1, \ldots, i_{n - m - \ell})$
satisfying
\begin{eqnarray}
\label{cond on j and k}
\{ 1, \ldots, n \} = \{ i_1, \ldots, i_{n - m - \ell} \} \cup  \{ j_1, \ldots, j_m \} \cup \{ k_1, \ldots, k_{\ell} \}.
\end{eqnarray}
For each such triple $(\mathbf{i}, \mathbf{j}, \mathbf{k})$ and Dirichlet characters $\chi_1, \ldots, \chi_m$ modulo $q$, we define
\begin{eqnarray}
&&\mathcal{A}(q, a; \mathbf{i}; (j_1, \chi_1), \ldots,  (j_m, \chi_m); \mathbf{k} )
\\
&=& \sum_{\substack{ \mathbf{h} \in \mathbb{U}_q^n } } \
\overline{\chi_1}(h_{j_1}) \cdots \overline{\chi_m}(h_{j_m}) \tsp
\overline{\widetilde{\chi} \chi^0 }(h_{k_1}) \cdots \overline{\widetilde{\chi} \chi^0} (h_{k_{\ell}})
\ e \left( \frac{a}{q}  F(\mathbf{h} )   \right);
\notag
\end{eqnarray}
if $\widetilde{r} \nmid q$, then we only consider $(\mathbf{i}, \mathbf{j}, \mathbf{k})$  with $\mathbf{k} = \emptyset$, i.e. $\ell = 0$.
We also define
\begin{eqnarray}
\label{DefnW}
&&\mathcal{W}(\tau; \mathbf{i};  (j_1, \chi_1), \ldots,  (j_m, \chi_m); \mathbf{k}  )
\\
&=&
\int_0^{\infty} \cdots \int_0^{\infty}
\sum_{x_{j_1}, \ldots, x_{j_m} \in \mathbb{N}} \varpi(\mathbf{x})
\prod_{1 \leq v \leq m} \chi_v (x_{j_v}) \Lambda^*(x_{j_v})
\cdot
\notag
\\
&& x_{k_1}^{\widetilde{\beta} - 1} \cdots x_{k_{\ell}}^{\widetilde{\beta} - 1}
\tsp e ( \tau F(\mathbf{x})  ) \tsp  d x_{i_1} \cdots d x_{i_{n - m - \ell}} d x_{k_1} \cdots d x_{k_{\ell}},
\notag
\end{eqnarray}
where for each $1 \leq v \leq m$ we replace $\chi_v(x_{j_v}) \Lambda^*(x_{j_v})$ with $\left( \chi^0(x_{j_v}) \Lambda^*(x_{j_v}) - 1 \right)$ when $\chi_{v} = \chi^0$, and
with $\left( \widetilde{\chi} \chi^0(x_{j_v}) \Lambda^*(x_{j_v}) + x_{j_v}^{\widetilde{\beta} - 1} \right)$ when $\chi_v = \widetilde{\chi} \chi^0$.
We extend this definition to allow each $\chi_v$ to be a primitive character modulo $r_v$ as follows
\begin{eqnarray}
\label{DefnW2}
&&\mathcal{W}(\tau; \mathbf{i};  (j_1, \chi_1), \ldots,  (j_m, \chi_m); \mathbf{k}  )
\\
&=&
\int_0^{\infty} \cdots \int_0^{\infty}
\sum_{x_{j_1}, \ldots, x_{j_m} \in \mathbb{N}} \varpi(\mathbf{x})
\prod_{1 \leq v \leq m} \chi_v (x_{j_v}) \Lambda^*(x_{j_v})
\cdot
\notag
\\
&& x_{k_1}^{\widetilde{\beta} - 1} \cdots x_{k_{\ell}}^{\widetilde{\beta} - 1}
\tsp e ( \tau F(\mathbf{x})  ) \tsp  d x_{i_1} \cdots d x_{i_{n - m - \ell}} d x_{k_1} \cdots d x_{k_{\ell}},
\notag
\end{eqnarray}
where for each $1 \leq v \leq m$ we replace $\chi_v(x_{j_v}) \Lambda^*(x_{j_v})$ with $\left( \Lambda^*(x_{j_v}) - 1 \right)$ when $r_{v} = 1$, and
with $\left( \widetilde{\chi} (x_{j_v})  \Lambda^*(x_{j_v}) + x_{j_v}^{\widetilde{\beta} - 1} \right)$ when $\chi_v = \widetilde{\chi}$.
With these notation we have the following technical estimate; the result is essentially \cite[Lemma 8.1]{Y2}
except that the range of $\tau$ is $|\tau| \leq X^{\vartheta_0 + \lambda - d}$ instead of $|\tau| <  X^{\vartheta_0 - d}$.
\begin{lem}
\label{first sum lemma}
Let $\alpha \in [0,1)$, where $\alpha = \frac{a}{q} + {\tau}$, $0 \leq a \leq q \leq X^{\vartheta_0}$, $\gcd(a,q) = 1$
and $|\tau| \leq X^{\vartheta_0 + \lambda - d}$.
Then we have
\begin{eqnarray}
\label{first sum}
S^*(\alpha) &=& \frac{1}{\phi(q)^n} \tsp  \mathcal{A} (q, a; (1, \ldots, n) ; \emptyset ;  \emptyset ) \tsp \mathcal{W}(\tau;(1, \ldots, n) ;  \emptyset ; \emptyset )
\\
\notag
&+& \sum_{(\mathbf{j}, \mathbf{k}) } \frac{(-1)^{\ell}}{\phi(q)^n} \tsp S_{\mathbf{j}, \mathbf{k}}(\alpha)
+ O (  X^{n + \vartheta_0 + \lambda -1} ),
\end{eqnarray}
where
\begin{eqnarray}
S_{\mathbf{j}, \mathbf{k}}(\alpha) =  \sum_{\chi_1, \ldots, \chi_m (\textnormal{mod} \tsp q)}
\mathcal{A}(q, a; \mathbf{i} ;  (j_1, \chi_1), \ldots,  (j_m, \chi_m); \mathbf{k} ) \tsp  \mathcal{W}(\tau; \mathbf{i};  (j_1, \chi_1), \ldots,  (j_m, \chi_m); \mathbf{k} ),
\notag
\end{eqnarray}
and the sum  $\sum_{(\mathbf{j}, \mathbf{k}) }$ in (\ref{first sum}) is over all $(\mathbf{j}, \mathbf{k})$ satisfying $(m, \ell) \not = (0,0)$
and (\ref{cond on j and k}) with an additional condition $\ell = 0$ if $\widetilde{r} \nmid q$ or
the exceptional zero does not exist.
\end{lem}

\begin{proof}
Here we only consider the case where the exceptional zero does exist and $\widetilde{r}|q$; the proof for the cases
$\widetilde{r} \nmid q$ or the exceptional zero does not exist
are identical to this case with only slight modifications. First note if $x_u \in \wp \cap [(x_{0,u} - \delta)X, (x_{0,u} + \delta) X]$, then $\gcd(x_u, q) =1$ because $q
\leq X^{\vartheta_0}$. Therefore, we obtain the following via the orthogonality relation of
the Dirichlet characters
\begin{eqnarray}
\label{S1}
S^*(\alpha) &=& \frac{1}{\phi(q)^n} \sum_{\chi_1, \ldots, \chi_n (\textnormal{mod} \tsp q)} \
\sum_{\mathbf{h} \in \mathbb{U}_q^n} \overline{\chi_1}(h_1) \cdots \overline{\chi_n}(h_n) \  e \left( \frac{a}{q} F(\mathbf{h}) \right) \cdot
\\
&&
\notag
\sum_{\mathbf{x} \in \mathbb{N}^n} \varpi(\mathbf{x}) \prod_{1 \leq u \leq n} \chi_u(x_u) \Lambda^*(x_u) \cdot  e ( \tau F(\mathbf{x}) ).
\end{eqnarray}

Let $\mathbf{i}' = (i'_1, \ldots, i'_{s'})$, $\mathbf{j}' = (j'_1, \ldots, j'_{m'})$, $\mathbf{k}' = (k'_1, \ldots, k'_{\ell'})$,
$\mathbf{k} = (k_1, \ldots, k_{\ell})$ and $\mathbf{i} = (i_1, \ldots, i_{n - s' - m' - \ell' - \ell})$, where
\begin{eqnarray}
\label{paritionnn}
 \{1, \ldots, n\} &=& \{i_1, \ldots, i_{n - s' - m' - \ell' - \ell} \} \cup \{ i'_1, \ldots, i'_{s'} \} \cup  \{ j'_1, \ldots, j'_{m'} \}
\\
\notag
&\cup&  \{ k'_1, \ldots, k'_{\ell'} \} \cup  \{ k_1, \ldots, k_{\ell} \}.
\end{eqnarray}
We remark that $s', m', \ell', \ell$ and $n - s' - m' - \ell' - \ell$ are allowed to be $0$.
Now we break up the summands of the inner sum $\sum_{\mathbf{x} \in \mathbb{N}^n}$ 
using the identities
$$
\chi^0(x_{u}) \Lambda^*(x_{u}) = \left( \chi^0(x_{u}) \Lambda^*(x_{u}) - 1 \right) + 1
$$
when $\chi_u = \chi^0$, and
$$
\widetilde{\chi} \chi^0(x_{u}) \Lambda^*(x_{u})  = \left( \widetilde{\chi} \chi^0(x_{u}) \Lambda^*(x_{u}) + x_u^{\widetilde{\beta} - 1} \right) - x_u^{\widetilde{\beta} - 1}
$$
when $\chi_u = \widetilde{\chi} \chi^0$ (If $ \widetilde{r} \nmid q$ or the exceptional zero does not exist, then we simply ignore this second identity.); we do this for each
$1 \leq u \leq n$ and obtain
\begin{eqnarray}
\label{S1 small}
&&\sum_{\mathbf{x} \in \mathbb{N}^n} \varpi(\mathbf{x}) \prod_{1 \leq u \leq n} \chi_u(x_u) \Lambda^*(x_u) \cdot  e ( \tau F(\mathbf{x}) )
\\
\notag
&=& \sum_{ (\mathbf{i}, \mathbf{i'}, \mathbf{j'}, \mathbf{k'}, \mathbf{k}) }  \sum_{\mathbf{x} \in \mathbb{N}^n} \varpi(\mathbf{x})
\prod_{w \in \{ i'_1, \ldots, i'_{s'} \}  } \chi^0(x_w) \Lambda^*(x_w)
\cdot
\\
&&
\notag
\prod_{v \in \{ j'_1, \ldots, j'_{m'} \}  }  \chi_v(x_v) \Lambda^*(x_v) \cdot
\prod_{u \in \{ k'_1, \ldots, k'_{\ell'} \}   } \widetilde{\chi} \chi^0(x_u) \Lambda^*(x_u) \cdot
\\
\notag
&&
(-1)^{\ell} x_{k_1}^{\widetilde{\beta} - 1} \cdots x_{k_{\ell}}^{\widetilde{\beta} - 1}
\tsp   e ( \tau F(\mathbf{x}) ),
\notag
\end{eqnarray}
where the sum $\sum_{ (\mathbf{i}, \mathbf{i'}, \mathbf{j'}, \mathbf{k'}, \mathbf{k}) }$ is over all $(\mathbf{i}, \mathbf{i'}, \mathbf{j'}, \mathbf{k'}, \mathbf{k})$
satisfying
$$
\{ w:  \chi_w = \chi^0  \} =  \{ i_1, \ldots, i_{n - s' - m' - \ell' - \ell} \} \cup   \{ i'_1, \ldots, i'_{s'} \},
$$
$$
\{ v:  \chi_v \neq \chi^0, \widetilde{\chi} \chi^0   \} =   \{ j'_1, \ldots, j'_{m'} \}
$$
and
$$
\{ u:  \chi_u =  \widetilde{\chi} \chi^0  \} =  \{ k'_1, \ldots, k'_{\ell'} \} \cup  \{ k_1, \ldots, k_{\ell} \},
$$
and the convention regarding $\chi^0(x_w) \Lambda^*(x_w)$ and  $\widetilde{\chi} \chi^0(x_u) \Lambda^*(x_u)$
described in the sentence following (\ref{DefnW}) is being used. By substituting (\ref{S1 small}) into (\ref{S1}) and
changing the order of summation, we obtain
\begin{eqnarray}
\notag
S^*(\alpha) &=& \sum_{ (\mathbf{i}, \mathbf{i'}, \mathbf{j'}, \mathbf{k'}, \mathbf{k}) }
\frac{1}{\phi(q)^n} \sum_{\chi_{j'_1}, \ldots, \chi_{j'_{m'}  }  \neq \chi^0, \widetilde{\chi} \chi^0 } \
\sum_{\mathbf{h} \in \mathbb{U}_q^n} \overline{\chi_{j'_1} }(h_{j'_1}) \cdots \overline{\chi_{j'_{m'}} }(h_{j'_{m'}}) \cdot
\\
\notag
&& \prod_{u \in \{ k'_1, \ldots, k'_{\ell'} \} \cup  \{ k_1, \ldots, k_{\ell} \}}  \overline{  \widetilde{\chi} \chi^0  }(h_{u})
\cdot  e \left( \frac{a}{q} F(\mathbf{h}) \right) \cdot \sum_{\mathbf{x} \in \mathbb{N}^n} \varpi(\mathbf{x})
\prod_{w \in \{ i'_1, \ldots, i'_{s'} \}  } \chi^0(x_w) \Lambda^*(x_w)
\cdot
\\
&&
\notag
\prod_{v \in \{ j'_1, \ldots, j'_{m'} \}  }  \chi_v(x_v) \Lambda^*(x_v) \cdot
\prod_{u \in \{ k'_1, \ldots, k'_{\ell'} \}   } \widetilde{\chi} \chi^0(x_u) \Lambda^*(x_u) \cdot
\\
\notag
&&
(-1)^{\ell} x_{k_1}^{\widetilde{\beta} - 1} \cdots x_{k_{\ell}}^{\widetilde{\beta} - 1}
\tsp   e ( \tau F(\mathbf{x}) ),
\notag
\end{eqnarray}
where the sum $\sum_{ (\mathbf{i}, \mathbf{i'}, \mathbf{j'}, \mathbf{k'}, \mathbf{k}) }$ is over all $(\mathbf{i}, \mathbf{i'}, \mathbf{j'}, \mathbf{k'}, \mathbf{k})$
satisfying (\ref{paritionnn}). Let us set $\mathbf{j} = (\mathbf{i}', \mathbf{j}', \mathbf{k}')$.
Then for fixed $\mathbf{h} \in \mathbb{U}_q^n$, $\x \in \NN^n$ and $(\mathbf{i}, \mathbf{k})$, it follows that
\begin{eqnarray}
&&\sum_{ (\mathbf{i'}, \mathbf{j'}, \mathbf{k'}) } \tsp
\sum_{\chi_{j'_1}, \ldots, \chi_{j'_{m'}  }  \neq \chi^0, \widetilde{\chi} \chi^0 } \
\overline{\chi_{j'_1} }(h_{j'_1}) \cdots \overline{\chi_{j'_{m'}} }(h_{j'_{m'}}) \prod_{u \in \{ k'_1, \ldots, k'_{\ell'} \} }  \overline{  \widetilde{\chi} \chi^0  }(h_{u}) \cdot
\\
\notag
&&
\prod_{w \in \{ i'_1, \ldots, i'_{s'} \}  } \chi^0(x_w) \Lambda^*(x_w)
\cdot
\notag
\prod_{v \in \{ j'_1, \ldots, j'_{m'} \}  }  \chi_v(x_v) \Lambda^*(x_v) \cdot
\prod_{u \in \{ k'_1, \ldots, k'_{\ell'} \}   } \widetilde{\chi} \chi^0(x_u) \Lambda^*(x_u)
\notag
\\
\notag
&=&\sum_{ (\mathbf{i'}, \mathbf{j'}, \mathbf{k'}) } \tsp
\prod_{w \in \{ i'_1, \ldots, i'_{s'} \}  }
\overline{  \chi^0  }(h_{w})
\chi^0(x_w) \Lambda^*(x_w) \cdot
\\
\notag
&& \sum_{\chi_{j'_1}, \ldots, \chi_{j'_{m'}  }  \neq \chi^0, \widetilde{\chi} \chi^0 } \
\prod_{v \in \{ j'_1, \ldots, j'_{m'} \}  } \overline{\chi_{v}} (h_{v}) \chi_v(x_v) \Lambda^*(x_v)  \cdot
\prod_{u \in \{ k'_1, \ldots, k'_{\ell'} \} }  \overline{  \widetilde{\chi} \chi^0  }(h_{u}) \widetilde{\chi} \chi^0(x_u) \Lambda^*(x_u)
\\
&=&
\notag
\sum_{\chi_{j_1}, \ldots, \chi_{j_{m}}  (\textnormal{mod} \tsp q)   }
\prod_{v \in \{ j_1, \ldots, j_{m} \}  }   \overline{\chi_{v} }(h_{v})  \chi_v(x_v) \Lambda^*(x_v),
\end{eqnarray}
where 
$m = s' + m' + \ell'$.
Therefore, we obtain
\begin{eqnarray}
\label{sum'''}
&&S^*(\alpha)
\\
\notag
&=&
\sum_{ (\mathbf{i}, \mathbf{j}, \mathbf{k}) }
\frac{1}{\phi(q)^n} \sum_{\chi_{j_1}, \ldots, \chi_{j_{m}}  (\textnormal{mod} \tsp q)   } \
\sum_{\mathbf{h} \in \mathbb{U}_q^n} \overline{\chi_{j_1} }(h_{j_1}) \cdots \overline{\chi_{j_{m}} }(h_{j_{m}})
\prod_{u \in \{ k_1, \ldots, k_{\ell} \}}  \overline{  \widetilde{\chi} \chi^0  }(h_{u})
\cdot  e \left( \frac{a}{q} F(\mathbf{h}) \right) \cdot
\\
\notag
&&
\sum_{\mathbf{x} \in \mathbb{N}^n} \varpi(\mathbf{x})
\prod_{v \in \{ j_1, \ldots, j_{m} \}  }  \chi_v(x_v) \Lambda^*(x_v)  \cdot
(-1)^{\ell} x_{k_1}^{\widetilde{\beta} - 1} \cdots x_{k_{\ell}}^{\widetilde{\beta} - 1}
\tsp   e ( \tau F(\mathbf{x}) )
\\
\notag
&=& \frac{1}{\phi(q)^n} \tsp  \mathcal{A} (q, a; (1, \ldots, n) ; \emptyset ;  \emptyset )
\sum_{\mathbf{x} \in \mathbb{N}^n} \varpi(\mathbf{x})   e ( \tau F(\mathbf{x})  )
\\
\notag
&+& \sum_{(\mathbf{j}, \mathbf{k}) } \frac{(-1)^{\ell}}{\phi(q)^n}
\sum_{\chi_1, \ldots, \chi_m (\textnormal{mod} \tsp q)}
\mathcal{A}(q, a; \mathbf{i} ;  (j_1, \chi_1), \ldots,  (j_m, \chi_m); \mathbf{k} ) \cdot
\\
&& \sum_{\mathbf{x} \in \mathbb{N}^n} \varpi(\mathbf{x})
\prod_{1 \leq v \leq m} \chi_v(x_{j_v}) \Lambda^*(x_{j_v})
\cdot   x_{k_1}^{\widetilde{\beta} - 1} \cdots x_{k_{\ell}}^{\widetilde{\beta} - 1}
\tsp e ( \tau F(\mathbf{x})  ),
\notag
\end{eqnarray}
where the sum $\sum_{(\mathbf{j}, \mathbf{k}) }$ is over all $(\mathbf{j}, \mathbf{k})$ satisfying $(m, \ell) \not = (0,0)$ and (\ref{cond on j and k}).
We also switched the notation from $\chi_{j_m}$ to $\chi_m$ to get the second equality.
Each summand of the sum $\sum_{(\mathbf{j}, \mathbf{k}) }$ in (\ref{sum'''}) can further be expressed as
\begin{eqnarray}
\notag
&& \frac{ (-1)^{\ell} }{\phi(q)^n}
\sum_{\mathbf{h} \in \mathbb{U}_q^n} \overline{\widetilde{\chi} \chi^0 }(h_{k_1}) \cdots
\overline{\widetilde{\chi} \chi^0}(h_{k_{\ell}}) \
e \left( \frac{a}{q} F(\mathbf{h}) \right) \cdot
\\
&&
\notag
\sum_{ \substack{ x_{j_1}, \ldots, x_{j_m} \in \mathbb{N} } } \
\prod_{  1 \leq v \leq m } \Lambda^*(x_{j_v}) \sum_{\chi_v (\textnormal{mod} \tsp q)}   \overline{\chi_v}(h_{j_v})  \chi_v(x_{j_v})
\cdot
\sum_{ \substack{ x_{u} \in \mathbb{N} \\ (u \not \in\{j_1, \ldots, j_m\}) } }
\varpi(\mathbf{x}) \tsp x_{k_1}^{\widetilde{\beta} - 1} \cdots x_{k_{\ell}}^{\widetilde{\beta} - 1} e ( \tau F(\mathbf{x}) )
\\
\notag
&=& \frac{ (-1)^{\ell} }{\phi(q)^n}
\sum_{\mathbf{h} \in \mathbb{U}_q^n} \overline{\widetilde{\chi} \chi^0 }(h_{k_1}) \cdots
\overline{\widetilde{\chi} \chi^0}(h_{k_{\ell}}) \
e \left( \frac{a}{q} F(\mathbf{h}) \right) \cdot
\\
&&
\notag
\phi(q)^m \sum_{ \substack{ x_{j_v} \in \mathbb{N}  \\  x_{j_v} \equiv h_{j_v}  (\textnormal{mod} \tsp q) \\  (1 \leq v \leq m)  } } \
\prod_{  1 \leq v \leq m } \Lambda^*(x_{j_v})
\cdot
\sum_{ \substack{ x_{u} \in \mathbb{N} \\ (u \not \in\{j_1, \ldots, j_m\}) } }
\varpi(\mathbf{x}) \tsp x_{k_1}^{\widetilde{\beta} - 1} \cdots x_{k_{\ell}}^{\widetilde{\beta} - 1} e ( \tau F(\mathbf{x}) ),
\end{eqnarray}
where we used the orthogonality relation of the Dirichlet characters to obtain the latter expression.
Next we apply the following estimate which can be deduced from the mean value theorem along with (\ref{excepZZ}).
\begin{lem}
\label{Lemma 2.22}
For each $x_{j_1}, \ldots, x_{j_m} \in \mathbb{N}$,  we have
\begin{eqnarray}
\notag
&&  \sum_{ \substack{ x_{u} \in \mathbb{N} \\ (u \not \in\{j_1, \ldots, j_m\}) } }
\varpi(\mathbf{x}) \tsp x_{k_1}^{\widetilde{\beta} - 1} \cdots x_{k_{\ell}}^{\widetilde{\beta} - 1} e ( \tau F(\mathbf{x})  )
\\
&=&
\notag
\int_0^{\infty} \cdots \int_0^{\infty}
\varpi(\mathbf{x}) \tsp x_{k_1}^{\widetilde{\beta} - 1} \cdots x_{k_{\ell}}^{\widetilde{\beta} - 1}
 e ( \tau F(\mathbf{x})  ) \tsp  d x_{i_1} \cdots d x_{i_{n - m - \ell}} d x_{k_1} \cdots d x_{k_{\ell}}
+ O(X^{n - m + \vartheta_0 + \lambda - 1}),
\end{eqnarray}
where the implicit constant is independent of $x_{j_1}, \ldots, x_{j_m}$.
\end{lem}
\begin{proof}
We deal with one variable at a time. By the mean value theorem, it follows that
\begin{eqnarray}
\notag
&& \sum_{  x_{k_1} \in \mathbb{N}   }
\varpi(\mathbf{x}) \tsp x_{k_1}^{\widetilde{\beta} - 1} e ( \tau F(\mathbf{x})  )
\\
&=&
\notag
\sum_{ t \in \mathbb{N}  }
\left(
\int_{ t }^{ t + 1}
\varpi(\mathbf{x}) \tsp x_{k_1}^{\widetilde{\beta} - 1}  e ( \tau F(\mathbf{x})  ) \tsp d x_{k_1}
+ O \left(  \sup_{ x_{k_1} \in [t, t+1]  }  \left|  \frac{\partial  }{  \partial x_{k_1} } \left(  \varpi(\mathbf{x}) \tsp x_{k_1}^{\widetilde{\beta} - 1} e ( \tau F(\mathbf{x})  ) \right)
\right|  \right) \right)
\\
&=&
\notag
\int_{ 0 }^{ \infty }
\varpi(\mathbf{x}) \tsp x_{k_1}^{\widetilde{\beta} - 1}  e ( \tau F(\mathbf{x})  ) \tsp d x_{k_1}
+ O \left( X  \sup_{ x_{k_1} \in [1, X]  }  \left|  \frac{\partial  }{  \partial x_{k_1} }
\left( \varpi(\mathbf{x}) \tsp x_{k_1}^{\widetilde{\beta} - 1} e ( \tau F(\mathbf{x})  ) \right)
\right|  \right).
\end{eqnarray}
By (\ref{excepZZ}) and the assumption $|\tau| \leq X^{\vartheta_0 + \lambda - d}$, we have
\begin{eqnarray}
\notag
\sup_{x_{k_1} \in [1, X]} \left|  \frac{\partial  }{  \partial x_{k_1} }  \left(  \varpi(\mathbf{x}) \tsp x_{k_1}^{\widetilde{\beta} - 1} e ( \tau F(\mathbf{x})  )
\right)
\right|
\ll
|\tau X^{d-1}|
<
|X^{\vartheta_0 + \lambda - 1}|.
\end{eqnarray}
Therefore, we obtain
\begin{eqnarray}
\notag
&&  \sum_{ \substack{ x_{u} \in \mathbb{N} \\ (u \not \in\{j_1, \ldots, j_m\}) } }
\varpi(\mathbf{x}) \tsp x_{k_1}^{\widetilde{\beta} - 1} \cdots x_{k_{\ell}}^{\widetilde{\beta} - 1} e ( \tau F(\mathbf{x})  )
\\
&=&
\notag
\int_0^{\infty}  \sum_{ \substack{ x_{u} \in \mathbb{N} \\ (u \not \in\{j_1, \ldots, j_m, k_1\}) } }
\varpi(\mathbf{x}) \tsp x_{k_1}^{\widetilde{\beta} - 1} \cdots x_{k_{\ell}}^{\widetilde{\beta} - 1}
 e ( \tau F(\mathbf{x})  ) \tsp  d x_{k_1}
+ O(X^{n - m + \vartheta_0 + \lambda - 1}).
\end{eqnarray}
The result follows by repeating this procedure for each of the remaining variables.
\end{proof}

We now apply Lemma \ref{Lemma 2.22} to both terms on the right hand side of (\ref{sum'''}); the first term becomes
\begin{eqnarray}
\notag
&& \frac{1}{\phi(q)^n} \tsp  \mathcal{A} (q, a; (1, \ldots, n) ; \emptyset ;  \emptyset )
\sum_{\mathbf{x} \in \mathbb{N}^n} \varpi(\mathbf{x})  e ( \tau F(\mathbf{x})  )
\\
\notag
&=&  \frac{1}{\phi(q)^n} \tsp  \mathcal{A} (q, a; (1, \ldots, n) ; \emptyset ;  \emptyset )
\tsp  \mathcal{W}(\tau; (1, \ldots, n) ; \emptyset ; \emptyset )
+ O(X^{n + \vartheta_0 + \lambda - 1}),
\end{eqnarray}
and each summand of the sum $\sum_{(\mathbf{j}, \mathbf{k}) }$ (see the expression above Lemma \ref{Lemma 2.22}) becomes
\begin{eqnarray}
\notag
&& \frac{(-1)^{\ell}}{\phi(q)^n}
\sum_{\chi_1, \ldots, \chi_m (\textnormal{mod} \tsp q)}
\mathcal{A}(q, a; \mathbf{i} ;  (j_1, \chi_1), \ldots,  (j_m, \chi_m); \mathbf{k} ) \cdot
\\
&& \sum_{\mathbf{x} \in \mathbb{N}^n} \varpi(\mathbf{x}) \tsp  \chi_1(x_{j_1}) \Lambda^*(x_{j_1}) \cdots \chi_m(x_{j_m}) \Lambda^*(x_{j_m})
\tsp  x_{k_1}^{\widetilde{\beta} - 1} \cdots x_{k_{\ell}}^{\widetilde{\beta} - 1}
\tsp e ( \tau F(\mathbf{x})  )
\notag
\\
&=&
\frac{(-1)^{\ell}}{\phi(q)^n}
\sum_{\chi_1, \ldots, \chi_m (\textnormal{mod} \tsp q)}
\mathcal{A}(q, a; \mathbf{i} ;  (j_1, \chi_1), \ldots,  (j_m, \chi_m); \mathbf{k} )
\tsp \mathcal{W}(\tau; \mathbf{i};  (j_1, \chi_1), \ldots,  (j_m, \chi_m); \mathbf{k} )
\notag
\\
&+&
\notag
O \left(
\frac{  X^{n - m + \vartheta_0 + \lambda - 1} }{\phi(q)^{n-m}}
\sum_{\mathbf{h} \in \mathbb{U}_q^n}
\prod_{  1 \leq v \leq m }
\sum_{ \substack{ 1 \leq x_{j_v} \leq X  \\  x_{j_v} \equiv h_{j_v}  (\textnormal{mod} \tsp q) } }
\Lambda^*(x_{j_v})
\right).
\end{eqnarray}
Therefore, we see that we have obtained the result apart form the error term.
For each $(\mathbf{j}, \mathbf{k})$, we have by the prime number theorem that
\begin{eqnarray}
\notag
&& \frac{  X^{n - m + \vartheta_0 + \lambda - 1} }{\phi(q)^{n-m}}
\sum_{\mathbf{h} \in \mathbb{U}_q^n}
\prod_{  1 \leq v \leq m }
\sum_{ \substack{ 1 \leq x_{j_v} \leq X  \\  x_{j_v} \equiv h_{j_v}  (\textnormal{mod} \tsp q) } }
\Lambda^*(x_{j_v})
\\
&=&
\notag
X^{n - m + \vartheta_0 + \lambda - 1}
\prod_{1 \leq v \leq m} \,
\sum_{h_{j_v} \in \mathbb{U}_q} \, \sum_{ \substack{ 1 \leq x_{j_v} \leq X \\  x_{j_v} \equiv h_{j_v} (\textnormal{mod} \tsp q) } }  \Lambda^*(x_{j_v})
\\
&\leq&
\notag
X^{n - m + \vartheta_0 + \lambda - 1}
\left( \sum_{ 1 \leq x \leq X }  \Lambda(x) \right)^m
\\
&\ll&
\notag
X^{n + \vartheta_0 + \lambda  -  1}.
\end{eqnarray}
From this bound it follows that the error term is as in the statement of the lemma.
\end{proof}
It follows from Lemma \ref{first sum lemma} that
\begin{eqnarray}
\label{MAJOR1}
&&\int_{\mathfrak{M}^+(\vartheta_0; \lambda)} S^*(\alpha) \tsp d \alpha
\\
\notag
&=&
\sum_{1 \leq q \leq X^{\vartheta_0}} \sum_{\substack{ 1 \leq a \leq q \\  \gcd (a,q) = 1 }} \frac{1}{\phi(q)^n} \tsp \mathcal{A} (q, a; (1, \ldots, n) ; \emptyset ;  \emptyset )
\cdot
\int_{|\tau| \leq X^{\vartheta_0 + \lambda - d} }
\mathcal{W} (\tau;(1, \ldots, n) ;  \emptyset ; \emptyset ) \tsp d \tau
\\
\notag
&+&
\sum_{(\mathbf{j}, \emptyset) }
\sum_{1 \leq q \leq X^{\vartheta_0}} \sum_{\substack{ 1 \leq a \leq q \\  \gcd (a,q) = 1 }}  \frac{ 1 }{\phi(q)^n} \tsp \int_{|\tau| \leq X^{\vartheta_0 + \lambda - d}  } S_{\mathbf{j}, \emptyset} \left( \frac{a}{q} + \tau \right)  d \tau
\\
\notag
&+&
\sum_{ \substack{ (\mathbf{j}, \mathbf{k})  \\   \mathbf{k} \neq \emptyset  } }
\sum_{  \substack{ 1 \leq q \leq X^{\vartheta_0}  \\  \widetilde{r} | q  } } \sum_{\substack{ 1 \leq a \leq q \\  \gcd (a,q) = 1 }}  \frac{(-1)^{\ell}}{\phi(q)^n} \tsp \int_{|\tau| \leq X^{\vartheta_0 + \lambda - d}  } S_{\mathbf{j}, \mathbf{k}} \left( \frac{a}{q} + \tau \right) d \tau
\\
\notag
&+& O(X^{n - d + 4 \vartheta_0 + 2 \lambda - 1}),
\end{eqnarray}
where the sum $\sum_{(\mathbf{j}, \mathbf{k}) }$ is as in the statement of Lemma \ref{first sum lemma} (If the exceptional zero does not exist, then the sum
$\sum_{ \substack{ (\mathbf{j}, \mathbf{k})  \\   \mathbf{k} \neq \emptyset  } }$ is an empty sum.). We prove that the first term on the right hand side of (\ref{MAJOR1}) contributes the main term, while the remaining terms are error terms. We consider the contribution from each $( \mathbf{j}, \mathbf{k})$ separately, where we deal with the case $\mathbf{k} = \emptyset$ in Section \ref{sec m>0}, and the case $\mathbf{k} \neq \emptyset$ in Section \ref{secm=0}. Finally, Proposition \ref{prop major} is established in Section \ref{secmajormainterm}.

\subsection{Case $\mathbf{k} = \emptyset$}
\label{sec m>0}
Since $(m, \ell) \not = (0,0)$, we necessarily have $m > 0$. Without loss of generality let $\mathbf{j} = (1, \ldots, m)$.
In this case, we have
\begin{eqnarray}
\label{eqn 1 m>0}
&& \Big{|} \sum_{1 \leq q \leq X^{\vartheta_0}} \sum_{ \substack{ 1 \leq a \leq q \\  \gcd(a,q) = 1 } } \frac{1}{\phi(q)^n}  \int_{|\tau| \leq X^{\vartheta_0 + \lambda- d} }
S_{\mathbf{j}, \emptyset } \left( \frac{a}{q} + \tau \right)  d \tau \Big{|}
\\
\notag
&=& \Big{|} \sum_{1 \leq q \leq X^{\vartheta_0}}  \sum_{  \chi'_1, \ldots, \chi'_m  (\textnormal{mod} \tsp q)  }  \sum_{\substack{ 1 \leq a \leq q \\  \gcd(a,q)
= 1 }} \frac{1}{\phi(q)^n} \tsp \mathcal{A}(q, a; \mathbf{i};  (1, \chi'_1), \ldots, (m, \chi'_m) ; \emptyset  ) \cdot  \\
&& \int_{|\tau| \leq X^{\vartheta_0 + \lambda - d} }  \mathcal{W}(\tau; \mathbf{i}; (1, \chi'_1), \ldots, (m, \chi'_m) ; \emptyset  )  \tsp d \tau  \Big{|}.
\notag
\end{eqnarray}
Let us denote $\chi'_v = \chi_v \chi^0$, where $\chi_v$ is the primitive character modulo $r_v$ which induces $\chi'_v$,
and $\chi^0$ is the principal character modulo $q$. We also denote
\begin{eqnarray}
\label{defnR}
R= \textnormal{lcm}(r_1, \ldots, r_m).
\end{eqnarray}
We let 
$$
\delta_{\min} =  \min_{1 \leq u \leq n} (x_{0,u} - \delta).
$$
Let $p \in [(x_{0,u} - \delta) X, (x_{0,u} + \delta) X]$ be a prime.
Since $q \leq X^{\vartheta_0} < \delta_{\min} X$, we have $\gcd(p,q) = 1$ and it follows that
$\chi_v(p) \chi^{0}(p) = \chi_v(p)$.
Consequently, we obtain
\begin{eqnarray}
\notag
\mathcal{W}(\tau; \mathbf{i}; (1, \chi'_1), \ldots, (m, \chi'_m) ; \emptyset  )
&=& \mathcal{W}(\tau; \mathbf{i}; (1, \chi_1 \chi^0), \ldots, (m, \chi_m \chi^0) ; \emptyset )
\\
\notag
&=& \mathcal{W}(\tau; \mathbf{i}; (1, \chi_1), \ldots, (m, \chi_m) ; \emptyset  ).
\end{eqnarray}
Recall the definition of $\kappa$ given in (\ref{defn kap}).
Since $q / \phi(q) \ll_{\varepsilon} q^{\varepsilon}$ for any $\varepsilon > 0$, it follows from Lemma \ref{exp sum modp} that
\begin{eqnarray}
\label{usinglem7.4}
&&\sum_{\substack{ 1 \leq q \leq X^{\vartheta_0} \\  R |q   }}  \sum_{\substack{ 1 \leq a \leq q \\  \gcd(a,q) = 1 }} \frac{1}{\phi(q)^n} | \mathcal{A}(q, a; \mathbf{i}; (1,
\chi_1 \chi^0), \ldots, (m, \chi_m \chi^0) ; \emptyset ) |
\\
\notag
&\ll&
\sum_{\substack{ 1 \leq q \leq X^{\vartheta_0} \\  R |q   }}  \sum_{\substack{ 1 \leq a \leq q \\  \gcd(a,q) = 1 }}
 \frac{ q^{n -  \kappa - 1 -  \varepsilon} }{\phi(q)^n}
\\
\notag
&\ll&
\sum_{\substack{ 1 \leq q \leq X^{\vartheta_0} \\  R |q   }} q^{- \kappa  }
\\
\notag
&\ll& R^{- \kappa}.
\end{eqnarray}
Therefore, the term on the right hand side of (\ref{eqn 1 m>0}) can be rewritten as
\begin{eqnarray}
\label{eqn 2 m>0}
&& \Big{|} \sum_{   1 \leq r_1, \ldots, r_m \leq X^{\vartheta_0} } \  \sideset{}{^*}\sum_{\substack{\chi_v (\textnormal{mod} \tsp r_v)  \\  (1 \leq v \leq m) } }
\sum_{\substack{ 1 \leq q \leq X^{\vartheta_0} \\  R |q   }}   \sum_{\substack{ 1 \leq a \leq q \\  \gcd(a,q) = 1 }} \frac{1}{\phi(q)^n} \tsp
\mathcal{A}(q, a; \mathbf{i}; (1, \chi_1 \chi^0), \ldots, (m, \chi_m \chi^0) ;\emptyset ) \cdot
\\
\notag
&& \int_{|\tau| \leq X^{\vartheta_0 + \lambda - d} } \mathcal{W}(\tau; \mathbf{i}; (1, \chi_1), \ldots, (m, \chi_m) ; \emptyset ) \tsp d \tau
\Big{|}
\\
\notag
&\ll& \sum_{   1 \leq r_1, \ldots, r_m \leq X^{\vartheta_0} } R^{-\kappa}  \sideset{}{^*}\sum_{\substack{\chi_v (\textnormal{mod} \tsp r_v) \\ (1 \leq v \leq m) } }
\int_{|\tau| \leq X^{\vartheta_0 + \lambda- d} }
| \mathcal{W}(\tau; \mathbf{i}; (1, \chi_1), \ldots, (m, \chi_m) ; \emptyset  )  |  \tsp d \tau.
\end{eqnarray}

We will make use of the following explicit formula (for example, it can be deduced from \cite[\S 17 and \S 19]{D}): Let $1 \leq q \leq Y$ and $2 \leq T' \leq Y^{1/2}$. For
any primitive character $\chi$ modulo $q$, we have
\begin{eqnarray}
\label{explicit formula}
\sum_{y \leq Y} \Lambda^*(y) \chi(y)  = \delta_{\chi = \chi^0} Y - \sum_{ \substack{ \rho \in B_{T'}  \\  L(\rho, \chi) = 0  }} \frac{Y^{\rho}}{\rho} + \mathfrak{R}(Y),
\end{eqnarray}
where $\delta_{\chi = \chi^0} = 1$ if $\chi = \chi^0$ and $0$ otherwise, and
\begin{eqnarray}
\label{explicit error}
| \mathfrak{R}(Y) |  \ll E(Y) = \frac{Y (\log Y)^2}{T'}.
\end{eqnarray}
Here the sum $\sum_{ \substack{ \rho \in B_{T'}  \\  L(\rho, \chi) = 0  }}$ is over all the zeros, with multiplicity, of $L(s, \chi)$ in $B_{T'}$.
We will be using (\ref{explicit formula}) with $T' = T$ and $X \ll Y \ll X$.

With these notation (also recall the notation introduced in the paragraph preceding Lemma \ref{Gbdd}) we obtain the following lemma, which is essentially \cite[Lemma 8.2]{Y2}; the result is obtained by replacing $\vartheta_0$ in the proof of \cite[Lemma 8.2]{Y2} with $\vartheta_0 + \lambda$.  Note we do not assume that $\mathbf{k} = \emptyset$ for this lemma.
\begin{lem}
\label{bound on W}
Without loss of generality let $\mathbf{j} = (1, \ldots, m)$ and $\mathbf{k} = (m+1, \ldots, m + \ell)$,
where $m > 0$ and $\ell \geq 0$. Suppose $\chi_v$ is a primitive character modulo $r_v$ $(1 \leq v \leq m)$.
Let $|\tau| \leq X^{\vartheta_0 + \lambda - d}$.
Then we have
\begin{eqnarray}
\label{eqn lem bdd W}
&&\Big{|}  \mathcal{W}(\tau; \mathbf{i};  (1, \chi_1), \ldots,  (m, \chi_m); \mathbf{k} ) \Big{|}
\\
\notag
&\ll& \Big{|} \int_{0}^{\infty} \cdots \int_0^{\infty}
\prod_{1 \leq  v \leq m }  \sideset{}{'}\sum_{\rho_v} x_v^{\rho_v - 1}
\cdot
x_{m+1}^{\widetilde{\beta} - 1} \cdots x_{m + \ell}^{\widetilde{\beta} -1}
\tsp \varpi(\mathbf{x}) \tsp  e (\tau F(\mathbf{x}) ) \tsp d \mathbf{x} \Big{|}
+ \widetilde{E},
\end{eqnarray}
where
\begin{eqnarray}
\notag
\widetilde{E} =  \sum_{ \boldsymbol{\epsilon} \in \{ 0, 1 \}^{m} \backslash \{ \mathbf{0} \} } X^{n} X^{(\vartheta_0 + \lambda - 1) (\epsilon_1 + \cdots + \epsilon_m)}
E(X)^{\epsilon_1 + \cdots + \epsilon_m} \prod_{ \epsilon_j = 0  }  \sideset{}{'}\sum_{\rho_j} (\delta_{\min} X)^{\beta_j - 1}
\end{eqnarray}
and $\textnormal{Re }\rho_j = \beta_j$.
\end{lem}

Recall the definition of $\varpi$ given in (\ref{defn varp}).
We substitute (\ref{eqn lem bdd W}) into (\ref{eqn 2 m>0}), and
apply Proposition \ref{prop osc int} to the first of the resulting terms in the following manner
\begin{eqnarray}
\label{eqn 3 m>0}
&&
\int_{|\tau| \leq X^{\vartheta_0 + \lambda - d} } \Big{|}  \int_{0}^{\infty} \cdots \int_0^{\infty}
\prod_{1 \leq  v \leq m}  \sideset{}{'}\sum_{\rho_v} x_v^{\rho_v - 1}
\cdot \varpi(\mathbf{x}) \tsp  e (\tau F(\mathbf{x}) )   \tsp d \mathbf{x}  \Big{|}  \tsp  d \tau
\\
&=&
X^{n-d} \int_{|\tau| \leq X^{\vartheta_0 + \lambda} } \Big{|}  \int_{0}^{\infty} \cdots \int_0^{\infty}
\prod_{ 1 \leq  v  \leq  m}  \sideset{}{'}\sum_{\rho_v} (X x_v)^{\rho_v - 1}
\cdot  \prod_{ 1 \leq  u \leq n} \omega(x_{u} - x_{0, u}) \cdot   e (\tau F(\mathbf{x}) )   \tsp d \mathbf{x}  \Big{|}  \tsp  d \tau
\notag
\\
\notag
&\ll&
X^{n - d } \sideset{}{'}\sum_{\rho_1, \ldots, \rho_m} | X^{\rho_1 - 1} \cdots X^{\rho_m - 1} |   \cdot
\\
&&
\notag
\int_{|\tau| \leq X^{\vartheta_0 + \lambda }} \Big{|} \int_{0}^{\infty} \cdots \int_0^{\infty}
x_1^{\rho_1 - 1} \cdots x_m^{\rho_m - 1} \prod_{ 1 \leq  u \leq n } \omega(x_{u} - x_{0, u}) \cdot  e (\tau F(\mathbf{x}) )   \tsp d \mathbf{x}   \Big{|} \tsp  d \tau
\\
\notag
&\ll&
X^{n - d} (\log X) \sideset{}{'}\sum_{\rho_1, \ldots, \rho_m} X^{\beta_1 - 1} \cdots X^{\beta_m - 1}.
\end{eqnarray}
Therefore, by combining  (\ref{eqn 1 m>0}), (\ref{eqn 2 m>0}), (\ref{eqn lem bdd W}) and (\ref{eqn 3 m>0}), we obtain
\begin{eqnarray}
\label{eqn 4 m>0}
&& \Big{|} \sum_{1 \leq q \leq X^{\vartheta_0}} \sum_{\substack{ 1 \leq a \leq q \\  \gcd(a,q) = 1 }}  \frac{1}{\phi(q)^n}  \int_{|\tau| \leq X^{\vartheta_0 + \lambda - d}  }  S_{\mathbf{j},
\emptyset} \left( \frac{a}{q} + \tau \right)  d \tau \Big{|}
\\
\notag
&\ll& \sum_{ 1 \leq r_1, \ldots, r_m \leq X^{\vartheta_0} } R^{-\kappa}  \sideset{}{^*}\sum_{\substack{ \chi_v (\textnormal{mod} \tsp r_v)  \\   (1 \leq v \leq m)  } } X^{n -
d} (\log X) \sideset{}{'}\sum_{\rho_1, \ldots, \rho_m} X^{\beta_1 - 1} \cdots X^{\beta_m - 1}
\\
\notag
&+&
\sum_{ 1 \leq r_1, \ldots, r_m \leq X^{\vartheta_0} }
R^{-\kappa}  \sideset{}{^*}\sum_{ \substack{ \chi_v (\textnormal{mod} \tsp r_v)  \\   (1 \leq v \leq m)  } }
\int_{|\tau| \leq X^{\vartheta_0 + \lambda - d}}  \widetilde{E}  \ d \tau. 
\end{eqnarray}

We begin by bounding the first term on the right hand side of (\ref{eqn 4 m>0}). By the definition of $R$ given in (\ref{defnR}), we clearly have $R^{-\kappa} \leq r_1^{-\kappa}$.
Let $D > 1$ and $A>0$. Then  by Lemma \ref{Gbdd} and Remark \ref{remZFR} we obtain
\begin{eqnarray}
\label{est 1 ell0}
&&
\sum_{ 1 \leq r_1, \ldots, r_m \leq X^{\vartheta_0} } R^{-\kappa}  \sideset{}{^*}\sum_{\substack{ \chi_v (\textnormal{mod} \tsp r_v)  \\   (1 \leq v \leq m)  } } X^{n -
d} (\log X) \sideset{}{'}\sum_{\rho_1, \ldots, \rho_m} X^{\beta_1 - 1} \cdots X^{\beta_m - 1}
\\
\notag
&\ll&
X^{n - d} (\log X) \left( \sum_{ \substack{ 1 \leq r_1 \leq X^{\vartheta_0}} } r_1^{-\kappa}  \sideset{}{^*}\sum_{\chi_1 (\textnormal{mod} \tsp r_1)}
\sideset{}{'}\sum_{\rho_1 } X^{\beta_1 - 1} \right)
\prod_{2 \leq v \leq m}  \
\sum_{ \substack{ 1 \leq r_v \leq X^{\vartheta_0}}} \  \sideset{}{^*}\sum_{\chi_v (\textnormal{mod} \tsp r_v)}
\sideset{}{'}\sum_{\rho_v} X^{\beta_v - 1}
\\
\notag
&\ll&  X^{n-d} (\log X) \sum_{  1 \leq r_1 \leq X^{\vartheta_0} } r_1^{-\kappa} \sideset{}{^*}\sum_{\chi_1 (\textnormal{mod} \tsp r_1)} \sideset{}{'}\sum_{\rho_1} X^{\beta_1
- 1}
\\
\notag
&\ll&  X^{n-d} (\log X)  \sum_{  1 \leq r_1 \leq (\log X)^D } \  \sideset{}{^*}\sum_{\chi_1 (\textnormal{mod} \tsp r_1)} \sum_{\substack{\rho_1 \in B_T \\ L(\rho_1, \chi_1) =
0 }} X^{\beta_1 - 1}
\\
\notag
&+&
X^{n-d} (\log X)^{- D \kappa + 1} \sum_{  1 \leq r_1 \leq X^{\vartheta_0} } \  \sideset{}{^*}\sum_{\chi_1 (\textnormal{mod} \tsp r_1)} \sideset{}{'}\sum_{\rho_1} X^{\beta_1 -
1}
\\
\notag
&\ll& X^{n-d} (\log X)^{- A} + X^{n - d} (\log X)^{- D \kappa + 1}.
\end{eqnarray}

Next we bound the second term on the right hand side of (\ref{eqn 4 m>0}). 
We define
$$
R_{\boldsymbol{\epsilon}} = \textnormal{lcm}  (r_{ \iota_1}, \ldots, r_{\iota_{\epsilon_1 + \cdots + \epsilon_m}}),
$$
where
$$
\{   \iota_1, \ldots, \iota_{\epsilon_1 + \cdots + \epsilon_m}  \} = \{ 1 \leq  i \leq m :  \epsilon_i = 1  \}.
$$
We have from (\ref{defnT}) and (\ref{explicit error}) that
$$
E(X) \ll (\log X)^2 X^{1 - \gamma}.
$$
Then by Lemma \ref{Gbdd} it follows that
\begin{eqnarray}
\label{est ell0 2}
&& \sum_{  1 \leq r_1, \ldots, r_m \leq X^{\vartheta_0} } R^{-\kappa}  \sideset{}{^*}\sum_{\substack{\chi_v (\textnormal{mod} \tsp r_v) \\  (1 \leq v \leq m)}}
\int_{|\tau| \leq X^{\vartheta_0 + \lambda - d}}
\widetilde{E} \ d \tau
\\
&\ll& \sum_{ \boldsymbol{\epsilon} \in \{ 0, 1 \}^{m} \backslash \{ \mathbf{0} \} }  X^{n - d + \vartheta_0 + \lambda} X^{(\vartheta_0 + \lambda - 1) (\epsilon_1 + \cdots +
\epsilon_m)} E(X)^{\epsilon_1 + \cdots + \epsilon_m} \cdot
\notag
\\
&&
\notag
\sum_{ \substack{  1 \leq r_i   \leq X^{\vartheta_0}  \\   (\epsilon_i = 1)} } R_{ \boldsymbol{\epsilon}  }^{- \kappa} \prod_{\epsilon_i = 1 } \  \sideset{}{^*}\sum_{\chi_{i} (\textnormal{mod} \tsp r_{i} ) } 1
\cdot  \prod_{ \epsilon_j = 0  } \,  \sum_{  1 \leq r_j \leq X^{\vartheta_0}} \ \sideset{}{^*}\sum_{\chi_j (\textnormal{mod} \tsp r_j)} \sideset{}{'}\sum_{\rho_j}
(\delta_{\min} X)^{\beta_j - 1}
\notag
\\
\notag
&\ll& \sum_{ \boldsymbol{\epsilon} \in \{ 0, 1 \}^{m} \backslash \{ \mathbf{0} \} }  X^{n - d + \vartheta_0 + \lambda} X^{(\vartheta_0 + \lambda - 1) (\epsilon_1 + \cdots +
\epsilon_m)} E(X)^{\epsilon_1 + \cdots + \epsilon_m} \sum_{ \substack{  1 \leq r_i   \leq X^{\vartheta_0}  \\   (\epsilon_i = 1)} } R_{ \boldsymbol{\epsilon}  }^{- \kappa}
\prod_{\epsilon_i = 1 } \phi(r_i)
\\
\notag
&\ll& (\log X)^{2m}
\sum_{ \boldsymbol{\epsilon} \in \{ 0, 1 \}^{m} \backslash \{ \mathbf{0} \} }  X^{n - d + \vartheta_0 + \lambda} X^{(\vartheta_0 + \lambda - \gamma) (\epsilon_1 +
\cdots + \epsilon_m)} \sum_{ \substack{  1 \leq r_i   \leq X^{\vartheta_0}  \\   (\epsilon_i = 1)} } R_{ \boldsymbol{\epsilon}  }^{- \kappa} \prod_{\epsilon_i = 1 }
\phi(r_i).
\end{eqnarray}

In order to estimate the final expression, we use the following.
\begin{lem}
\label{div fxn lem}
Let $t \in \mathbb{R}_{>1}$ and $h \in \mathbb{N}$. Then
$$
\sum_{1 \leq x_1 \leq B} \cdots \sum_{1 \leq x_h \leq B} \frac{\phi(x_1) \cdots \phi(x_h) }{ \textnormal{lcm}(x_1, \ldots, x_h)^t  }
\ll
\begin{cases}
 B^{\varepsilon}  & \mbox{if } t \geq h + 1, \\
 B^{h - t + 1 + \varepsilon}  & \mbox{if } t < h + 1.
\end{cases}
$$
\end{lem}
\begin{proof}
First we recall the following well-known estimate for the divisor function
$$
\sum_{m| x} 1 \ll x^{\varepsilon}.
$$
Suppose $ t \geq h + 1$. Then since
$$
\frac{\phi(x_1) \cdots \phi(x_h) }{ \textnormal{lcm}(x_1, \ldots, x_h)^t  }
\leq
\frac{1}{\textnormal{lcm}(x_1, \ldots, x_h)},
$$
it follows that
\begin{eqnarray}
\notag
&&\sum_{1 \leq x_1 \leq B} \cdots \sum_{1 \leq x_h \leq B} \frac{\phi(x_1) \cdots \phi(x_h) }{ \textnormal{lcm}(x_1, \ldots, x_h)^t  }
\\
\notag
&\leq&
\sum_{1 \leq x_1 \leq B} \cdots \sum_{1 \leq x_h \leq B} \frac{1}{\textnormal{lcm}(x_1, \ldots, x_h)}
\\
\notag
&\leq&
\sum_{1 \leq x \leq B^{h}} \frac{  \#\{ (x_1, \ldots, x_h) \in \mathbb{N}^h: \textnormal{lcm}(x_1, \ldots, x_h) = x   \} }{ x }
\\
\notag
&\leq&
\sum_{1 \leq x \leq B^{h}} \frac{  \left( \sum_{m| x} 1 \right)^h  }{ x }
\\
&\ll&
B^{\varepsilon}.
\notag
\end{eqnarray}
On the other hand, suppose $ t < h + 1$. Then since
$$
\frac{\phi(x_1) \cdots \phi(x_h) }{ \textnormal{lcm}(x_1, \ldots, x_h)^t }
\leq
\frac{\phi(x_1) \cdots \phi(x_{h - \lfloor t - 1 \rfloor}) }{ \textnormal{lcm}(x_1, \ldots, x_h)^{1 + (t-1) - \lfloor t - 1 \rfloor  } }
\leq
\frac{x_1 \cdots x_{h - \lfloor t - 1 \rfloor} }{ \textnormal{lcm}(x_1, \ldots, x_h)  x_1^{(t-1) - \lfloor t - 1 \rfloor  } },
$$
it follows that
\begin{eqnarray}
\notag
&&\sum_{1 \leq x_1 \leq B} \cdots \sum_{1 \leq x_h \leq B} \frac{\phi(x_1) \cdots \phi(x_h) }{ \textnormal{lcm}(x_1, \ldots, x_h)^t  }
\\
\notag
&\leq& B^{h - \lfloor t - 1 \rfloor - 1 } \sum_{1 \leq x_1 \leq B} \cdots \sum_{1 \leq x_h \leq B} \frac{ x_1 }{ \textnormal{lcm}(x_1, \ldots, x_h) x_1^{(t-1) - \lfloor t - 1 \rfloor  }  }
\\
\notag
&\leq&
B^{h - \lfloor t - 1 \rfloor - (t-1) + \lfloor t - 1 \rfloor  } \sum_{1 \leq x_1 \leq B} \cdots \sum_{1 \leq x_h \leq B} \frac{1}{\textnormal{lcm}(x_1, \ldots, x_h) }
\\
\notag
&\leq&
B^{h - (t - 1) } \sum_{1 \leq x \leq B^{h}} \frac{  \#\{ (x_1, \ldots, x_h) \in \mathbb{N}^h: \textnormal{lcm}(x_1, \ldots, x_h) = x   \} }{ x }
\\
\notag
&\leq&
B^{h - t + 1} \sum_{1 \leq x \leq B^{h}} \frac{  \left( \sum_{m| x} 1 \right)^h  }{ x }
\\
&\ll&
B^{h - t + 1 + \varepsilon}.
\notag
\end{eqnarray}
\end{proof}

We now estimate the summands in (\ref{est ell0 2}).
Suppose $\kappa \geq   1 + \epsilon_1 + \cdots + \epsilon_m$. In this case, by Lemma \ref{div fxn lem} we have
\begin{eqnarray}
\label{div est 1}
&& (\log X)^{2m} X^{n - d + \vartheta_0 + \lambda} X^{(\vartheta_0 + \lambda - \gamma) (\epsilon_1 + \cdots + \epsilon_m)}
\sum_{ \substack{  1 \leq r_i   \leq X^{\vartheta_0}  \\   (\epsilon_i = 1)} } R_{ \boldsymbol{\epsilon}  }^{- \kappa} \prod_{\epsilon_i = 1 } \phi(r_i)
\\
\notag
&\ll&
X^{n - d + \vartheta_0 + \lambda + (\vartheta_0 + \lambda - \gamma ) (\epsilon_1 + \cdots + \epsilon_m) + \varepsilon}
\\
\notag
&\ll&
X^{n - d - \varepsilon},
\end{eqnarray}
because $\gamma > 2 \vartheta_0 + 2 \lambda$.
On the other hand, suppose $\kappa < 1 + \epsilon_1 + \cdots + \epsilon_m$. In this case,  by Lemma \ref{div fxn lem}  we have
\begin{eqnarray}
\label{div est 2}
&& (\log X)^{2m} X^{n - d + \vartheta_0 + \lambda} X^{(\vartheta_0 + \lambda - \gamma) (\epsilon_1 + \cdots + \epsilon_m)} \sum_{ \substack{  1 \leq r_i   \leq X^{\vartheta_0}  \\
(\epsilon_i = 1)} } R_{ \boldsymbol{\epsilon}  }^{- \kappa} \prod_{\epsilon_i = 1 } \phi(r_i)
\\
\notag
&\ll&
X^{n - d + \vartheta_0 + \lambda + (2 \vartheta_0 + \lambda - \gamma ) (\epsilon_1 + \cdots + \epsilon_m) - (\kappa - 1) \vartheta_0 + \varepsilon  }
\\
\notag
&\ll&
X^{n - d + \vartheta_0 + \lambda + (2 \vartheta_0 + \lambda - \gamma ) (\kappa - 1)   - (\kappa - 1) \vartheta_0 + \varepsilon }
\\
\notag
&\ll&
X^{n - d + \vartheta_0 + \lambda + (\vartheta_0 + \lambda - \gamma ) (\kappa - 1)  + \varepsilon }
\\
\notag
&\ll&
X^{n - d - \varepsilon},
\end{eqnarray}
because $\gamma > 2 \vartheta_0 + 2 \lambda$ and $\kappa > 2$.
Therefore, by combining (\ref{eqn 4 m>0}),
(\ref{est 1 ell0}), (\ref{est ell0 2}), (\ref{div est 1}) and (\ref{div est 2}),
we obtain
\begin{eqnarray}
\label{est1}
\Big{|} \sum_{1 \leq q \leq X^{\vartheta_0}} \sum_{ \substack{ 1 \leq a \leq q \\  \gcd(a,q) = 1 } } \frac{1}{\phi(q)^n}  \int_{|\tau| \leq X^{\vartheta_0 + \lambda- d} }
S_{\mathbf{j}, \emptyset } \left( \frac{a}{q} + \tau \right)  d \tau \Big{|}
\ll
X^{n - d} (\log X)^{- C}
\end{eqnarray}
for any $C >0$.

\subsection{Case $\mathbf{k} \neq \emptyset$}
\label{secm=0}
In this case, we only need to consider $q$ divisible by $\widetilde{r}$.
Without loss of generality let $\mathbf{j} = (1, \ldots, m)$ and $\mathbf{k} = (m+1, \ldots, m+{\ell})$.
First we suppose $m=0$, i.e. $\mathbf{j} = \emptyset$.
By the same calculations as in (\ref{usinglem7.4}), we have
\begin{eqnarray}
\notag
\sum_{\substack{ 1 \leq q \leq X^{\vartheta_0} \\  \widetilde{r} |q   }}  \sum_{\substack{ 1 \leq a \leq q \\  \gcd(a,q) = 1 }} \frac{1}{\phi(q)^n} | \mathcal{A}(q, a;
\mathbf{i}; \emptyset ; (1, \ldots, \ell)  ) |
\notag
\ll \widetilde{r}^{- \kappa}.
\end{eqnarray}
Therefore, by Proposition \ref{prop osc int} we obtain
\begin{eqnarray}
\label{est2}
&& \Big{|} \sum_{ \substack{1 \leq q \leq X^{\vartheta_0} \\ \widetilde{r} | q }}    \sum_{\substack{ 1 \leq a \leq q \\  \gcd (a,q) = 1 }}
\frac{ (-1)^{\ell} }{\phi(q)^n}
\int_{|\tau| \leq X^{\vartheta_0 + \lambda
- d}}  S_{\emptyset, \mathbf{k}} \left( \frac{a}{q} + \tau \right)  d \tau \Big{|}
\\
\notag
&\leq& \sum_{ \substack{1 \leq q \leq X^{\vartheta_0} \\ \widetilde{r} | q }}  \sum_{\substack{ 1 \leq a \leq q \\ \gcd (a,q) = 1 }}   \frac{ 1 }{\phi(q)^n} | \mathcal{A}(q, a;
\mathbf{i}; \emptyset ; (1, \ldots, \ell)  ) | \cdot
\\
\notag
&&
 \int_{|\tau| \leq X^{\vartheta_0 + \lambda - d}} \Big{|}  \int_{0}^{\infty} \cdots \int_0^{\infty}
x_{1}^{\widetilde{\beta} - 1} \cdots x_{\ell}^{\widetilde{\beta} - 1} \tsp \varpi(\mathbf{x}) \tsp  e (\tau F(\mathbf{x}) )   \tsp d \mathbf{x} \Big{|}  \tsp  d \tau
\\
&\ll&
\widetilde{r}^{-\kappa} X^{n - d + \ell ( \widetilde{\beta} - 1) }
\int_{|\tau| \leq X^{\vartheta_0 + \lambda} } \Big{|} \int_{0}^{\infty} \cdots \int_0^{\infty}
x_{1}^{\widetilde{\beta} - 1} \cdots x_{\ell}^{\widetilde{\beta} - 1} \cdot
\notag
\\
&&
\prod_{1 \leq u \leq n} \omega(x_{u} - x_{0, u}) \cdot e (\tau F(\mathbf{x}) )   \tsp d \mathbf{x}  \Big{|} \tsp
d \tau
\notag
\\
&\ll&
\widetilde{r}^{-\kappa} X^{n - d + \ell ( \widetilde{\beta} - 1) } (\log X)
\notag
\\
&\ll&
X^{n - d} (\log X)^{-A}
\notag
\end{eqnarray}
for any $A >0$, where the final inequality follows from Remark \ref{remZFR}. 

Next we suppose $m > 0$.  In this case, we have
\begin{eqnarray}
\label{eqn 1 m>0+}
&& \Big{|} \sum_{ \substack{ 1 \leq q \leq X^{\vartheta_0} \\  \widetilde{r}|q  } } \sum_{ \substack{ 1 \leq a \leq q \\  \gcd (a,q) = 1 } }
\frac{ (-1)^{\ell}  }{\phi(q)^n}
\int_{|\tau| \leq X^{\vartheta_0 +
\lambda - d}}  S_{\mathbf{j}, \mathbf{k} } \left( \frac{a}{q} + \tau \right)  d \tau \Big{|}
\\
\notag
&=& \Big{|} \sum_{ \substack{ 1 \leq q \leq X^{\vartheta_0} \\  \widetilde{r}|q  } }  \sum_{ \substack{ \chi'_1, \ldots, \chi'_m \\ (\textnormal{mod }q) } }  \sum_{\substack{
1 \leq a \leq q \\ \gcd (a,q) = 1 }} \frac{ (-1)^{\ell} }{\phi(q)^n} \ \mathcal{A}(q, a; \mathbf{i};  (1, \chi'_1), \ldots, (m, \chi'_m) ; \mathbf{k} ) \cdot  \\
&& \int_{|\tau| \leq X^{\vartheta_0  + \lambda- d}}  \mathcal{W}(\tau; \mathbf{i}; (1, \chi'_1), \ldots, (m, \chi'_m) ; \mathbf{k} )  \tsp d \tau  \Big{|}.
\notag
\end{eqnarray}
As in Section \ref{sec m>0}, let us denote $\chi'_v = \chi_v \chi^0$, where $\chi_v$ is the primitive character modulo $r_v$ which induces $\chi'_v$,
and $\chi^0$ is the principal character modulo $q$. We also denote
\begin{eqnarray}
\notag
R= \textnormal{lcm}(r_1, \ldots, r_m).
\end{eqnarray}
We let 
$$
\delta_{\min} =  \min_{1 \leq u \leq n} (x_{0,u} - \delta).
$$
Let $p \in [(x_{0,u} - \delta) X, (x_{0,u} + \delta) X]$ be a prime.
Since $q \leq X^{\vartheta_0} < \delta_{\min} X$, we have $\gcd(p,q) = 1$ and it follows that
$\chi_v(p) \chi^{0}(p) = \chi_v(p)$.
Consequently, we obtain
$$
\mathcal{W}(\tau; \mathbf{i}; (1, \chi'_1), \ldots, (m, \chi'_m) ; \mathbf{k}  )
= \mathcal{W}(\tau; \mathbf{i}; (1, \chi_1), \ldots, (m, \chi_m) ; \mathbf{k}  ).
$$
By the same calculations as in (\ref{usinglem7.4}), we have
$$
\sum_{\substack{ 1 \leq q \leq X^{\vartheta_0} \\  \widetilde{r} | q  \\  R | q   }}  \sum_{\substack{ 1 \leq a \leq q \\ \gcd (a,q) = 1  }} \frac{1}{\phi(q)^n} | \mathcal{A}(q, a;
\mathbf{i}; (1, \chi_1 \chi^0), \ldots, (m, \chi_m \chi^0) ; \mathbf{k} ) | \ll R^{- \kappa}.
$$
Therefore, the term on the right hand side of (\ref{eqn 1 m>0+}) can be rewritten as
\begin{eqnarray}
\label{eqn 2 m>0+}
&& \Big{|} \sum_{   1 \leq r_1, \ldots, r_m \leq X^{\vartheta_0} } \  \sideset{}{^*}\sum_{\substack{\chi_v (\textnormal{mod} \tsp r_v)  \\  (1 \leq v \leq m) } }
\sum_{\substack{ 1 \leq q \leq X^{\vartheta_0}  \\  \widetilde{r} | q  \\  R |q   }}   \sum_{\substack{ 1 \leq a \leq q \\  \gcd(a,q) = 1 }} \frac{  (-1)^{\ell}  }{\phi(q)^n} \tsp
\mathcal{A}(q, a; \mathbf{i}; (1, \chi_1 \chi^0), \ldots, (m, \chi_m \chi^0) ;\mathbf{k} ) \cdot
\\
\notag
&& \int_{|\tau| \leq X^{\vartheta_0 + \lambda - d} } \mathcal{W}(\tau; \mathbf{i}; (1, \chi_1), \ldots, (m, \chi_m) ; \mathbf{k} ) \tsp d \tau
\Big{|}
\\
\notag
&\ll& \sum_{   1 \leq r_1, \ldots, r_m \leq X^{\vartheta_0} } R^{-\kappa}  \sideset{}{^*}\sum_{\substack{\chi_v (\textnormal{mod} \tsp r_v) \\ (1 \leq v \leq m) } }
\int_{|\tau| \leq X^{\vartheta_0 + \lambda- d} }
| \mathcal{W}(\tau; \mathbf{i}; (1, \chi_1), \ldots, (m, \chi_m) ; \mathbf{k}  )  |  \tsp d \tau.
\end{eqnarray}
We substitute (\ref{eqn lem bdd W}) into (\ref{eqn 2 m>0+}), and
apply Proposition \ref{prop osc int} to the first of the resulting terms in the following manner
\begin{eqnarray}
\label{eqn 3 m>0+}
\\
\notag
&&
\int_{|\tau| \leq X^{\vartheta_0 + \lambda - d} } \Big{|}  \int_{0}^{\infty} \cdots \int_0^{\infty}
\prod_{ 1 \leq  v \leq m}  \sideset{}{'}\sum_{\rho_v} x_v^{\rho_v - 1}
\cdot x_{m+1}^{\widetilde{\beta} - 1 } \cdots x_{m + \ell}^{\widetilde{\beta} - 1 } \tsp \varpi(\mathbf{x}) \tsp  e(\tau F(\mathbf{x}))   \tsp d \mathbf{x}  \Big{|}  \tsp  d
\tau
\\
&=&
\notag
X^{n-d} \int_{|\tau| \leq X^{\vartheta_0 + \lambda} } \Big{|}  \int_{0}^{\infty} \cdots \int_0^{\infty}
\prod_{1 \leq v \leq m}  \sideset{}{'}\sum_{\rho_v} (X x_v)^{\rho_v - 1}
\cdot
\\
&&  (X x_{m+1})^{\widetilde{\beta} - 1 } \cdots (X x_{m + \ell})^{\widetilde{\beta} - 1 } \prod_{ 1 \leq  u \leq n} \omega(x_{u} - x_{0, u}) \cdot   e (\tau F(\mathbf{x}) )   \tsp d
\mathbf{x}  \Big{|}  \tsp  d \tau
\notag
\\
\notag
&\ll&
X^{\ell (\widetilde{\beta} - 1) } X^{n - d } \sideset{}{'}\sum_{\rho_1, \ldots, \rho_m} | X^{\rho_1 - 1} \cdots X^{\rho_m - 1} |   \cdot
\\
&&
\notag
\int_{|\tau| \leq X^{\vartheta_0  + \lambda}} \Big{|} \int_{0}^{\infty} \cdots \int_0^{\infty}
x_1^{\rho_1 - 1} \cdots x_m^{\rho_m - 1} x_{m+1}^{\widetilde{\beta} - 1 } \cdots x_{m + \ell}^{\widetilde{\beta} - 1 }  \prod_{1 \leq  u \leq n} \omega(x_{u} - x_{0, u})
\cdot  e (\tau F(\mathbf{x}) )   \tsp d \mathbf{x}   \Big{|} \tsp  d \tau
\\
\notag
&\ll&
X^{n - d} (\log X) \sideset{}{'}\sum_{\rho_1, \ldots, \rho_m} X^{\beta_1 - 1} \cdots X^{\beta_m - 1}.
\end{eqnarray}
Therefore, by combining (\ref{eqn 1 m>0+}), (\ref{eqn 2 m>0+}), (\ref{eqn lem bdd W}) and (\ref{eqn 3 m>0+}), we obtain
\begin{eqnarray}
\label{eqn 4 m>0+}
&& \Big{|} \sum_{  \substack{ 1 \leq q \leq X^{\vartheta_0}  \\  \widetilde{r} | q  } } \sum_{\substack{ 1 \leq a \leq q \\  \gcd(a,q) = 1 }}
\frac{  (-1)^{\ell} }{\phi(q)^n}
\int_{|\tau| \leq X^{\vartheta_0 + \lambda - d}  }  S_{\mathbf{j},
\mathbf{k}} \left( \frac{a}{q} + \tau \right)  d \tau \Big{|}
\\
\notag
&\ll& \sum_{ 1 \leq r_1, \ldots, r_m \leq X^{\vartheta_0} } R^{-\kappa}  \sideset{}{^*}\sum_{\substack{ \chi_v (\textnormal{mod} \tsp r_v)  \\   (1 \leq v \leq m)  } } X^{n -
d} (\log X) \sideset{}{'}\sum_{\rho_1, \ldots, \rho_m} X^{\beta_1 - 1} \cdots X^{\beta_m - 1}
\\
\notag
&+&
\sum_{ 1 \leq r_1, \ldots, r_m \leq X^{\vartheta_0} }
R^{-\kappa}  \sideset{}{^*}\sum_{ \substack{ \chi_v (\textnormal{mod} \tsp r_v)  \\   (1 \leq v \leq m)  } }
\int_{|\tau| \leq X^{\vartheta_0 + \lambda - d}}  \widetilde{E}  \ d \tau.
\end{eqnarray}
Since the right hand side of (\ref{eqn 4 m>0+}) is identical to that of (\ref{eqn 4 m>0}),
we can bound it in the same manner. Therefore, it follows that
\begin{eqnarray}
\Big{|} \sum_{   \substack{ 1 \leq q \leq X^{\vartheta_0}  \\  \widetilde{r} | q  } } \sum_{ \substack{ 1 \leq a \leq q \\  \gcd(a,q) = 1 } } \frac{ (-1)^{\ell} }{\phi(q)^n}  \int_{|\tau| \leq X^{\vartheta_0 + \lambda- d} }
S_{\mathbf{j}, \mathbf{k} } \left( \frac{a}{q} + \tau \right)  d \tau \Big{|}
\ll
X^{n - d} (\log X)^{- C}
\label{est3}
\end{eqnarray}
for any $C >0$.

\subsection{Proof of Proposition \ref{prop major}} \label{secmajormainterm}
We combine (\ref{S and S1}) and (\ref{MAJOR1}) with the estimates from Sections \ref{sec m>0} and \ref{secm=0}, namely (\ref{est1}), (\ref{est2}) and (\ref{est3}).
As a result, we obtain
\begin{eqnarray}
\label{eqnsec3.3}
&&\int_{\mathfrak{M}^+(\vartheta_0; \lambda )} S(\alpha) \tsp d\alpha
\\
\notag
&=&\int_{\mathfrak{M}^+(\vartheta_0; \lambda)} S^*(\alpha) \tsp d\alpha
+ O \left(  X^{n - d + 3 \vartheta_0  + \lambda - \frac12}  \right)
\\
&=&
\notag
\sum_{1 \leq q \leq X^{\vartheta_0}} \sum_{\substack{ 1 \leq a \leq q \\  \gcd (a,q) = 1 }} \frac{1}{\phi(q)^n} \tsp \mathcal{A} (q, a; (1, \ldots, n) ; \emptyset ;  \emptyset )
\cdot
\int_{|\tau|
 \leq X^{\vartheta_0 + \lambda - d}}
\mathcal{W} (\tau;(1, \ldots, n) ;  \emptyset ; \emptyset ) \tsp d \tau
\\
\notag
&+& O \left(X^{n - d + 4 \vartheta_0 + 2 \lambda - 1} +  X^{n - d + 3 \vartheta_0  + \lambda - \frac12}  +  \frac{X^{n - d}}{(\log X)^{A} } \right)
\end{eqnarray}
for any $A > 0$. We now estimate the singular series and the singular integral by standard arguments; we keep the details to a minimum.
We begin with the singular series (see for example \cite[Section 7]{CM} for more details). Since
$$
\min \left(  \frac{\textnormal{codim} \tsp V^*_{F } }{  (d-1) 2^{d-1}  + 2  (2 d - 1)    },   \frac{  (2d-1) \textnormal{codim} \tsp V^*_{F}   }{ (d-1) d 2^{d+1} }  \right)  > 2,
$$
by invoking\footnote{We may invoke \cite[Lemma 10]{CM} instead, but since the proof contains a minor oversight, we prove (\ref{bound on Aq}) in Appendix \ref{A}. In the final paragraph in the proof of \cite[Lemma 10]{CM}, they make the following statement under the assumptions that $t \in \ZZ_{> d}$
and $a_1, \ldots, a_r \in \mathbb{U}_q$: Set $P = p^{t-1}$ and
$q_1 = p^{t-d}$. Then, for each $i = 1, \ldots, r$,
$$
2|q'a_i - a'_i q_1| \leq P^{-(d-1) + (d-1) r \theta}
$$
and
$$
1 \leq q' \leq P^{(d-1) r \theta}
$$
cannot be satisfied if $\theta < 1/(d-1)r$.

However, this statement can be seen to be not true by the following counterexample.
Let $d \geq 2$, $r=1$ and
$$
\frac{t - d}{  (t-1) (d-1)  } \leq \theta < \frac{1}{d-1}.
$$
Then we can choose
$$
q'= q_1 = p^{t-d}  =  p^{(t-1) \frac{t-d}{t-1}  }   \leq   P^{ (d-1) \theta }
\quad
\textnormal{ and }
\quad
a'_1 = a_1,
$$
which yields $q'a_1 - a'_1 q_1 = 0.$}
(\ref{bound on Aq}) in Appendix \ref{A}, it follows that there exists $\xi_1 >0$ such that
\begin{eqnarray}
\notag
&&\sum_{1 \leq q \leq X^{\vartheta_0}} \sum_{\substack{ 1 \leq a \leq q \\ \gcd (a,q) = 1 }} \frac{1}{\phi(q)^n} \tsp \mathcal{A} (q, a; (1, \ldots, n) ; \emptyset ;  \emptyset )
\\
\notag
&=&
\sum_{1 \leq q \leq X^{\vartheta_0}} \sum_{\substack{ 1 \leq a \leq q \\ \gcd  (a,q) = 1 }} \frac{1}{\phi(q)^n} \sum_{\substack{ \mathbf{h} \in \mathbb{U}_q^n } }  e \left(
\frac{a}{q}  F(\mathbf{h} )   \right)
\\
&=& c_1(F) + O \left( \sum_{q >  X^{\vartheta_0}} q^{- 1- \xi_1} \right)
\notag
\\
\notag
&=& c_1(F) +
O ( X^{- \xi_1 \vartheta_0}),
\end{eqnarray}
where $c_1(F) \geq 0$ is a constant depending only on $F$.

Next we deal with the singular integral (see for example \cite[Section 5]{BP} or \cite[\S 2]{SS} for more details). The same argument to establish \cite[(5.5)]{BP} yields
$$
\int_{0}^{\infty} \cdots \int_{0}^{\infty}
\prod_{1 \leq u \leq n} \omega(x_{u} - x_{0, u}) \cdot e( \tau F(\mathbf{x}))  \tsp  d \mathbf{x} \ll
\min \left( 1,  |\tau|^{ - \frac{ \textnormal{codim} \tsp V^*_{F}}{(d-1) 2^{d-1}}  +  \varepsilon}  \right).
$$
Therefore, since $\textnormal{codim} \tsp V^*_{F} > (d-1) 2^{d-1} + 1$, it follows that there exists $\xi_2 >0$ such that
\begin{eqnarray}
\notag
&&\int_{|\tau| \leq X^{\vartheta_0 + \lambda - d}}
\mathcal{W}(\tau;(1, \ldots, n) ;  \emptyset ; \emptyset ) \tsp d \tau
\\
\notag
&=&
X^{n-d} \int_{|\tau| \leq X^{\vartheta_0 + \lambda}} \int_{0}^{\infty} \cdots \int_{0}^{\infty}
\prod_{1 \leq  u \leq n} \omega(x_{u} - x_{0, u}) \cdot e( \tau F(\mathbf{x}))  \tsp  d \mathbf{x} \tsp d \tau
\\
\notag
&=&  c_2(F; \omega, \mathbf{x}_0) \, X^{n - d}
+
O \left(  X^{n - d}   \int_{|\tau| > X^{\vartheta_0 + \lambda}} \min \left( 1,  |\tau|^{ - \frac{ \textnormal{codim} \tsp V^*_{F}}{(d-1) 2^{d-1}}  +  \varepsilon}
\right)  d \tau   \right)
\\
\notag
&=& c_2(F; \omega, \mathbf{x}_0) \, X^{n - d} + O(X^{n-d - \xi_2}),
\end{eqnarray}
where $c_2(F; \omega, \mathbf{x}_0) \geq  0$ is a constant depending only on $F$, $\omega$ and $\mathbf{x}_0$.
The constant $c_1(F) c_2(F; \omega, \mathbf{x}_0)$ is a product of local densities; in fact, $c_1(F) c_2(F; \omega, \mathbf{x}_0) > 0$ provided $F$ satisfies the local
conditions \textnormal{($\star$)} and $\omega$ is as in this section. Therefore, by combining these estimates with (\ref{eqnsec3.3}), we obtain Proposition \ref{prop major} with
$c(F; \omega, \mathbf{x}_0) = c_1(F) c_2(F; \omega, \mathbf{x}_0)$.

\section{Preliminaries for the minor arcs analysis}
\label{prem minor}
We collect some definitions and results from \cite{Y2}, which we will need in the next section.
Let $F \in \mathbb{Z}[x_1, \ldots, x_n]$ be a homogeneous form.
Given partitions of variables $\mathbf{x} = (\mathbf{z}, \mathbf{w})$ and
$\mathbf{z} = (\mathbf{s}, \mathbf{t})$,  let us denote
$$
F_{\mathbf{z}}(\mathbf{z}) = F(\mathbf{z}, \mathbf{0}) = F(\mathbf{x}) |_{\mathbf{w} = \mathbf{0}}
$$
and
\begin{equation}
\label{den G 1}
\mathfrak{G}(\mathbf{s}, \mathbf{t} ) = F_{\mathbf{z}}(\mathbf{s}, \mathbf{t}) - F_{\mathbf{z}}(\mathbf{s}, \mathbf{0}) - F_{\mathbf{z}}(\mathbf{0}, \mathbf{t}).
\end{equation}
It is clear that $\mathfrak{G}(\mathbf{s}, \mathbf{0} )$ and $\mathfrak{G}(\mathbf{0}, \mathbf{t} )$ are the zero polynomials.
With these notation we consider two cases based on the structure of $F$.
\begin{defn}
\label{dichotomy}
We define a structural dichotomy of $F$ with respect to $\mathcal{C}_0 > 0$ as follows:
\begin{enumerate}[(I)]
\item There exist partitions of variables $\mathbf{x} = (\mathbf{z}, \mathbf{w})$ and
$\mathbf{z} = (\mathbf{s}, \mathbf{t})$ such that $\textnormal{codim} \tsp V^*_{\mathfrak{G}} > \mathcal{C}_0$.
\item Given any partitions $\mathbf{x} = (\mathbf{z}, \mathbf{w})$ and
$\mathbf{z} = (\mathbf{s}, \mathbf{t})$, we have $\textnormal{codim} \tsp V^*_{\mathfrak{G}} \leq \mathcal{C}_0$.
\end{enumerate}
\end{defn}
Let us set
\begin{equation}
\label{defn C0}
\mathcal{C}_0 =  4 (d-1) 2^d + 1.
\end{equation}
The situation when $F$ satisfies (I) is easier to handle, and this is dealt with in Section \ref{conc}.
We deal with the more challenging case when $F$ satisfies (II) in Sections \ref{minor1} and \ref{minor2}. The following
result \cite[Lemma 4.3]{Y2} played a crucial role in the author's previous work \cite{Y2}.
\begin{lem}
\label{lemma rank concn}
Let $F \in \mathbb{Z}[x_1, \ldots, x_n]$ be a homogeneous form of degree $d \geq 2$ satisfying \textnormal{(II)} of Definition \ref{dichotomy}. Suppose we have a partition of the $\mathbf{x}$ variables
into $H$ sets $\mathbf{x} = (\mathbf{y}_0, \ldots, \mathbf{y}_{H-1})$. For each $0 \leq i \leq H-1$, let us denote
$$
F_{i} (\mathbf{y}_i) = F(\mathbf{x})|_{ \mathbf{y}_{\ell} = \mathbf{0} \tsp (\ell \not = i) }.
$$
Then there exists $j_0 \in \{ 0, 1, \ldots, H - 1 \}$ satisfying
\begin{eqnarray}
\label{rank concn}
\textnormal{codim} \tsp V^*_{F_{j_0}} \geq    \frac{\textnormal{codim} \tsp V^*_{F}- (H-1) \mathcal{C}_0 }{H}.
\end{eqnarray}
\end{lem}

Let $\mathcal{G}( \u; \v )$ be a polynomial in variables $\u = (u_1, \ldots, u_h)$ and $\v = (v_1, \ldots, v_h)$, and
say it is bihomogeneous of bidegree $(d_1, d_2)$ if
$$
\mathcal{G}( a \u; b \v ) = a^{d_1} b^{d_2} \mathcal{G}( \u; \v )
$$
for all $a, b \in \CC$.
We define the following affine variety
\begin{equation}
\label{singlocbhmg}
V_{\mathcal{G}, 2}^* = \left\{ (\u, \v)  \in  \mathbb{A}_{ \mathbb{C} }^{2 h}:
\frac{\partial \mathcal{G}}{\partial v_i} ( \u; \v)
= 0 \quad (1 \leq i \leq h) \right \}.
\end{equation}
We take the partial derivatives with respect to the second set of variables (in the notation of $\mathcal{G}$) for $V_{\mathcal{G}, 2}^*$.
The following was the key estimate in establishing the main result in \cite{Y1}; we present a
different proof which improves \cite[Theorem 5.1]{Y1} by almost a factor of $2$ under the additional assumption that $F$ satisfies (II) of Definition \ref{dichotomy}.
\begin{prop}
\label{semiprimeprop}
Let $F \in \mathbb{Z}[x_1, \ldots, x_n]$ be a homogeneous form of degree $d \geq 2$ satisfying (II) of Definition \ref{dichotomy}. Let
$\z = (z_1, \ldots, z_h)$ with $h \geq 1$. We let
$\mathbf{x} = (\z, \mathbf{w})$ be a partition of variables and denote
$$
\mathfrak{F} (\z) =  F_{\z} (\z) =  F( \z, \mathbf{0}) = F(\mathbf{x}) |_{\mathbf{w} = \mathbf{0}}.
$$
Let us define a bihomogeneous form
$$
\mathcal{G}( \u; \v ) = \mathfrak{F}(u_1 v_1, \ldots, u_h v_h).
$$
Then
$$
\textnormal{codim} \tsp  V^*_{\mathcal{G}, 2}  \geq  \textnormal{codim}  \tsp   V^*_{\mathfrak{F}} - \mathcal{C}_0.
$$
\end{prop}
\begin{proof}
We begin by proving the following result.
\begin{lem}
\label{lemlem}
Let $G_j \in \CC[v_1, \ldots, v_k]$ be a homogeneous form for each $1 \leq j \leq s$.
We define
$$
L =  \{ (v_1, \ldots, v_k) \in \mathbb{A}_{\CC}^k: G_j(v_1, \ldots, v_k) = 0  \quad (1 \leq j \leq s) \}
$$
and
$$
\widetilde{L} =  \{ (u_1, \ldots, u_k, v_1, \ldots, v_k) \in \mathbb{A}_{\CC}^{2k}: G_j(u_1 v_1, \ldots, u_k v_k) = 0  \quad (1 \leq j \leq s) \}.
$$
Then
$$
\dim \widetilde{L} \leq \dim L + k.
$$
\end{lem}
\begin{proof}
Let $W$ be an irreducible component of $\widetilde{L}$.
Suppose there exists $1 \leq i \leq k$ such that
\begin{equation}
\label{W set}
W \subseteq V(u_i) \cap V(v_i).
\end{equation}
Without loss of generality let us suppose $i = 1$. Then $W$ is of the form
$$
W = \{ (0, \u', 0, \mathbf{v}') \in \mathbb{A}_{\CC}^{2k}:  (\u', \mathbf{v}') \in W'     \},
$$
where $\u' = (u_2, \ldots, u_k)$, $\mathbf{v}' = (v_2, \ldots, v_k)$ and $W' \subseteq \mathbb{A}_{\CC}^{2k-2}$ is an irreducible affine variety.
It is easy to see that
$$
W \subsetneq \{ (0, \u', v_1, \mathbf{v}') \in \mathbb{A}_{\CC}^{2k}:  v_1 \in \mathbb{A}_{\CC},  (\u', \mathbf{v}') \in W'     \} \subseteq \widetilde{L},
$$
and this is a contradiction because
$$
\{ (0, \u', v_1, \mathbf{v}') \in \mathbb{A}_{\CC}^{2k}:  v_1 \in \mathbb{A}_{\CC},  (\u', \mathbf{v}') \in W'     \} \cong \mathbb{A}_{\CC} \times  W'
$$
is irreducible; here $\cong$ denotes an isomorphism of affine varieties. Therefore, we obtain that (\ref{W set}) does not hold for any $1 \leq i \leq k$.
In particular, this means that there exists $(\eta_1, \ldots, \eta_k, \zeta_1, \ldots, \zeta_k) \in W$
such that either $\eta_i \neq 0$ or  $\zeta_i \neq 0$ for each $1 \leq i \leq k$ by the following reason.
Let us suppose otherwise, in which case it follows that
$$
W = \bigcup_{1 \leq i \leq k} W \cap (V(u_i) \cap V(v_i)).
$$
Then
$$
W =  W \cap (V(u_{i_0}) \cap V(v_{i_0}))
$$
for some $1 \leq i_0  \leq k$, because $W$ is irreducible; however, this contradicts the fact that (\ref{W set}) does not hold for any $1 \leq i \leq k$.
Without loss of generality let us suppose $\eta_i \neq 0$ for all $1 \leq i \leq k$. Then we obtain
\begin{eqnarray}
\emptyset &\neq& W \cap \bigcap_{1 \leq i \leq k} V(u_i - \eta_i)
\notag
\\
\notag
&\subseteq&
\{ (\eta_1, \ldots, \eta_k) \} \times
\{ (v_1, \ldots, v_k) \in \mathbb{A}_{\CC}^k: G_j( \eta_1 v_1, \ldots, \eta_k v_k) = 0  \quad (1 \leq j \leq s) \}
\\
\notag
&\cong&
L,
\end{eqnarray}
and by \cite[Proposition I.7.1]{H} that
$$
\dim W  - k \leq \dim  \left( W \cap \bigcap_{1 \leq i \leq k} V(u_i - \eta_i) \right).
$$
Therefore, we have
$$
\dim W  - k  \leq  \dim L.
$$
Since this holds for any irreducible component $W$ of $\widetilde{L}$, the result follows.
\end{proof}

By the definitions of $\mathcal{G}$ and $V^*_{\mathcal{G}, 2}$, we have
\begin{eqnarray}
\notag
V^*_{\mathcal{G}, 2} &=& \left\{ (\u, \v)  \in  \mathbb{A}_{ \mathbb{C} }^{2 h} :
u_i \frac{\partial \mathfrak{F}}{\partial v_i} (u_1 v_1, \ldots, u_h v_h)
= 0 \quad (1 \leq i \leq h) \right \}
\\
\notag
&=&
\bigcup_{I \subseteq \{1, \ldots, h \}} Z_{I},
\end{eqnarray}
where
$$
Z_I =
\left\{ (\u, \v)  \in  \mathbb{A}_{ \mathbb{C} }^{2 h} :
\frac{\partial \mathfrak{F}}{\partial v_i} (u_1 v_1, \ldots, u_h v_h)
= 0 \quad (i \in I ) \right \}
\medcap  \,
\bigcap_{i \not \in I} V(u_i).
$$
Let us fix a choice of $I$ and without loss of generality we assume $I = \{1, \ldots, k \}$, where
$0 \leq k \leq h$ (we let $k = 0$ if $I = \emptyset$). Let us suppose $k \geq 1$. Then
$$
Z_I \cong
\left\{ (u_1, \ldots, u_k, v_1, \ldots, v_k)  \in  \mathbb{A}_{ \mathbb{C} }^{2k} :
\frac{\partial \mathfrak{F}}{\partial v_i} (u_1 v_1, \ldots, u_k v_k, \mathbf{0} )
= 0 \quad (1 \leq  i  \leq k) \right \}
\times \mathbb{A}^{h - k}_{\CC}.
$$
Let us define
$$
T_k
=
\left\{ (v_1, \ldots, v_k)  \in  \mathbb{A}_{ \mathbb{C} }^k :
\frac{\partial \mathfrak{F}}{\partial v_i} (v_1, \ldots, v_k, \mathbf{0} )
= 0 \quad (1 \leq  i  \leq k) \right \}
$$
and
$$
T'_k
=
\left\{ (v_{k+1}, \ldots, v_h)  \in  \mathbb{A}_{ \mathbb{C} }^{h - k} :
\frac{\partial \mathfrak{F}}{\partial v_i} (\mathbf{0},  v_{k+1}, \ldots, v_h)
= 0 \quad (k + 1 \leq  i  \leq h) \right \}.
$$
Then by Lemma \ref{lemlem} it follows that
\begin{eqnarray}
\notag
\dim Z_I
\leq
\dim T_k + h.
\end{eqnarray}
The result follows on proving
\begin{eqnarray}
\label{ineqqq}
\dim T_k \leq \dim V_{\mathfrak{F}}^* + \mathcal{C}_0.
\end{eqnarray}
With this estimate it follows that
$$
\dim Z_I \leq \dim V_\mathfrak{F}^* + h + \mathcal{C}_0;
$$
note the inequality is trivial when $k = 0$, because $Z_I \cong \mathbb{A}^h_{\CC}$ in this case.
Since this holds for all $I \subseteq \{1, \ldots, h\}$, we then obtain
$$
\dim  V^*_{\mathcal{G}, 2} \leq \dim V_\mathfrak{F}^* + h + \mathcal{C}_0,
$$
and equivalently
$$
2h - \dim V^*_{\mathcal{G}, 2} \geq h - \dim V_\mathfrak{F}^* - \mathcal{C}_0.
$$

We now prove (\ref{ineqqq}) and complete the proof. We assume $1 \leq k < h$ as the inequality is trivial when $k = h$. Let us consider the following partition of variables
$\v = (\s, \t)$, where $\s = (v_1, \ldots, v_k)$ and $\t = (v_{k+1}, \ldots, v_h)$.
Let
$$\mathfrak{F}_{1}(v_1, \ldots, v_k) = \mathfrak{F}(v_1, \ldots, v_k, \mathbf{0}), \quad
\mathfrak{F}_{2}(v_{k+1}, \ldots, v_h) = \mathfrak{F}(\mathbf{0}, v_{k+1}, \ldots, v_h)$$
and
\begin{equation}
\notag
\mathfrak{G} (\v) = \mathfrak{F}(\v) - \mathfrak{F}_{1}(v_1, \ldots, v_k) - \mathfrak{F}_{2}(v_{k+1}, \ldots, v_h).
\end{equation}
Then on recalling the definition of $\mathfrak{F}$, we see that
\begin{eqnarray}
\notag
\mathfrak{G} (\v) = \mathfrak{G}(\s, \t)  = \mathfrak{F}(\s, \t) - \mathfrak{F}(\s, \mathbf{0}) - \mathfrak{F}( \mathbf{0}, \t)
\end{eqnarray}
as defined in (\ref{den G 1}). Therefore, since $F$ satisfies (II) of Definition \ref{dichotomy} it follows that
\begin{equation}
\label{C0}
h -    \dim V_{\mathfrak{G}}^*
= \textnormal{codim} \tsp V_{\mathfrak{G}}^* \leq \mathcal{C}_0.
\end{equation}
We define the following homogeneous forms:
$$
\lambda_i (\v) = \frac{\partial \mathfrak{F}}{ \partial v_i} (\v)   \quad (1 \leq i \leq h), \quad
\tau_i (\v) = \frac{\partial \mathfrak{G} }{ \partial v_i}(\v) \quad (1 \leq i \leq h),
$$
$$
\varphi_i (\v) = \frac{\partial \mathfrak{F}_{1}}{ \partial v_i} (v_1, \ldots, v_k)  \quad (1 \leq i \leq k), \quad
\varphi_i (\mathbf{v}) = \frac{\partial \mathfrak{F}_{2}}{ \partial v_i} (v_{k+1}, \ldots, v_h) \quad (k + 1 \leq i \leq h).
$$
Clearly $\lambda_i = \varphi_i + \tau_i$ $(1 \leq i \leq h)$.
Let us define the following ideals: $I_1 = (\varphi_1, \ldots, \varphi_k)$, $I_2 = (\varphi_{k+1}, \ldots, \varphi_h)$ and $J = (\tau_1, \ldots, \tau_h)$.
Then we have
$$
V(I_1) \cong T_k \times \mathbb{A}_{\CC}^{h-k}, \quad V(I_2) \cong T'_{k} \times \mathbb{A}_{\CC}^{k}, \quad
V(J) = V_{\mathfrak{G}}^*
$$
and
$$
V(I_1) \cap V(I_2) \cap V(J) =   V(I_1 + I_2 + J) \subseteq V(\lambda_1, \ldots, \lambda_h) = V_\mathfrak{F}^* \subseteq \mathbb{A}_{\mathbb{C}}^h.
$$
Therefore, by \cite[Proposition I.7.1]{H} we obtain
\begin{eqnarray}
\notag
\dim V^*_{\mathfrak{F}}
&\geq& \dim ( V(I_1) \cap V(I_2) \cap V(J) )
\\
\notag
&\geq&  \dim V(I_1)  +  \dim V(I_2)  +  \dim V(J) - 2h
\\
\notag
&=& \dim T_k + (h-k)  + \dim T'_k + k   -  h  - ( h -    \dim V_{\mathfrak{G}}^*)
\\
\notag
&\geq&  \dim T_k  + \dim T'_k   -  ( h -    \dim V_{\mathfrak{G}}^*)
\notag
\\
\notag
&\geq&  \dim T_k   -  \mathcal{C}_0,
\end{eqnarray}
where the final inequality follows from (\ref{C0}).
\end{proof}

Next we recall the Heath-Brown identity for the von Mangoldt function (see \cite[Proposition 13.3]{IK}).
In order to apply this identity, we mainly follow the exposition in \cite[Section 3]{Poly}.
Let $\mathfrak{L} \geq 1$ and $z \geq 1$. Then for any $x < 2 z^{\mathfrak{L}}$ we have
\begin{eqnarray}
\label{HB iden}
\Lambda(x)  = \sum_{1 \leq \ell \leq \mathfrak{L}} \nu_\ell(x),
\end{eqnarray}
where
\begin{eqnarray}
\label{HB iden2}
\nu_\ell(x)  =  (-1)^{\ell+1} {\mathfrak{L} \choose \ell}
\sum_{ \substack{ m_1 \cdots m_\ell n_1 \cdots n_\ell = x \\ m_1, \ldots, m_\ell \leq z }}
\mu(m_1) \cdots \mu(m_\ell) \log n_\ell,
\end{eqnarray}
and  $\mu$ denotes the M\"{o}bius function, i.e.
\begin{equation}
\label{defmob}
\mu(m) = \begin{cases}
           0 & \mbox{if } p^2|m \mbox{ for some } p \in \wp, \\
           1 & \mbox{if } m = 1, \\
           (-1)^s & \mbox{if } m \mbox{ is a product of $s$ distinct primes}.
         \end{cases}
\end{equation}
Let $\omega$ and $\psi_1, \ldots, \psi_n$ be as in the statement of Theorem \ref{mainthm}, and let
\begin{equation}
\label{supp psi}
\supp \psi_i = [a_i X, b_i X] \quad (1 \leq i \leq n).
\end{equation}
In particular, it follows that we only consider
$$
\varepsilon X < x_1, \ldots, x_n < (1 - \varepsilon) X.
$$
Therefore, we shall take
$$
z = X^{\frac{1}{\mathfrak{L}}}.
$$
Then by the Heath-Brown identity we have
\begin{eqnarray}
\label{S11}
S(\alpha) = \sum_{1 \leq \ell_1 \leq \mathfrak{L} } \cdots \sum_{1 \leq \ell_n \leq \mathfrak{L} } S_{\boldsymbol{\ell}} (\alpha),
\end{eqnarray}
where
\begin{eqnarray}
\label{S1111}
S_{ \boldsymbol{\ell} } (\alpha) = \sum_{\mathbf{x} \in [0, X]^n } \prod_{1 \leq i \leq n} \psi_i(x_i)  \nu_{\ell_i}(x_i) \cdot e( \alpha F(\mathbf{x}) ).
\end{eqnarray}
We fix a constant $\Theta > 1$. Let $\Psi: \mathbb{R} \to \mathbb{R}$ be a smooth function such that $\supp \Psi = [- \Theta, \Theta]$,
$$
\Psi(y) = 1 \quad   ( y \in [-1, 1])
$$
and obeys the derivative estimate
$$
|\Psi^{(s)}(y)| \ll 1 \quad (y \in \mathbb{R})
$$
for any fixed $s \geq 0$. We then have a smooth partition of unity
$$
1 = \sum_{T \in \mathfrak{D}} \Psi_{T}(x) \quad (x \in \mathbb{N})
$$
indexed by the multiplicative semigroup
\begin{eqnarray}
\label{DDDDD}
\mathfrak{D} = \{ \Theta^t: t \in \ZZ_{\geq 0} \},
\end{eqnarray}
where
$$
\Psi_{T}(x) = \Psi \left(  \frac{x}{T}  \right) - \Psi \left( \frac{\Theta x}{T} \right).
$$
It is clear that
\begin{equation}
\label{supp PSI}
\supp \Psi_T  \cap  \mathbb{R}_{>0} \subseteq [\Theta^{-1}T, \Theta T].
\end{equation}
Let $1 \leq i \leq n$.
Then
\begin{eqnarray}
\label{big mess for nu}
\nu_{\ell_i} (x_i)
&=&
(-1)^{\ell_i +1} {\mathfrak{L} \choose \ell_i} \sum_{ \substack{ m_1 \cdots m_{\ell_i}  n_1 \cdots n_{\ell_i} = x_i \\ m_1, \ldots, m_{\ell_i} \leq z }}
\mu(m_1) \cdots \mu(m_{\ell_i}) (\log n_{\ell_i})
\\
&=&
\notag
\sum_{ M^{(i)}_1, \ldots, M^{(i)}_{\ell_i} \in \mathfrak{D} } \sum_{ N^{(i)}_1, \ldots, N^{(i)}_{\ell_i} \in \mathfrak{D} }
\Phi_{\ell_i; \mathbf{M}^{(i)}, \mathbf{N}^{(i)}   } (x_i),
\end{eqnarray}
where
\begin{eqnarray}
\notag
\Phi_{\ell_i; \mathbf{M}^{(i)}, \mathbf{N}^{(i)}   } (x_i) &=& (-1)^{\ell_i +1} {\mathfrak{L} \choose \ell_i} \sum_{ \substack{ m_1 \cdots m_{\ell_i}  n_1 \cdots n_{\ell_i} = x_i \\ m_1, \ldots, m_{\ell_i} \leq z }}
\mu(m_1) \Psi_{M^{(i)}_1}(m_1)  \cdots \mu(m_{\ell_i}) \Psi_{M^{(i)}_{\ell_i}}(m_{\ell_i})
\\
\notag
&&  \Psi_{N^{(i)}_1}(n_1)  \cdots \Psi_{N^{(i)}_{\ell_i}}(n_{\ell_i})  (\log n_{\ell_i}).
\end{eqnarray}
We see that $\Phi_{\ell_i; \mathbf{M}^{(i)}, \mathbf{N}^{(i)}   } (x_i)   \psi_i(x_i)$ vanishes unless
\begin{eqnarray}
\label{cond1}
M^{(i)}_1, \ldots, M^{(i)}_{\ell_i} \ll z = X^{\frac{1}{\mathfrak{L}}}
\end{eqnarray}
and
\begin{eqnarray}
\label{cond2}
a_i X \Theta^{-2 \mathfrak{L} }  \leq M^{(i)}_1 \cdots M^{(i)}_{\ell_i} N^{(i)}_1 \cdots N^{(i)}_{\ell_i}  \leq  b_i X \Theta^{2 \mathfrak{L}}.
\end{eqnarray}
We define $\Xi_i$ to be the set of all $(\mathbf{M}^{(i)}, \mathbf{N}^{(i)}) \in \mathfrak{D}^{2 \ell_i}$ satisfying (\ref{cond1}) and (\ref{cond2}).
We can easily deduce that
\begin{eqnarray}
\label{LOG}
\# \Xi_i \ll 1,
\end{eqnarray}
where the implicit constant depends only on $\mathfrak{L}, a_1, \ldots, a_n, b_1, \ldots, b_n$ and $\Theta$.

We now make use of the following combinatorial lemma.
\begin{lem} \cite[Lemma 3.1]{Poly}
\label{comblem}
Let $\frac{1}{6} < \sigma_0 < \frac12$, and let $t_1, \ldots, t_R \in \RR_{\geq 0}$ be such that $t_1 + \cdots + t_R = 1$.
Then at least one of the following two statements holds:
\begin{itemize}
\item[(Type I)] There exists $1 \leq i \leq R$ such that  $t_i \geq \frac{1}{2} + \sigma_0$.

\item[(Type II)] There is a partition $\{1, \ldots, R \} = \mathcal{S} \cup \mathcal{T}$ such that
$$ \frac{1}{2} - \sigma_0 < \sum_{i \in \mathcal{S}} t_i \leq \sum_{i \in \mathcal{T}} t_i < \frac{1}{2} + \sigma_0.$$
\end{itemize}
\end{lem}

\begin{lem}
\label{NU}
Let $2 \sigma_0 > 1/ \mathfrak{L}$.
Suppose $(\mathbf{M}^{(i)}, \mathbf{N}^{(i)}) \in \Xi_i$. Then we may write $\Phi_{i; \mathbf{M}^{(i)}, \mathbf{N}^{(i)}   } (x_i)$
in one of the following two ways:
\begin{itemize}
\item[(Type I)] We have
$$
\Phi_{\ell_i; \mathbf{M}^{(i)}, \mathbf{N}^{(i)}   } (x_i) = \sum_{ u_i v_i = x_i } K_i (u_i) L_i(v_i),
$$
where
$|K_i(u_i)| \ll X^{\varepsilon}$, $\supp K_i \subseteq [\mathfrak{c}_1 U_i, \mathfrak{c}_2 U_i]$ with $U_i \geq 1$,
and $L_i(v_i)$ is a smooth function of the form (\ref{defnL}), $\supp L_i \subseteq [\mathfrak{c}_1 V_i, \mathfrak{c}_2 V_i]$, with $V_i \gg X^{\frac12 + \sigma_0}$, $V_i \asymp X / U_i$ and
$\mathfrak{c}_2 > \mathfrak{c}_1 > 0$ are constants that depend only on $\mathfrak{L}$ and $\Theta$.

\item[(Type II)] We have
$$
\Phi_{\ell_i; \mathbf{M}^{(i)}, \mathbf{N}^{(i)}   } (x_i) = \sum_{ u_i v_i = x_i } K_i (u_i) L_i(v_i),
$$
where $|K_i(u_i)| \ll X^{\varepsilon}$, $\supp K_i \subseteq [\mathfrak{c}_1 U_i, \mathfrak{c}_2 U_i]$ with $U_i \gg X^{\frac12 - \sigma_0}$,
and $|L_i(v_i)| \ll X^{\varepsilon}$,  $\supp L_i \subseteq [\mathfrak{c}_1 V_i, \mathfrak{c}_2 V_i]$  with $V_i \asymp X / U_i$,  $V_i \geq U_i \geq 1$
and $\mathfrak{c}_2 > \mathfrak{c}_1 > 0$ are constants that depend only on $\mathfrak{L}$ and $\Theta$.
\end{itemize}
\end{lem}
\begin{proof}
Let
$$
X^{t_1} =  M^{(i)}_1, \ldots, X^{t_{\ell_i}} =  M^{(i)}_{\ell_i},
$$
and
$$
X^{t_{\ell_i + 1}} =  N^{(i)}_1, \ldots, X^{ t_{2 \ell_i}} =  N^{(i)}_{\ell_i}.
$$
We recall from (\ref{DDDDD}) that
$$
M^{(i)}_1, \ldots, M^{(i)}_{\ell_i}, N^{(i)}_1, \ldots,  N^{(i)}_{\ell_i} \geq 1.
$$
Then it follows from (\ref{cond2}) that there exists $\eta = O(1/\log X)$ such that
$$
t_1 + \cdots + t_{2 \ell_i - 1}  +  t_{2 \ell_i} + \eta  = 1.
$$
We now apply  Lemma \ref{comblem} with  $R = 2 \ell_i$ (with $t_{i_0} + \eta > 0$ in place of $t_{i_0}$ for some $1 \leq i_0 \leq 2 \ell_i$).
Suppose we are in Type I of Lemma \ref{comblem}, i.e.  there exists $j$ such that
$$
t_j \geq  \frac12 + \sigma_0 + O \left( \frac{1}{\log X} \right) > 2 \sigma_0 > \frac{1}{\mathfrak{L}}.
$$
Then it follows from (\ref{cond1}) that
$$
V_i = N^{(i)}_{j} \gg X^{\frac{1}{2} + \sigma_0}
$$
holds for some $1 \leq j \leq \ell_i$. Therefore, by setting $v_i = n_j$, we obtain
\begin{eqnarray}
\notag
\Phi_{\ell_i; \mathbf{M}^{(i)}, \mathbf{N}^{(i)}   } (x_i) &=&
\notag
(-1)^{\ell_i +1} {\mathfrak{L} \choose \ell_i} \sum_{ \substack{ m_1 \cdots m_{\ell_i}  n_1 \cdots n_{\ell_i} = x \\ m_1, \ldots, m_{\ell_i} \leq z }} \mu(m_1) \Psi_{M^{(i)}_1}(m_1)  \cdots \mu(m_{\ell_i}) \Psi_{M^{(i)}_{\ell_i}}(m_{\ell_i})
\\
\notag
&&   \Psi_{N^{(i)}_1}(n_1)  \cdots \Psi_{N^{(i)}_{\ell_i}}(n_{\ell_i})  (\log n_{\ell_i})
\\
\notag
&=&
\sum_{ u_i v_i = x_i } K_i (u_i) L_i(v_i),
\end{eqnarray}
where
\begin{eqnarray}
\notag
K_i(u_i) &=&
(-1)^{\ell_i +1} {\mathfrak{L} \choose \ell_i} \sum_{ \substack{ m_1 \cdots m_{\ell_i}  \prod_{ s \neq j } n_s = u_i \\ m_1, \ldots, m_{\ell_i} \leq z }} \mu(m_1) \Psi_{M^{(i)}_1}(m_1)  \cdots \mu(m_{\ell_i}) \Psi_{M^{(i)}_{\ell_i}}(m_{\ell_i}) \cdot
\\
\notag
&&
\prod_{ \substack{ 1 \leq s \leq \ell_i \\ s \neq j } } \Psi_{N^{(i)}_s}(n_s) (\log n_s )^{ \mathbbm{1}_{ \{\ell_i\}} (s)   }
\end{eqnarray}
and
\begin{eqnarray}
\label{defnL}
L_i(v_i) = \begin{cases}
           \Psi_{N^{(i)}_j}(v_i) & \mbox{if } 1 \leq j < \ell_i,  \\
           (\log v_i) \Psi_{N^{(i)}_{\ell_i}}(v_i) & \mbox{if } j = \ell_i.
         \end{cases}
\end{eqnarray}
It then follows from (\ref{supp PSI}) that
$L_i(v_i) \neq 0$ only if $v_i \in [\Theta^{-1} V_i, \Theta V_i]$ and
$K_i(u_i) \neq 0$ only if $u_i \in [\Theta^{- (2 \ell_i -1 ) } U_i, \Theta^{2 \ell_i -1} U_i]$ with
$$
U_i = \frac{ M^{(i)}_1 \cdots  M^{(i)}_{\ell_i}  N^{(i)}_1 \cdots  N^{(i)}_{\ell_i}}{N^{(i)}_{j}} = \frac{ M^{(i)}_1 \cdots  M^{(i)}_{\ell_i}  N^{(i)}_1 \cdots  N^{(i)}_{\ell_i}}{ V_i } \geq 1.
$$
It then follows from (\ref{cond2}) that $U_i V_i \asymp X$.
Finally, it follows from the well-known estimate for the divisor function that $|K_i(u_i)| \ll X^{\varepsilon}$.

On the other hand, suppose we are in Type II of Lemma \ref{comblem}. Let us  set
$$
U_i = \prod_{ \substack{ t_{j_1} \in \mathcal{S}  \\  1 \leq j_1 \leq \ell_i    }  } M^{(i)}_{j_1}  \cdot
\prod_{ \substack{ t_{j_2 + \ell_i} \in \mathcal{S}  \\  1 \leq j_2 \leq \ell_i    }  } N^{(i)}_{j_2} \geq 1,
\quad
V_i = \prod_{ \substack{ t_{j_1} \in \mathcal{T}  \\  1 \leq j_1 \leq \ell_i    }  } M^{(i)}_{j_1}  \cdot
\prod_{ \substack{ t_{j_2 + \ell_i} \in \mathcal{T}  \\  1 \leq j_2 \leq \ell_i    }  } N^{(i)}_{j_2} \geq 1
$$
and
$$
u_i = \prod_{ \substack{ t_{j_1} \in \mathcal{S}  \\  1 \leq j_1 \leq \ell_i    }  } m_{j_1}  \cdot
\prod_{ \substack{ t_{j_2 + \ell_i} \in \mathcal{S}  \\  1 \leq j_2 \leq \ell_i    }  } n_{j_2},
\quad
v_i = \prod_{ \substack{ t_{j_1} \in \mathcal{T}  \\  1 \leq j_1 \leq \ell_i    }  } m_{j_1}  \cdot
\prod_{ \substack{ t_{j_2 + \ell_i} \in \mathcal{T}  \\  1 \leq j_2 \leq \ell_i    }  } n_{j_2},
$$
where we switch the roles of $\mathcal{S}$ and $\mathcal{T}$ if $U_i > V_i$ to ensure $U_i \leq V_i$. Then we obtain
\begin{eqnarray}
\notag
\Phi_{\ell_i; \mathbf{M}^{(i)}, \mathbf{N}^{(i)}   } (x_i) &=&
\notag
(-1)^{\ell_i +1} {\mathfrak{L} \choose \ell_i} \sum_{ \substack{ m_1 \cdots m_{\ell_i}  n_1 \cdots n_{\ell_i} = x \\ m_1, \ldots, m_{\ell_i} \leq z }} \mu(m_1) \Psi_{M^{(i)}_1}(m_1)  \cdots \mu(m_{\ell_i}) \Psi_{M^{(i)}_{\ell_i}}(m_{\ell_i})
\\
\notag
&&  \Psi_{N^{(i)}_1}(n_1)  \cdots \Psi_{N^{(i)}_{\ell_i}}(n_{\ell_i})  (\log n_{\ell_i})
\\
\notag
&=&
\sum_{ u_i v_i = x_i } K_i (u_i) L_i(v_i)
\end{eqnarray}
by collecting the variables $u_i$ and $v_i$ as above, and defining $K_i (u_i)$ and $L_i(v_i)$ accordingly.
It then follows from (\ref{supp PSI}) that
$L_i(v_i) \neq 0$ only if $v_i \in [\Theta^{- \# \mathcal{T}} V_i, \Theta^{\# \mathcal{T}} V_i]$ and
$K_i(u_i) \neq 0$ only if $u_i \in [\Theta^{- \# \mathcal{S}} U_i, \Theta^{\# \mathcal{S}} U_i]$.
The remaining details follow in a similar manner as the previous case.
\end{proof}
From this point on we set
$$
\sigma_0 = \frac12 - \Delta,  \quad \Delta = \frac{1}{33 d} \quad  \textnormal{ and }  \quad \mathfrak{L} = \lceil  (2 \sigma_0)^{-1}  \rceil + 1.
$$
Let us further suppose that $F$ satisfies (II) of Definition \ref{dichotomy}; we shall deal with the case when $F$ satisfies (I) of Definition \ref{dichotomy}
in Section \ref{conc}.
By substituting (\ref{big mess for nu}) into (\ref{S1111}), we obtain
\begin{eqnarray}
\label{S111111}
|S_{ \boldsymbol{\ell} } (\alpha)|
&\leq&
\sum_{ (\mathbf{M}^{(1)}, \mathbf{N}^{(1)} )  \in \mathfrak{D}^{2 \ell_1} }
\cdots
\sum_{ (\mathbf{M}^{(n)}, \mathbf{N}^{(n)} )  \in \mathfrak{D}^{2 \ell_n} }
| S_{ \boldsymbol{\ell};  \underline{ \mathbf{M} },\underline{ \mathbf{N} } }(\alpha)|
\\
\notag
&=&
\sum_{ (\mathbf{M}^{(1)}, \mathbf{N}^{(1)} )  \in \Xi_1 }
\cdots
\sum_{ (\mathbf{M}^{(n)}, \mathbf{N}^{(n)} )  \in \Xi_n } | S_{ \boldsymbol{\ell};  \underline{ \mathbf{M} },\underline{ \mathbf{N} } }(\alpha)|,
\end{eqnarray}
where
$$
S_{ \boldsymbol{\ell};  \underline{ \mathbf{M} },\underline{ \mathbf{N} } } (\alpha)
= \sum_{\mathbf{x} \in [0, X]^n } \prod_{1 \leq i \leq n}
\Phi_{\ell_i; \mathbf{M}^{(i)}, \mathbf{N}^{(i)}   } (x_i)
\psi_i(x_i)   \cdot e( \alpha F(\mathbf{x})).
$$
Our estimates for $S_{ \boldsymbol{\ell};  \underline{ \mathbf{M} },\underline{ \mathbf{N} } } (\alpha)$ in the next sections will be independent of
$\underline{\mathbf{M}}$ and $\underline{\mathbf{N}}$. Let us fix
$$
(\mathbf{M}^{(1)}, \mathbf{N}^{(1)} )  \in \Xi_1,
\ldots,  (\mathbf{M}^{(n)}, \mathbf{N}^{(n)} )  \in \Xi_n.
$$
Let $H \in \NN$ to be chosen later and set
\begin{eqnarray}
\label{defn delta}
\sigma = \frac{\sigma_0}{H} = \frac{ \frac12 - \Delta }{H} = \frac{ \frac12 - \frac{1}{33 d} }{H}.
\end{eqnarray}
We now consider the  partition of $\mathbf{x} = (x_1, \ldots, x_{n})$ into $H+1$ sets $\mathbf{x}  = (\mathbf{y}_0, \ldots, \mathbf{y}_{H})$,
where $\mathbf{y}_H$ is the collection of $x_i$ such that $i$ satisfies Type I of Lemma \ref{NU}
and for each $0 \leq j < H$, $\mathbf{y}_j$ is the collection of $x_i$ such that $i$ satisfies Type II and
$$
X^{\frac12 - \sigma_0 + j \sigma } \ll  U_i  \ll X^{ \frac12 - \sigma_0 + (j+1) \sigma }.
$$
Then by Lemma \ref{lemma rank concn} there exists $j_0 \in \{ 0, 1, \ldots, H  \}$ satisfying
\begin{eqnarray}
\label{codimcondF4-0}
\textnormal{codim} \tsp V^*_{F_{j_0}} \geq   \frac{ \textnormal{codim} \tsp V^*_{F}   - H  \mathcal{C}_0 }{H+1}.
\end{eqnarray}
Without loss of generality let $\mathbf{y}_{j_0} = (x_1, \ldots, x_{h})$.
By estimating every other variable trivially, we obtain
\begin{eqnarray}
\label{reduced sum 1}
&& S_{ \boldsymbol{\ell};  \underline{ \mathbf{M} },\underline{ \mathbf{N} } } (\alpha)
\\
\notag
&\ll&
X^{\varepsilon}
\sum_{ \substack{ \mathfrak{c}_1 U_i  \leq  u_i  \leq \mathfrak{c}_2 U_i  \\ \mathfrak{c}_1 V_i  \leq  v_i  \leq \mathfrak{c}_2 V_i \\ (h + 1 \leq i \leq n) }}
\Big{|}
\sum_{ \substack{ \mathfrak{c}_1 U_i  \leq u_i  \leq \mathfrak{c}_2 U_i  \\ \mathfrak{c}_1  V_i  \leq  v_i  \leq \mathfrak{c}_2 V_i \\ (1 \leq i \leq h) }}
\prod_{1 \leq i \leq h} K_i(u_i) L_i(v_i) \psi_i(u_i v_i)
\cdot e(\alpha F (u_1 v_1, \ldots, u_n v_n ) ) \Big{|},
\end{eqnarray}
where $K_i$ and $L_i$ are as in Type I of Lemma \ref{NU} if $j_0 = H$ and as in Type II if $0 \leq j_0 \leq H - 1$.
We prove in Section \ref{minor1} for $j_0= H$ (Type I) and in Section \ref{minor2} for $0 \leq j_0 \leq H - 1$ (Type II) that
\begin{eqnarray}
\label{min arc est}
\int_{[0,1] \setminus \mathfrak{M}^+(\frac{1}{16} - \varepsilon; \frac{1}{12})} |S_{ \boldsymbol{\ell};  \underline{ \mathbf{M} },\underline{ \mathbf{N} } } (\alpha)| \tsp d \alpha
\ll X^{n - d - \varepsilon}.
\end{eqnarray}
Then, by combining this estimate with (\ref{S11}) and  (\ref{S111111}), we obtain the following proposition.

\begin{prop}
\label{prop minor}
Suppose $F$ is as in the statement of Theorem \ref{mainthm}.
Furthermore, suppose $F$ satisfies (II) of Definition \ref{dichotomy}.
Then
$$
\int_{[0,1] \setminus \mathfrak{M}^+(\frac{1}{16} - \varepsilon; \frac{1}{12})} |S(\alpha)| \tsp d \alpha
\ll X^{n - d - \varepsilon}.
$$
\end{prop}

\section{Minor Arcs Type I}
\label{minor1}

Our goal in this section is to prove (\ref{min arc est}) which we recall is
$$
\int_{[0,1] \setminus \mathfrak{M}^+(\frac{1}{16} - \varepsilon; \frac{1}{12})} |S_{ \boldsymbol{\ell};  \underline{ \mathbf{M} },\underline{ \mathbf{N} } } (\alpha)| \tsp d \alpha
\ll X^{n - d - \varepsilon}
$$
when $j_0 = H$ in (\ref{reduced sum 1}).
First let us rewrite (\ref{reduced sum 1}) as
\begin{eqnarray}
\label{reduced sum 1-}
S_{ \boldsymbol{\ell};  \underline{ \mathbf{M} },\underline{ \mathbf{N} } } (\alpha) \ll
X^{\varepsilon}
\sum_{ \substack{ \mathfrak{c}_1 U_i  \leq  u_i  \leq \mathfrak{c}_2 U_i  \\ \mathfrak{c}_1 V_i  \leq  v_i  \leq \mathfrak{c}_2 V_i \\ (h + 1 \leq i \leq n) }}
\sum_{ \substack{ \mathfrak{c}_1 U_i  \leq u_i  \leq \mathfrak{c}_2 U_i  \\  (1 \leq i \leq h) }}
|E_\u(\alpha)|,
\end{eqnarray}
where
$$
E_\u(\alpha) =
\sum_{ \substack{ \mathfrak{c}_1  V_i  \leq  v_i  \leq \mathfrak{c}_2 V_i \\ (1 \leq i \leq h) }}
\prod_{1 \leq i \leq h} L_i(v_i) \psi_i(u_i v_i)
\cdot e(\alpha G_{\u} ( \v ) ),
$$
$\u = (u_1, \ldots, u_{h})$, $\v = (v_1, \ldots, v_{h})$,
\begin{eqnarray}
\notag
G_{\u} ( \v ) = F (u_1 v_1, \ldots, u_{h} v_{h}, \widetilde{\x})
\end{eqnarray}
and
\begin{equation}
\notag
\widetilde{\x} = (u_{h+1} v_{h+1}, \ldots, u_n v_n).
\end{equation}
We recall that $K_i$ and $L_i$ are as in Type I of Lemma \ref{NU}. By following through the proof of \cite[Lemma 2.5]{SS}
with smooth weight $\prod_{1 \leq i \leq h} L_i(v_i) \psi_i(u_i v_i)$, we obtain the following lemma (see also the proof of \cite[Lemma 3.5]{Y2}).
\begin{lem}\cite[Lemma 2.5]{SS}
\label{SSSS}
Let $0 < \vartheta < 1$.
Let $\varepsilon > 0$ be sufficiently small and $V = \max_{1 \leq j \leq h} V_j$.
Then for $X$ sufficiently large, at least one of the following alternatives holds:

\textnormal{i)} One has the upper bound
\begin{eqnarray}
| E_\u (\alpha) | \ll
\notag
V_1 \cdots V_h X^{\varepsilon} (X^{\frac12 - \sigma_0}  V^{- \vartheta})^{ \frac{\textnormal{codim} \tsp V^*_{ G_\u^{[d]} } }{2^{d -1}} },
\end{eqnarray}
where the implicit constant is independent of $\u$ and $\widetilde{\x}$.

\textnormal{ii)}
There exist $1 \leq q \leq  X^{ (\frac{1}{2} - \sigma_0) d } V^{  \vartheta  (d-1) }$ and $a \in \mathbb{Z}$ with
$\gcd( {a}, q) = 1$ such that
\begin{eqnarray}
| q \alpha - a | \leq X^{ \frac{1}{2} - \sigma_0}  V^{- d  + \vartheta (d-1) }.
\notag
\end{eqnarray}

\end{lem}
We define $\mathfrak{M}(\vartheta)$ to be the set of $\alpha \in [0,1)$
satisfying ii) in Lemma \ref{SSSS}, i.e. there exist $1 \leq q \leq X^{ (\frac{1}{2} - \sigma_0) d } X^{  \vartheta (d-1) (1 + \varepsilon)}$ and $a \in \mathbb{Z}$ with
$\gcd( {a}, q) = 1$ such that
\begin{eqnarray}
| q \alpha - a | \leq X^{ \frac{1}{2} - \sigma_0}  X^{ (- d  + \vartheta (d-1)) (\frac{1}{2} + \sigma_0 - \varepsilon)}.
\notag
\end{eqnarray}

Recall $\sigma_0 = \frac12 - \Delta$ with $\Delta = \frac{1}{33 d}$.
We shall set
$$
\vartheta =   \frac{1}{32 d}
$$
so that
$$
\left( \frac{1}{2} - \sigma_0 \right) d  +  \vartheta  (d-1) =  \Delta d +  \vartheta (d-1) <  \frac{1}{16}
$$
and
\begin{eqnarray}
\notag
\frac{1}{2} - \sigma_0 - ( d  - \vartheta (d-1)) \left(\frac{1}{2} + \sigma_0\right)  &=&  \Delta +  \vartheta  (1 - \Delta)  (d-1) - d (1 - \Delta)
\\
\notag
&<&
\Delta +  \vartheta (d-1) + \frac{1}{33} - d
\\
\notag
&<& \frac{1}{16} + \frac{1}{33} - d
\\
\notag
&<& \frac{1}{16} + \frac{1}{12} - d.
\end{eqnarray}
It then follows that
$$
\mathfrak{M}(\vartheta) \subseteq  \mathfrak{M}_+ (1/16 - \varepsilon; 1/12).
$$
In particular, we have
$$
\mathfrak{m}^+ =  [0,1) \setminus  \mathfrak{M}^+(1/16 - \varepsilon; 1/12) \subseteq  [0,1) \setminus  \mathfrak{M}(\vartheta).
$$

\subsection{Sliding scale argument}
We begin by setting
$$
\mathcal{K}_0 = \frac{\textnormal{codim} \tsp V^{*}_{  G_\u^{[d]} }  }{2^{d - 1} }.
$$
For any $0 < \vartheta' < 1$,
the Lebesgue measure of $\mathfrak{M}( \vartheta')$ is bounded by the following quantity
\begin{eqnarray}
\textnormal{meas}(\mathfrak{M}( \vartheta') )
&\ll& \sum_{1 \leq  q \leq  X^{  (\frac{1}{2} - \sigma_0)  d } X^{  \vartheta' (d-1) (1 + \varepsilon) }  }
\sum_{ \substack{  0 \leq a \leq q  \\  \gcd(a, q) = 1 } } q^{-1}  X^{ \frac{1}{2} - \sigma_0}  X^{ (- d  + \vartheta' (d-1)) (\frac{1}{2} + \sigma_0) + \varepsilon }.
\label{Lm1++}
\\
\notag
&\ll&X^{- 2 d \sigma_0 + \frac{1}{2} - \sigma_0    +   \vartheta'  (d-1)   ( \frac{3}{2} +  \sigma_0)  + \varepsilon}.
\end{eqnarray}
We then define a sequence
$$
0 < \vartheta = \vartheta_0 < \vartheta_1 < \cdots < \vartheta_J < 1
$$
satisfying
$$
\frac{\tau}{2} >  (d-1)  \left(  \frac{3}{2}  +  \sigma_0 \right) (\vartheta_{j+1} - \vartheta_{j})
\quad (0 \leq j \leq J-1)
$$
and
$$
1 - \tau < \vartheta_J < 1,
$$
where $\tau > 0$ is sufficiently small. In particular, $J \ll 1$.
Then, under the assumption that
\begin{eqnarray}
\label{K00000000000000}
\mathcal{K}_0 >    4 (d-1)    > \frac{(d-1) (  2 - \Delta  )  }{1 - \Delta}
= \frac{(d-1) \left(  \frac32 + \sigma_0  \right)  }{\frac12 + \sigma_0}
\end{eqnarray}
holds, by Lemma \ref{SSSS} and (\ref{Lm1++}) we have
\begin{eqnarray}
\label{minor arc 1 bdd1}
&&\int_{[0,1) \setminus  \mathfrak{M}( \vartheta) }  |E_\u (\alpha )| \tsp {d} {\alpha}
\\
\notag
&=&
\int_{[0,1) \setminus  \mathfrak{M} ( \vartheta_J) } |E_\u (\alpha )|  \tsp {d} {\alpha}
+
\sum_{j = 0}^{J-1}  \int_{\mathfrak{M} ( \vartheta_{j+1}) \setminus  \mathfrak{M}( \vartheta_{j}) }  |E_\u (\alpha )|  \tsp {d} {\alpha}
\\
\notag
&\ll&
V_1 \cdots V_h  X^{-  2 \sigma_0 \mathcal{K}_0  + \varepsilon}
+
V_1 \cdots V_h X^{ (\frac{1}{2} -  \sigma_0) ( \mathcal{K}_0 + 1 ) -  2 d  \sigma_0 + \varepsilon}  \sum_{j = 0}^{J-1}
X^{   \vartheta_{j+1}   (d-1)   (\frac{3}{2} +  \sigma_0)   - \vartheta_j  (\frac12 + \sigma_0) \mathcal{K}_0  }
\\
\notag
&=&
V_1 \cdots V_h  X^{- (1 - 2 \Delta) \mathcal{K}_0  + \varepsilon}
+
V_1 \cdots V_h X^{ \Delta ( \mathcal{K}_0 + 1 ) -  2 d  \sigma_0 + \varepsilon}
X^{   \vartheta   (d-1)   (\frac{3}{2} +  \sigma_0)   - \vartheta  (\frac12 + \sigma_0) \mathcal{K}_0  }
\\
&\ll&
V_1 \cdots V_h  X^{- (1 - 2 \Delta) \mathcal{K}_0  + \varepsilon}
+
V_1 \cdots V_h X^{ \Delta ( \mathcal{K}_0 + 1 ) -   d (1 - 2 \Delta) + \varepsilon}
X^{   \vartheta   (d-1)   (2 - \Delta)   - \vartheta  (1 - \Delta) \mathcal{K}_0  }.
\notag
\end{eqnarray}
First it can be verified that
\begin{eqnarray}
-d >  \Delta ( \mathcal{K}_0 + 1 ) -   d (1 - 2 \Delta) +   \vartheta   (d-1)   (2 - \Delta)   - \vartheta  (1 - \Delta) \mathcal{K}_0
\label{-dddd}
\end{eqnarray}
is equivalent to
\begin{eqnarray}
\mathcal{K}_0 >  \frac{2 d \Delta + \vartheta (d-1) (2 - \Delta ) + \Delta }{ \vartheta (1 - \Delta ) - \Delta  }.
\notag
\end{eqnarray}
We note that
\begin{eqnarray}
\vartheta  (1 -  \Delta)   -  \Delta &=& \frac{ 1 }{ 32 d  }  \left( 1  - \frac{1}{ 33 d }  \right)  - \frac{1}{33 d}
\notag
=
\frac{ d - 1 }{ (32 d) (33d)  }.
\end{eqnarray}
Therefore, we see that (\ref{K00000000000000}) and  (\ref{-dddd}) hold provided
\begin{eqnarray}
\label{K0bdddd}
\mathcal{K}_0  > \frac{3 (32 d) (33d)}{d-1}.
\end{eqnarray}
Furthermore, this condition implies that
$$
- (1 - 2 \Delta) \mathcal{K}_0 < - d.
$$
Therefore, it follows that
$$
\int_{ \mathfrak{m}^+  }  |E_\u (\alpha )|  \tsp d \alpha
\leq   \int_{[0,1) \setminus  \mathfrak{M}( \vartheta) }  |E_\u (\alpha )| \tsp {d} {\alpha} \ll V_1 \cdots V_h X^{ - d - \varepsilon}
$$
holds provided we have (\ref{K0bdddd}). By combining this estimate with (\ref{reduced sum 1-}),  we obtain
\begin{eqnarray}
\label{min est -1}
\int_{ \mathfrak{m}^+  }  |S_ {\boldsymbol{\ell}; \underline{\mathbf{M}}, \underline{\mathbf{N}} } ( \alpha )|  \tsp d \alpha
\ll X^{n - d - \varepsilon}
\end{eqnarray}
as desired. As we shall see in Section \ref{minor2}, we choose $H = 6 d$.
Since
$$
\textnormal{codim} \tsp V^{*}_{  G_\u^{[d]} } = \textnormal{codim} \tsp V^*_{ F_{H} },
$$
it follows from (\ref{codimcondF4-0}) that
\begin{eqnarray}
\notag
\textnormal{codim} \tsp V^{*}_{  G_\u^{[d]} } \geq   \frac{ \textnormal{codim} \tsp V^*_{F}   -  6 d   \mathcal{C}_0 }{6d + 1}.
\end{eqnarray}
Therefore, we see that (\ref{min est -1}) holds provided we have
$$
\textnormal{codim} \tsp V^*_{F}  >  \frac{3 (32 d) (33d) (6d + 1)}{d-1} 2^{d-1} +  6 d   \mathcal{C}_0;
$$
it can be verified that (\ref{codim of F}) implies this inequality.

\section{Minor Arcs Type II}
\label{minor2}

Our goal in this section is to prove (\ref{min arc est}) which we recall is
$$
\int_{[0,1] \setminus \mathfrak{M}^+(\frac{1}{16} - \varepsilon ; \frac{1}{12})} |S_{ \boldsymbol{\ell};  \underline{ \mathbf{M} },\underline{ \mathbf{N} } } (\alpha)| \tsp d \alpha
\ll X^{n - d - \varepsilon}
$$
when $0 \leq j_0 \leq H - 1$ in (\ref{reduced sum 1}).
First let us rewrite (\ref{reduced sum 1}) as
\begin{eqnarray}
\label{reduced sum 1+}
&& S_{ \boldsymbol{\ell};  \underline{ \mathbf{M} },\underline{ \mathbf{N} } } (\alpha)
\\
\notag
&\ll&
X^{\varepsilon}
\sum_{ \substack{ \mathfrak{c}_1 U_i  \leq  u_i  \leq \mathfrak{c}_2 U_i  \\ \mathfrak{c}_1 V_i  \leq  v_i  \leq \mathfrak{c}_2 V_i \\ (h + 1 \leq i \leq n) }}
\Big{|}
\sum_{ \substack{ \mathfrak{c}_1 U_i  \leq u_i  \leq \mathfrak{c}_2 U_i  \\ \mathfrak{c}_1  V_i  \leq  v_i  \leq \mathfrak{c}_2 V_i \\ (1 \leq i \leq h) }}
\prod_{1 \leq i \leq h} K_i(u_i) L_i(v_i) \psi_i(u_i v_i)
\cdot e(\alpha G (\u, \v ) ) \Big{|},
\end{eqnarray}
where $\u = (u_1, \ldots, u_{h})$, $\v = (v_1, \ldots, v_{h})$,
\begin{eqnarray}
\notag
G (\u, \v ) = F (u_1 v_1, \ldots, u_{h} v_{h}, \widetilde{\x})
\end{eqnarray}
and
\begin{equation}
\label{defnxtil}
\widetilde{\x} = (u_{h+1} v_{h+1}, \ldots, u_n v_n).
\end{equation}
In particular, the coefficients of the lower degree terms of $G$ may depend on $\widetilde{\x}$. The degree $2d$ homogeneous portion of $G(\mathbf{u}, \mathbf{v})$ is
$G^{[2d]}(\mathbf{u};\mathbf{v})  = F_{j_0}(u_1v_1, \ldots, u_{h}v_h)$.
With this set-up it follows from Proposition \ref{semiprimeprop} and (\ref{codimcondF4-0}) that
\begin{eqnarray}
\label{codim dj0}
\textnormal{codim} \tsp V^*_{ G^{[2d]}, 2}  \geq
\textnormal{codim} \tsp V^*_{F_{j_0}} - \mathcal{C}_0 \geq   \frac{ \textnormal{codim} \tsp V^*_F - (2H+1)  \mathcal{C}_0 }{H+1}.
\end{eqnarray}
We also recall that $K_i$ and $L_i$ are as in Type II of Lemma \ref{NU}.

Next step is a technical improvement that simplifies the argument by the author in \cite{Y2}.
Let
\begin{equation}
\label{VV}
U_{\min} = \min_{1 \leq j \leq h} U_j, \quad
W_{i; s_i} = \mathfrak{c}_1 U_i +  s_i \mathfrak{c}_2 U_{\min},
\end{equation}
\begin{equation}
\notag
V_{\min} = \min_{1 \leq j \leq h} V_j,
\quad W'_{i; s'_i} = \mathfrak{c}_1 V_i +  s'_i \mathfrak{c}_2 V_{\min}
\end{equation}
for each $0 \leq s_i \leq \frac{U_i}{U_{\min}}$, $0 \leq s'_i \leq \frac{V_i}{V_{\min}}$ and $1 \leq i \leq h$.
In particular, we have
\begin{eqnarray}
\label{some estimate}
U_i V_{\min} \leq U_i V_i \ll X \quad (1 \leq i \leq h)
\end{eqnarray}
and
\begin{eqnarray}
\label{UV}
\\
\notag
X^{ \frac{1}{2} - \sigma_0 +   j_0 \sigma }  \ll   U_i  \ll  X^{ \frac{1}{2} - \sigma_0 +  (j_0 + 1) \sigma }
\quad
\textnormal{ and } \quad
X^{\frac{1}{2} + \sigma_0  - (j_0 + 1) \sigma }  \ll   V_i  \ll  X^{\frac{1}{2} + \sigma_0  - j_0 \sigma } \quad (1 \leq i \leq h).
\end{eqnarray}
We partition the range of summation in
(\ref{reduced sum 1+}) for each $1 \leq i \leq h$ as follows
\begin{eqnarray}
\notag
&& \sum_{\substack{  \mathfrak{c}_1 U_i  \leq  u_i  \leq \mathfrak{c}_2  U_i  \\   \mathfrak{c}_1 V_i  \leq  v_i  \leq \mathfrak{c}_2  V_i  }} K_i(u_i) L_i (v_i)
\psi_i ( u_i v_i )
\\
\notag
&=&
\sum_{ \substack{ 0 \leq  s_i \leq  \frac{U_i}{U_{\min}}  \\  0 \leq s'_i \leq \frac{V_i}{V_{\min}}  }  }
\sum_{\substack{ W_{i; s_i } \leq u_i < W_{i; s_i} + \mathfrak{c}_2 U_{\min}
\\
W'_{i; s'_i } \leq v_i < W'_{i; s'_i} + \mathfrak{c}_2 V_{\min}
 }} K_i(u_i)  L_i (v_i)  \mathbbm{1}_{[\mathfrak{c}_1 U_i, \mathfrak{c}_2 U_i]}(u_i) \mathbbm{1}_{[\mathfrak{c}_1 V_i, \mathfrak{c}_2 V_i]}(v_i) \psi_i ( u_i v_i ).
\end{eqnarray}
Let us denote $\s = (s_1, \ldots, s_h)$,  $\s' = (s'_1, \ldots, s'_h)$,
$U = \mathfrak{c}_2 U_{\min}$, $V = \mathfrak{c}_2 V_{\min}$ and
\begin{eqnarray}
\notag
&&\mathcal{E}_{ \widetilde{\x} ; \s, \s' }(\alpha)
\\
&=&
\notag
\sum_{ \substack{ W_{i; s_i } \leq u_i < W_{i; s_i} + U  \\   W'_{i; s'_i } \leq v_i < W'_{i; s'_i} +  V   \\ (1 \leq i \leq h) }}
\prod_{1 \leq i \leq h} K_i(u_i) L_i(v_i) \mathbbm{1}_{[\mathfrak{c}_1 U_i, \mathfrak{c}_2 U_i]}(u_i) \mathbbm{1}_{[\mathfrak{c}_1 V_i, \mathfrak{c}_2 V_i]}(v_i)\psi_i ( u_i v_i )
\cdot e(\alpha  G(\u, \v) ).
\end{eqnarray}
Then (\ref{reduced sum 1+}) becomes
\begin{eqnarray}
\label{case 4 ineq 1}
|S_{ \boldsymbol{\ell};  \underline{ \mathbf{M} },\underline{ \mathbf{N} } } (\alpha)|
\ll
X^{\varepsilon}
\sum_{ \substack{ \mathfrak{c}_1  U_j  \leq  u_j  \leq \mathfrak{c}_2 U_j  \\    \mathfrak{c}_1 V_j  \leq  v_j  \leq \mathfrak{c}_2  V_j \\ (h + 1 \leq j \leq n) }}
\tsp
\sum_{ \substack{   {0} \leq  s_i \leq \frac{ U_i }{ U_{\min} }  \\   {0} \leq s'_i  \leq \frac{ V_i }{ V_{\min} }  \\ (1 \leq i \leq h)  } }
|\mathcal{E}_{ \widetilde{\x}; \s, \s' }(\alpha)|.
\end{eqnarray}
Let us fix a choice of $\widetilde{\x}$ (defined in (\ref{defnxtil})), $\s$ and $\s'$ as in the above summation; our estimates will be independent of these choices. We also note that
$$
\sum_{ \substack{ \mathfrak{c}_1  U_j  \leq  u_j  \leq \mathfrak{c}_2 U_j  \\    \mathfrak{c}_1 V_j  \leq  v_j  \leq \mathfrak{c}_2  V_j \\ (h + 1 \leq j \leq n) }}
\tsp
\sum_{ \substack{   {0} \leq  s_i \leq \frac{ U_i }{ U_{\min} }  \\   {0} \leq s'_i  \leq \frac{ V_i }{ V_{\min} }  \\ (1 \leq i \leq h)  } } 1 \ll \frac{ X^{n} }{ U^h V^h  }.
$$
We define $\mathbf{w}, \w' \in \ZZ^h$
by setting
$$
w_i = \lceil W_{i; s_i } \rceil  \quad   \textnormal{ and } \quad    w'_i = \lceil  W'_{i; s'_i } \rceil
$$
for each $1 \leq i \leq h$. In particular, it follows easily from the definition that
\begin{equation}
\label{some estimate2}
w_i \ll  W_{i; s_i } \ll U_i \quad (1 \leq i \leq h).
\end{equation}
Then we have
\begin{eqnarray}
\notag
\mathcal{E}_{ \widetilde{\x} ; \s, \s' }(\alpha) =  \mathcal{E} (\alpha)  = \sum_{  \substack{ \u \in [0, U]^{h}  \\    \v  \in [0, V]^{h} }}
\mathbf{K}(\u) \mathbf{L}(\v) \boldsymbol{\psi} (\u; \v) e( \alpha  g(\u, \v)  ),
\end{eqnarray}
where
$$
g (\u, \v) = G(\u + \w, \v + \w'),
$$
$$
\mathbf{K} (\mathbf{u}) = \prod_{1 \leq i \leq h} K_i(u_i + w_i) \mathbbm{1}_{[\mathfrak{c}_1 U_i - w_i, \mathfrak{c}_2 U_i - w_i]}(u_i) \mathbbm{1}_{[0, W_{i; s_i } + U -
w_i)}(u_i),
$$
$$
\mathbf{L} (\mathbf{v})= \prod_{1 \leq i \leq h} L_i(v_i + w'_i) \mathbbm{1}_{[\mathfrak{c}_1 V_i - w'_i, \mathfrak{c}_2 V_i - w'_i]}(v_i) \mathbbm{1}_{[0, W'_{i; s'_i } + V -
w'_i)}(v_i)
$$
and
\begin{eqnarray}
\label{defn of psi2}
\boldsymbol{\psi}(\u; \v) = \prod_{1 \leq i \leq h} \psi_i ( (u_i + w_i) (v_i + w'_i)  ).
\end{eqnarray}
It is clear that $g(\u, \v)$ is a degree $2d$ polynomial in $\mathbf{u}$ and $\mathbf{v}$, and its degree $2d$ homogeneous portion is
\begin{eqnarray}
\label{ggggggggg}
g^{[2d]}(\u; \v) =  G^{[2d]}(\u; \v) =  F_{j_0}(u_1 v_1, \ldots, u_{h} v_{h}) =  F(u_1 v_1, \ldots, u_{h} v_{h}, 0, \ldots, 0).
\end{eqnarray}
Therefore, by (\ref{codim dj0}) we have
\begin{eqnarray}
\label{codim dj0+}
\textnormal{codim} \tsp V^*_{ g^{[2d]}, 2}  = \textnormal{codim} \tsp V^*_{ G^{[2d]}, 2}
\geq \frac{ \textnormal{codim} \tsp V^*_{F}  - (2H + 1) \mathcal{C}_0 }{H + 1}.
\end{eqnarray}
We now estimate $\mathcal{E}(\alpha)$ by the standard argument in \cite{Bir};
we use a slight variant of the bihomogeneous version of the argument developed by Schindler in \cite{DS}.

\subsection{Weyl Differencing and Geometry of Numbers}
\label{WD1'}
We let $\mathbf{j} = (j_1, \ldots, j_d)$ and $\mathbf{k} = (k_1, \ldots, k_d)$, and denote
\begin{eqnarray}
\label{defng2dsymm}
g^{[2d]}(\mathbf{u};\mathbf{v}) &=& \sum_{j_1 = 1}^{h} \cdots \sum_{j_d = 1}^{h}
\sum_{k_1 = 1}^{h} \cdots \sum_{k_d = 1}^{h}
G_{\mathbf{j}, \mathbf{k}} \tsp  u_{j_1} \cdots u_{j_d} \tsp v_{k_1} \cdots v_{k_d}
\\
\notag
&=& \sum_{1 \leq \mathbf{j} \leq h} \sum_{1 \leq \mathbf{k} \leq h}
G_{\mathbf{j}, \mathbf{k}} \tsp  u_{j_1} \cdots u_{j_d} \tsp v_{k_1} \cdots v_{k_d}
\end{eqnarray}
with each $G_{\mathbf{j}, \mathbf{k}} \in \mathbb{Q}$ symmetric in $(j_1, \ldots, j_d)$ and also in $(k_1, \ldots, k_d)$. Note we have $(d!)^2 G_{\mathbf{j}, \mathbf{k}} \in
\mathbb{Z}$. Also $G_{\mathbf{j}, \mathbf{k}} = 0$ unless $(j_1, \ldots, j_d)$ is a permutation of $(k_1, \ldots, k_d)$; this is because  $g^{[2d]}(\mathbf{u};\mathbf{v}) =
F (u_1v_1, \ldots, u_h v_h, 0, \ldots, 0)$.

Let us define
$$
\mathcal{U} = [0, U]^h \quad \textnormal{ and } \quad  \mathcal{V} = [0, V]^h.
$$
By H\"{o}lder's inequality we obtain
\begin{equation}
\label{ineq 1-1}
|\mathcal{E} (\alpha)|^{ 2^{d - 1} } \ll X^{\varepsilon} V^{ h( 2^{d - 1} - 1)} \sum_{ \substack{ \v  \in \mathcal{V} }}
| T_{\mathbf{v}} ({\alpha}) |^{2^{d - 1}},
\end{equation}
where
\begin{eqnarray}
\label{somesumTTT}
T_{\mathbf{v}} ({\alpha}) = \sum_{ \u  \in \mathcal{U} } \mathbf{K}(\u)  \bpsi(\u; \v) e ( \alpha g( \u, \v ) ).
\end{eqnarray}
Next we use a form of Weyl's inequality as in \cite[Lemma 11.1]{S} to bound $| T_{\mathbf{v}} (\alpha) |^{2^{d} - 1}$.
Given a subset $\mathcal{X} \subseteq \mathbb{R}^{h}$, we
denote $\mathcal{X}^D = \mathcal{X} - \mathcal{X} = \{ \mathbf{z} - \mathbf{z}' : \mathbf{z}, \mathbf{z}' \in \mathcal{X} \}$.
Also for any $\mathbf{z}_1, \ldots, \mathbf{z}_t \in \mathbb{R}^{h}$,  we let
$$
\mathcal{X} (\mathbf{z}_1, \ldots, \mathbf{z}_t) = \medcap_{\epsilon_1 \in \{0, 1\} } \cdots \medcap_{\epsilon_t \in \{0, 1\} }
(\mathcal{X} - \epsilon_1 \mathbf{z}_1 - \cdots - \epsilon_t \mathbf{z}_t).
$$
In particular, it follows that
$$
\mathcal{X} (\mathbf{z}_1, \ldots, \mathbf{z}_t) = \mathcal{X} (\mathbf{z}_1, \ldots, \mathbf{z}_{t-1} ) \medcap  (\mathcal{X} (\mathbf{z}_1, \ldots, \mathbf{z}_{t-1} ) - \mathbf{z}_t).
$$

Let
\begin{eqnarray}
\label{defnF=g}
\widetilde{\mathcal{F}}(\mathbf{u}) =  \alpha g( \mathbf{u}, \mathbf{v} )
\end{eqnarray}
and
$$
\mathcal{F}(\mathbf{u})
=
\alpha g^{[2d]}(\u; \v)
= \alpha
\sum_{1 \leq \mathbf{j} \leq h} \left( \sum_{1 \leq \mathbf{k} \leq h}
G_{\mathbf{j}, \mathbf{k}} \tsp v_{k_1} \cdots v_{k_d} \right) u_{j_1} \cdots u_{j_d}.
$$
Then, viewing these as polynomials in $\u$, we have
\begin{eqnarray}
\label{defnF[d]}
\widetilde{\mathcal{F}}^{[d]}(\mathbf{u}) =  \mathcal{F}(\mathbf{u})  +
\mathcal{G}_{\mathbf{v}} (\mathbf{u}),
\end{eqnarray}
where $\mathcal{G}_{\mathbf{v}} (\mathbf{u})$ is a polynomial of degree at most $d-1$ in $\mathbf{v}$ for a fixed $\mathbf{u}$.
For each $t \in \mathbb{N}$ we denote
\begin{eqnarray}
\label{defnFd}
\widetilde{\mathcal{F}}_{t}(\mathbf{u}_1, \ldots, \mathbf{u}_{t}) = \sum_{\epsilon_1 \in \{0, 1\} } \cdots \sum_{\epsilon_{t} \in \{0, 1\} } \,
(-1)^{\epsilon_1 + \cdots + \epsilon_{t} }
\widetilde{\mathcal{F}}( \epsilon_1 \mathbf{u}_1 + \cdots + \epsilon_{t} \mathbf{u}_{t} ),
\end{eqnarray}
and let $\widetilde{\mathcal{F}}_0$ be the zero polynomial. In particular, it follows that
\begin{eqnarray}
\label{defnFd2}
&&\widetilde{\mathcal{F}}_{t}(\mathbf{u}_1, \ldots, \mathbf{u}_{t})
\\
\notag
&=&
\widetilde{\mathcal{F}}_{t-1}(\mathbf{u}_1, \ldots, \mathbf{u}_{t-1})
-
\sum_{\epsilon_1 \in \{0, 1\} } \cdots \sum_{\epsilon_{t-1} \in \{0, 1\} } \,
(-1)^{\epsilon_1 + \cdots + \epsilon_{t-1} }
\widetilde{\mathcal{F}}( \epsilon_1 \mathbf{u}_1 + \cdots + \epsilon_{t-1} \mathbf{u}_{t-1} + \mathbf{u}_{t} )
\\
\notag
&=&
\widetilde{\mathcal{F}}_{t-1}(\mathbf{u}_1, \ldots, \mathbf{u}_{t-1})
-
\widetilde{\mathcal{F}}_{t-1}(\mathbf{u}_1, \ldots,  \mathbf{u}_{t-2}, \mathbf{u}_{t-1} + \mathbf{u}_t )
+ \widetilde{\mathcal{F}}_{t-1}(\mathbf{u}_1, \ldots, \mathbf{u}_{t-2},  \mathbf{u}_{t}).
\end{eqnarray}
We let
$$
\bpsi_{\mathbf{u}_1} (\z; \v) = \bpsi(\z ; \v) \tsp \bpsi( \u_1 + \z; \mathbf{v})
$$
and recursively define
\begin{eqnarray}
\label{star}
\bpsi_{\mathbf{u}_1, \ldots, \mathbf{u}_{t}} (\mathbf{z}; \v) &=& \bpsi_{\mathbf{u}_1, \ldots, \mathbf{u}_{t - 1}} (\mathbf{z}; \v ) \tsp  \bpsi_{\mathbf{u}_1, \ldots,
\mathbf{u}_{t - 1}} (\mathbf{u}_t + \mathbf{z}; \v)
\\
\notag
&=&
\prod_{\boldsymbol{\epsilon} \in \{ 0, 1 \}^t }  \bpsi(  \epsilon_1 \u_1 + \cdots +  \epsilon_t \u_t +    \z ; \v)
\end{eqnarray}
for $t \geq 2$. We also define $\mathbf{K}_{\u_1, \ldots, \u_t}$ in a similar manner. By following the proof of \cite[Lemma 11.1]{S}, while taking into account the weights,
we obtain
\begin{eqnarray}
\label{beforeCS}
\\
\notag
| T_{\mathbf{v}} ({\alpha}) |^{2^{t - 1} }
&\leq&  | \mathcal{U}^D |^{2^{t-1} - t} \sum_{\mathbf{u}_1 \in \mathcal{U}^D} \cdots
\sum_{\mathbf{u}_{t - 1} \in \mathcal{U}^D}
\\
&&
\notag
\Big{|}  \sum_{\mathbf{u}_t \in \mathcal{U}(\mathbf{u}_1, \ldots, \mathbf{u}_{t - 1})  }
\mathbf{K}_{\mathbf{u}_1, \ldots, \mathbf{u}_{t - 1}} (\u_t) \tsp \bpsi_{\mathbf{u}_1, \ldots, \mathbf{u}_{t - 1}} (\u_t; \v) \tsp e \left( \widetilde{\mathcal{F}}_{t}(\mathbf{u}_1, \ldots,
\mathbf{u}_t) \right) \Big{|}
\end{eqnarray}
for each $t \geq 1$. Then it follows by taking the square of the inequality (with $t = d-1$) and applying
the Cauchy-Schwarz inequality that
\begin{eqnarray}
\label{BiGMess}
\\
\notag
| T_{\mathbf{v}} ({\alpha}) |^{2^{d - 1} }
&\leq&  | \mathcal{U}^D |^{2^{d - 1} - d} \sum_{\mathbf{u}_1 \in \mathcal{U}^D} \cdots
\sum_{\mathbf{u}_{d - 2} \in \mathcal{U}^D}
\\
\notag
&& \Big{|}  \sum_{\mathbf{z} \in \mathcal{U}(\mathbf{u}_1, \ldots, \mathbf{u}_{d - 2})  }
\mathbf{K}_{\mathbf{u}_1, \ldots, \mathbf{u}_{d - 2}} (\mathbf{z})
\tsp \bpsi_{\mathbf{u}_1, \ldots, \mathbf{u}_{d - 2}} (\z; \mathbf{v}) \tsp  e \left( \widetilde{\mathcal{F}}_{d - 1}(\mathbf{u}_1, \ldots, \mathbf{u}_{d - 2}, \mathbf{z}) \right) \Big{|}^2
\\
\notag
&=&
| \mathcal{U}^D |^{2^{d - 1} - d} \sum_{\mathbf{u}_1 \in \mathcal{U}^D} \cdots
\sum_{\mathbf{u}_{d - 2} \in \mathcal{U}^D}
\\
\notag
&& \sum_{\mathbf{z}, \mathbf{z}'   \in \mathcal{U}(\mathbf{u}_1, \ldots, \mathbf{u}_{d - 2})  }
\mathbf{K}_{\mathbf{u}_1, \ldots, \mathbf{u}_{d - 2}} (\mathbf{z})  \tsp  \mathbf{K}_{\mathbf{u}_1, \ldots, \mathbf{u}_{d - 2}} (\mathbf{z}')
\tsp
\bpsi_{\mathbf{u}_1, \ldots, \mathbf{u}_{d - 2}} (\z; \mathbf{v})  \tsp  \bpsi_{\mathbf{u}_1, \ldots, \mathbf{u}_{d - 2}} (\z'; \mathbf{v})
\cdot
\\
\notag
&&
e \left( \widetilde{\mathcal{F}}_{d - 1}(\mathbf{u}_1, \ldots, \mathbf{u}_{d - 2}, \mathbf{z}) - \widetilde{\mathcal{F}}_{d - 1}(\mathbf{u}_1, \ldots, \mathbf{u}_{d - 2}, \mathbf{z}')  \right).
\end{eqnarray}
Given $\mathbf{z}, \mathbf{z}' \in \mathcal{U}(\mathbf{u}_1, \ldots, \mathbf{u}_{d - 2})$, let us set
$\mathbf{u}_{d - 1} = (\mathbf{z}' - \mathbf{z} ) \in \mathcal{U}(\mathbf{u}_1, \ldots, \mathbf{u}_{d - 2})^D$
and
$$
\mathbf{u}_{d} =  \mathbf{z} \in \mathcal{U}(\mathbf{u}_1, \ldots, \mathbf{u}_{d - 2}) \medcap (\mathcal{U}(\mathbf{u}_1, \ldots, \mathbf{u}_{d - 2}) -
\mathbf{u}_{d-1}) =  \mathcal{U}(\mathbf{u}_1, \ldots, \mathbf{u}_{d - 1}).
$$
It also follows from (\ref{defnFd2}) that
\begin{eqnarray}
&&
\widetilde{\mathcal{F}}_{d-1}(\mathbf{u}_1, \ldots, \mathbf{u}_{d-2}, \mathbf{z}) - \widetilde{\mathcal{F}}_{d-1}(\mathbf{u}_1, \ldots, \mathbf{u}_{d-2}, \mathbf{z}')
\notag
\\
&=&
\widetilde{\mathcal{F}}_{d-1}(\mathbf{u}_1, \ldots, \mathbf{u}_{d-2}, \mathbf{u}_{d}) - \widetilde{\mathcal{F}}_{d-1}(\mathbf{u}_1, \ldots, \mathbf{u}_{d-2}, \mathbf{u}_{d - 1} + \mathbf{u}_{d})
\notag
\\
&=& \widetilde{\mathcal{F}}_{d}(\mathbf{u}_1, \ldots, \mathbf{u}_{d-1}, \mathbf{u}_{d}) - \widetilde{\mathcal{F}}_{d-1}(\mathbf{u}_1, \ldots, \mathbf{u}_{d - 1}).
\notag
\end{eqnarray}
Therefore, we obtain
\begin{eqnarray}
\notag
&& \Big{|}  \sum_{\mathbf{z} \in \mathcal{U}(\mathbf{u}_1, \ldots, \mathbf{u}_{d - 2})  }
\mathbf{K}_{\mathbf{u}_1, \ldots, \mathbf{u}_{d - 2}} (\mathbf{z})
\tsp \bpsi_{\mathbf{u}_1, \ldots, \mathbf{u}_{d - 2}} (\z; \mathbf{v}) \tsp  e \left( \widetilde{\mathcal{F}}_{d - 1}(\mathbf{u}_1, \ldots, \mathbf{u}_{d - 2}, \mathbf{z}) \right) \Big{|}^2
\\
\notag
&=&  \sum_{\mathbf{u}_{d - 1} \in \mathcal{U}(\mathbf{u}_1, \ldots, \mathbf{u}_{d - 2})^D  }  \,   \sum_{\mathbf{u}_{d} \in \mathcal{U}(\mathbf{u}_1, \ldots, \mathbf{u}_{d - 1})  } \mathbf{K}_{\mathbf{u}_1, \ldots, \mathbf{u}_{d - 1}} (\mathbf{u}_{d})
\tsp  \bpsi_{\mathbf{u}_1, \ldots, \mathbf{u}_{d - 1}} (\mathbf{u}_d; \mathbf{v})
\cdot
\\
\notag
&&
e \left( \widetilde{\mathcal{F}}_{d}(\mathbf{u}_1, \ldots, \mathbf{u}_{d-1}, \mathbf{u}_{d}) - \widetilde{\mathcal{F}}_{d-1}(\mathbf{u}_1, \ldots, \mathbf{u}_{d - 1}) \right),
\end{eqnarray}
and (\ref{BiGMess}) becomes
\begin{eqnarray}
\label{ineq 1-2}
|T_{\mathbf{v}} ({\alpha}) |^{2^{d-1}}
&\leq&
| \mathcal{U}^D |^{2^{d - 1} - d}
\sum_{\mathbf{u}_1 \in \mathcal{U}^D} \cdots \sum_{\mathbf{u}_{d - 2} \in \mathcal{U}^D}
\ \sum_{\mathbf{u}_{d - 1} \in \mathcal{U}(\mathbf{u}_1, \ldots, \mathbf{u}_{d - 2})^D  }
\\
\notag
&&   \sum_{\mathbf{u}_{d} \in \mathcal{U}(\mathbf{u}_1, \ldots, \mathbf{u}_{d - 1})  }
\mathbf{K}_{\mathbf{u}_1, \ldots, \mathbf{u}_{d - 1}} (\mathbf{u}_d ) \tsp
\bpsi_{\mathbf{u}_1, \ldots, \mathbf{u}_{d - 1}} (\u_d; \v) \cdot
\\
&&e \left( \widetilde{\mathcal{F}}_{d}(\mathbf{u}_1, \ldots, \mathbf{u}_{d}) - \widetilde{\mathcal{F}}_{d-1}(\mathbf{u}_1, \ldots, \mathbf{u}_{d - 1}) \right).
\notag
\end{eqnarray}
We note that $\mathcal{U}^D$, $\mathcal{U}(\mathbf{u}_1, \ldots, \mathbf{u}_{d - 2})^D$ are boxes contained in $[-U, U]^h$,
and $\mathcal{U}(\mathbf{u}_1, \ldots, \mathbf{u}_{d - 1})$ is a box contained in $[0, U]^h$.

Let
$$
\mathcal{F}_{d}(\mathbf{u}_1, \ldots, \mathbf{u}_{d}) = \sum_{\epsilon_1 \in \{0, 1\} } \cdots \sum_{\epsilon_{d} \in \{0, 1\} } \,
(-1)^{\epsilon_1 + \cdots + \epsilon_{d} }
\mathcal{F}( \epsilon_1 \mathbf{u}_1 + \cdots + \epsilon_{d} \mathbf{u}_{d} )
$$
and
$$
(\mathcal{G}_{\v})_d(\mathbf{u}_1, \ldots, \mathbf{u}_{d}) = \sum_{\epsilon_1 \in \{0, 1\} } \cdots \sum_{\epsilon_{d} \in \{0, 1\} } \,
(-1)^{\epsilon_1 + \cdots + \epsilon_{d} }
\mathcal{G}_{\v}( \epsilon_1 \mathbf{u}_1 + \cdots + \epsilon_{d} \mathbf{u}_{d} ).
$$
By \cite[Lemma 11.4]{S}, the polynomial $\mathcal{F}_{d}$ is the unique symmetric multilinear form associated to $\mathcal{F}^{[d]}$, i.e. $\mathcal{F}_{d}$ satisfies
\begin{eqnarray}
\label{FdF[d]}
\mathcal{F}_{d}(\mathbf{u}, \ldots, \mathbf{u}) = (-1)^d \tsp  d! \tsp \mathcal{F}^{[d]}(\mathbf{u}).
\end{eqnarray}
Recalling (\ref{defnF[d]}) and making use of \cite[Lemma 11.2]{S} and \cite[Lemma 11.4]{S}, it follows that
\begin{eqnarray}
\widetilde{\mathcal{F}}_{d}(\mathbf{u}_1, \ldots, \mathbf{u}_d)
&=&
\mathcal{F}_{d}(\mathbf{u}_1, \ldots, \mathbf{u}_{d}) + (\mathcal{G}_{\v})_d(\mathbf{u}_1, \ldots, \mathbf{u}_{d})
\notag
\\
\notag
&=&
(-1)^d d!  \alpha \sum_{ 1 \leq  \mathbf{j} \leq h } \left(  \sum_{ 1 \leq   \mathbf{k}  \leq h } G_{ \mathbf{j}, \mathbf{k} } \tsp v_{k_1} \cdots v_{k_d} \right)  u_{1, j_1} \cdots u_{d, j_{d}}
+
(\mathcal{G}_{\v})_d(\mathbf{u}_1, \ldots, \mathbf{u}_{d}),
\end{eqnarray}
where
$(\mathcal{G}_{\v})_d(\mathbf{u}_1, \ldots, \mathbf{u}_{d})$ is a polynomial of degree
at most $d-1$ in $\mathbf{v}$ for fixed $\mathbf{u}_1, \ldots, \mathbf{u}_{d}$.
Let us denote
$$
- \widetilde{\mathcal{F}}_{d-1}(\mathbf{u}_1, \ldots, \mathbf{u}_{d - 1})
= \alpha \sum_{1 \leq \mathbf{k}  \leq h } v_{k_1} \cdots v_{k_{d}} \tsp
\widetilde{\mathcal{H}}_{\mathbf{k}}(\mathbf{u}_1, \ldots, \mathbf{u}_{d-1})
+
\mathfrak{H}_{\mathbf{v}}(\mathbf{u}_1, \ldots, \mathbf{u}_{d-1}),
$$
where $\mathfrak{H}_{\mathbf{v}} (\mathbf{u}_1, \ldots, \mathbf{u}_{d-1})$ is a polynomial of degree at most $d-1$ in $\mathbf{v}$ for fixed $\mathbf{u}_1, \ldots, \mathbf{u}_{d-1}$, and each
$\widetilde{\mathcal{H}}_{\mathbf{k}}(\mathbf{u}_1, \ldots, \mathbf{u}_{d-1})$ is symmetric in $(k_1, \ldots, k_d)$.
Clearly $\widetilde{\mathcal{H}}_{ \mathbf{k} }$ and $\mathfrak{H}_{\mathbf{v}}$ are polynomials of degrees at most $d$ in $(\mathbf{u}_1, \ldots, \mathbf{u}_{d-1})$.

Let us write $\underline{\mathbf{u}} = (\mathbf{u}_1, \ldots, \mathbf{u}_{d})$.
We substitute the inequality (\ref{ineq 1-2}) into (\ref{ineq 1-1}), and we interchange the order of summation moving
the sum over $\mathbf{v}$ inside the sums over $\mathbf{u}_j$.
Then we apply H\"{o}lder's inequality to obtain
\begin{eqnarray}
\label{Hod1}
\\
\notag
|\mathcal{E}(\alpha)|^{2^{2 d -2}}
\ll
X^{\varepsilon}
U^{h( 2^{2d -2} - d)} V^{ h( 2^{2d -2} - 2^{d -1} ) }
\sum_{\mathbf{u}_1 \in \mathcal{U}^D} \cdots \sum_{\mathbf{u}_{d - 1} \in \mathcal{U}^D}
\sum_{\mathbf{u}_{d} \in \mathcal{U}(\mathbf{u}_1, \ldots, \mathbf{u}_{d - 1})  } | \mathfrak{T}_{ \underline{\mathbf{u}} } ({\alpha}) |^{2^{d - 1} },
\end{eqnarray}
where
\begin{eqnarray}
&&\mathfrak{T}_{\underline{\mathbf{u}} } (\alpha)
\notag
\\
&=& \notag
\sum_{ \mathbf{v} \in \mathcal{V}} \bpsi_{\mathbf{u}_1, \ldots, \mathbf{u}_{d - 1}} (\u_d; \v) \tsp
e \left(  \alpha \sum_{ 1 \leq  \mathbf{k} \leq h}  v_{k_1} \cdots v_{k_{d}} \tsp
\mathcal{H}_{\mathbf{k}}(\underline{\mathbf{u}})  + (\mathcal{G}_{\mathbf{v}})_d (\underline{\mathbf{u}})  +   \mathfrak{H}_{\mathbf{v}}(\mathbf{u}_1, \ldots, \mathbf{u}_{d-1}) \right)
\end{eqnarray}
and
\begin{eqnarray}
\label{defHj}
\mathcal{H}_{\mathbf{k}}(\underline{\mathbf{u}}) =
\sum_{ 1 \leq  \mathbf{j} \leq h }  G_{\mathbf{j}, \mathbf{k} }
(-1)^d d! \tsp u_{1, j_1} \cdots u_{{d}, j_{d}} + \widetilde{\mathcal{H}}_{\mathbf{k}}(\mathbf{u}_1, \ldots,
\mathbf{u}_{d-1}).
\end{eqnarray}
Let us set $\widehat{\mathbf{u}} = (\mathbf{u}_1, \ldots, \mathbf{u}_{d - 1})$. Let $\bpsi_{ \widehat{\mathbf{u}} }(\mathbf{u}_d;  \mathbf{v} )
= \bpsi_{ \mathbf{u}_1, \ldots, \mathbf{u}_{d-1} }(\mathbf{u}_d;  \mathbf{v} )$.
Similarly as before we let
$$
\bpsi_{\widehat{\mathbf{u}}; \v_1 } (\u_d; \mathbf{z}) = \bpsi_{\widehat{\mathbf{u}}}(\u_d;  \mathbf{z})
\tsp \bpsi_{ \widehat{\mathbf{u}}}(\u_d; \mathbf{v}_1 + \mathbf{z})
$$
and recursively define
\begin{eqnarray}
\label{starstar}
\bpsi_{\widehat{\mathbf{u}}; \mathbf{v}_1, \ldots, \mathbf{v}_{t}} (\mathbf{u}_d; \z)
&=& \bpsi_{\widehat{\mathbf{u}}; \mathbf{v}_1, \ldots, \mathbf{v}_{t-1}} (\u_d; \mathbf{z}) \tsp
\bpsi_{\widehat{\mathbf{u}}; \mathbf{v}_1, \ldots, \mathbf{v}_{t-1}} ( \u_d; \mathbf{v}_{t} + \mathbf{z})
\\
\notag
&=&
\prod_{\boldsymbol{\epsilon}' \in \{0,1\}^t } \bpsi_{ \widehat{\mathbf{u}}}(\u_d; \epsilon_1' \mathbf{v}_1 + \cdots + \epsilon_t' \mathbf{v}_t + \mathbf{z})
\end{eqnarray}
for $t \geq 2$. We also set $\widehat{\mathbf{v}} = (\mathbf{v}_1, \ldots, \mathbf{v}_{d - 1})$ and $\underline{\mathbf{v}} = (\mathbf{v}_1, \ldots, \mathbf{v}_{d})$.
Let $\bpsi_{\widehat{\mathbf{u}} ; \widehat{\mathbf{v}} } (\mathbf{u}_d; \mathbf{z}) = \bpsi_{  \widehat{\mathbf{u}} ;  \mathbf{v}_1, \ldots, \mathbf{v}_{d - 1} }
(\mathbf{u}_d; \mathbf{z})$. We now apply the same differencing process as before to $\mathfrak{T}_{\underline{\mathbf{u}}} (\alpha)$;
this time, instead of (\ref{defnF=g}),  we use
\begin{eqnarray}
\label{defnF=g2}
\widetilde{\mathcal{F}}(\mathbf{v}) =  \alpha
 \sum_{ 1 \leq  \mathbf{k} \leq h} \mathcal{H}_{\mathbf{k}}(\underline{\mathbf{u}}) \tsp v_{k_1} \cdots v_{k_{d}}
 + (\mathcal{G}_{\mathbf{v}})_d (\underline{\mathbf{u}})  +   \mathfrak{H}_{\mathbf{v}}(\mathbf{u}_1, \ldots, \mathbf{u}_{d-1}).
\end{eqnarray}
Then
\begin{eqnarray}
\label{defnF[d]2}
\widetilde{\mathcal{F}}^{[d]}(\mathbf{v}) =  \alpha  \sum_{ 1 \leq  \mathbf{k} \leq h} \mathcal{H}_{\mathbf{k}}(\underline{\mathbf{u}}) \tsp v_{k_1} \cdots v_{k_{d}}
\end{eqnarray}
and we let
\begin{eqnarray}
\label{defnF[d]2+}
\mathcal{L} (\underline{\mathbf{u}}; \underline{\mathbf{v}})   =
\widetilde{\mathcal{F}}_{d}(\mathbf{v}_1, \ldots, \mathbf{v}_d)  =     (-1)^d  d! \alpha  \sum_{ 1 \leq  \mathbf{k} \leq h}
\mathcal{H}_{\mathbf{k}}(\underline{\mathbf{u}}) \tsp  v_{1, k_1} \cdots v_{d, k_{d}},
\end{eqnarray}
where the second equality follows by \cite[Lemma 11.4]{S}.
Therefore, by the same argument as in obtaining (\ref{beforeCS}) and substituting $t = d$ yield
\begin{eqnarray}
\notag
| \mathfrak{T}_{\underline{\mathbf{u}}} (\alpha) |^{2^{d - 1} }
\leq  | \mathcal{V}^D |^{2^{d-1} - d} \sum_{\mathbf{v}_1 \in \mathcal{V}^D} \cdots
\sum_{\mathbf{v}_{d - 1} \in \mathcal{V}^D}
\Big{|}  \sum_{\mathbf{v}_d \in \mathcal{V}(\mathbf{v}_1, \ldots, \mathbf{v}_{d - 1})  }
\bpsi_{\widehat{\mathbf{u}} ; \widehat{\mathbf{v}} } (\mathbf{u}_d ; \mathbf{v}_d)  e (  \mathcal{L} (\underline{\mathbf{u}}; \underline{\mathbf{v}})  ) \Big{|},
\end{eqnarray}
and  (\ref{Hod1}) becomes
\begin{eqnarray}
\label{ineq 2'}
|\mathcal{E} (\alpha)|^{2^{2 d -2}} &\ll&
X^{\varepsilon}
U^{ h ( 2^{2 d -2} - d ) }
V^{ h ( 2^{2 d -2} - d) }
\sum_{\mathbf{u}_1 \in \mathcal{U}^D} \cdots \sum_{\mathbf{u}_{d - 1} \in \mathcal{U}^D}
\sum_{\mathbf{u}_{d} \in \mathcal{U}(\mathbf{u}_1, \ldots, \mathbf{u}_{d - 1})  }
\\
\notag
&&  \sum_{\mathbf{v}_1 \in \mathcal{V}^D} \cdots \sum_{\mathbf{v}_{d - 1}  \in \mathcal{V}^D}
\Big{|} \sum_{\mathbf{v}_{d} \in \mathcal{V}(\mathbf{v}_1, \ldots, \mathbf{v}_{d-1})}
\bpsi_{\widehat{\mathbf{u}} ; \widehat{\mathbf{v}} } (\mathbf{u}_d ; \mathbf{v}_d) e (  \mathcal{L} (\underline{\mathbf{u}}; \underline{\mathbf{v}})  )  \Big{|}.
\end{eqnarray}
We note that $\mathcal{V}^D$
is a box contained in $[-V, V]^h$, and $\mathcal{V}(\mathbf{v}_1, \ldots, \mathbf{v}_{d - 1})$ is a box contained in
$[0, V]^h$.
We now change the order of summation in (\ref{ineq 2'}), and bound the exponential sum
\begin{eqnarray}
\label{sum 1}
&&\sum_{\mathbf{u}_{d} \in \mathcal{U}(\mathbf{u}_1, \ldots, \mathbf{u}_{d-1}) } \Big{|} \sum_{\mathbf{v}_{d} \in \mathcal{V}(\mathbf{v}_1, \ldots, \mathbf{v}_{d-1}) }
\bpsi_{ \widehat{\mathbf{u}} ; \widehat{\mathbf{v}} }  (\mathbf{u}_d ; \mathbf{v}_d)
e (  \mathcal{L} (\underline{\mathbf{u}}; \underline{\mathbf{v}})  )   \Big{|}.
\end{eqnarray}
Recall (\ref{defn of psi}), (\ref{defn of psi2}), (\ref{star}) and (\ref{starstar}).
First we have
\begin{eqnarray}
\notag
\bpsi_{ \widehat{\mathbf{u}} ; \widehat{\mathbf{v}} }  (\mathbf{u}_d ; \mathbf{v}_d)
&=&
\prod_{\boldsymbol{\epsilon} \in \{0,1\}^{d-1} } \prod_{\boldsymbol{\epsilon}' \in \{0,1\}^{d-1} }
\bpsi ( \epsilon_1 \mathbf{u}_1 + \cdots + \epsilon_{d-1} \mathbf{u}_{d-1} + \mathbf{u}_d  ; \epsilon_1' \mathbf{v}_1 + \cdots + \epsilon_{d-1}' \mathbf{v}_{d-1} + \mathbf{v}_d )
\\
\notag
&=&
\prod_{1 \leq i \leq h} \Omega_i (v_{d, i}),
\end{eqnarray}
where
\begin{eqnarray}
&& \Omega_i (z)
\notag
\\
\notag
&=&
\prod_{\boldsymbol{\epsilon} \in \{0,1\}^{d-1} } \prod_{\boldsymbol{\epsilon}' \in \{0,1\}^{d-1} }
\omega
\left(  \frac{1}{X}  \left(   \sum_{ 1 \leq s \leq d -1 } \epsilon_s u_{s, i}  + u_{d, i}  + w_i \right)
\left(   \sum_{ 1 \leq t \leq d -1 } \epsilon'_t v_{t, i}  + z  + w'_i \right) - x_{0, i} \right).
\end{eqnarray}
Let
$$
\mathcal{V}(\mathbf{v}_1, \ldots, \mathbf{v}_{d-1}) = \prod_{1 \leq i \leq h}  \mathcal{V}_i(\mathbf{v}_1, \ldots, \mathbf{v}_{d-1}) \subseteq \mathbb{R}^h,
$$
where each $\mathcal{V}_i(\mathbf{v}_1, \ldots, \mathbf{v}_{d-1})$ is an interval contained in $[0, V]$.
With these notation we have
\begin{eqnarray}
\label{sum 1+usedtobe}
&&\sum_{\mathbf{u}_{d} \in \mathcal{U}(\mathbf{u}_1, \ldots, \mathbf{u}_{d - 1})} \Big{|} \sum_{\mathbf{v}_{d} \in \mathcal{V}(\mathbf{v}_1, \ldots, \mathbf{v}_{d-1}) }
\bpsi_{\widehat{\mathbf{u}} ; \widehat{\mathbf{v}} } (\mathbf{u}_d ; \mathbf{v}_d)
 e (  \mathcal{L} (\underline{\mathbf{u}}; \underline{\mathbf{v}})  )   \Big{|}
\\
\notag
&=& \sum_{\mathbf{u}_{d} \in \mathcal{U}(\mathbf{u}_1, \ldots, \mathbf{u}_{d - 1})} \prod_{1 \leq i \leq h}   \Big{|}
\sum_{v_{d, i} \in \mathcal{V}_i(\mathbf{v}_1, \ldots, \mathbf{v}_{d-1}) }
\Omega_i    (v_{d, i} )
e  (  \mathcal{L} (\underline{\mathbf{u}};  \widehat{\mathbf{v}}, v_{d, i} \tsp \mathbf{e}_{i}   )  )  \Big{|},
\end{eqnarray}
where $\mathbf{e}_{i}$ is the $i$-th unit vector in $\RR^h$.  
Given $z \in \mathbb{R}$, let
$$
\| z \| = \min_{y \in \mathbb{Z}} |z - y|.
$$
It is clear from (\ref{defnF[d]2+}) that
$$
\mathcal{L} (\underline{\mathbf{u}};  \widehat{\mathbf{v}}, v_{d, i} \tsp \mathbf{e}_{i}   )  =   \mathcal{L} (\underline{\mathbf{u}};  \widehat{\mathbf{v}},  \mathbf{e}_{i}   ) v_{d, i}.
$$
Therefore, by partial summation we obtain
\begin{eqnarray}
\label{inequality17}
&&\Big{|}
\sum_{v_{d, i} \in \mathcal{V}_i(\mathbf{v}_1, \ldots, \mathbf{v}_{d-1}) }
\Omega_i (v_{d, i}) e  (  \mathcal{L}( \underline{\mathbf{u}};   \widehat{\mathbf{v}}, v_{d, i} \tsp \mathbf{e}_{i}  )  )  \Big{|}
\\
&\ll&
\notag
\left(  \sup_{ v_{d, i} \in [0, V] } \Omega_i (v_{d, i})
+
\int_{ 0 }^{V} \Big{|} \frac{d \Omega_i}{dz}  (v_{d, i}) \Big{|} \tsp d v_{d, i} \right)  \cdot
\min \left( V, \| \mathcal{L} ( \underline{\mathbf{u}}; \widehat{\mathbf{v}}, \mathbf{e}_{i}   )   \|^{-1}  \right).
\end{eqnarray}
Recall (\ref{some estimate}) and (\ref{some estimate2}).
Then since
$$
\left( \sum_{ 1 \leq s \leq d -1 } \epsilon_s u_{s, i}  + u_{d, i}  + w_i \right) \frac{V}{X}
\ll
\frac{(U + U_i) V}{X}
=
\frac{(\mathfrak{c}_2 U_{\min} + U_i) \mathfrak{c}_2 V_{\min}}{X}
\ll \frac{U_i V_{\min}}{X} \ll 1
$$
for any $\boldsymbol{\epsilon}$ and $\underline{\mathbf{u}}$  under consideration, we have
\begin{eqnarray}
\int_{0}^{V} \Big{|} \frac{d \Omega_i}{dz}  (v_{d, i}) \Big{|} \tsp d v_{d, i}
\ll
\frac{U + U_i}{X} \int_{0}^{V} 1 \tsp d v_{d, i} \ll 1.
\notag
\end{eqnarray}
It is easy to see that
$$
\sup_{ v_{d, i} \in [0, V] } \Omega_i (v_{d, i}) \ll 1.
$$
Therefore, it follows from (\ref{sum 1+usedtobe}) and (\ref{inequality17}) that
\begin{eqnarray}
\label{sum 1+}
&&\sum_{\mathbf{u}_{d} \in \mathcal{U}(\mathbf{u}_1, \ldots, \mathbf{u}_{d - 1})} \Big{|} \sum_{\mathbf{v}_{d} \in \mathcal{V}(\mathbf{v}_1, \ldots, \mathbf{v}_{d-1}) }
\bpsi_{ \widehat{\mathbf{u}} ; \widehat{\mathbf{v}}}   (\mathbf{u}_{d} ; \mathbf{v}_{d} )  e (  \mathcal{L} (\underline{\mathbf{u}}; \underline{\mathbf{v}})  )   \Big{|}
\\
\notag
&\ll& \sum_{\mathbf{u}_{d} \in \mathcal{U}(\mathbf{u}_1, \ldots, \mathbf{u}_{d - 1})} \prod_{1 \leq i  \leq h} \min \left( V,
\|  \mathcal{L}( \underline{\mathbf{u}}; \widehat{\mathbf{v}}, \mathbf{e}_{i} )   \|^{-1}   \right),
\end{eqnarray}
where the implicit constant is independent of $\widehat{\mathbf{u}}$ and $\widehat{\mathbf{v}}$.

For $z \in \RR$ we define its fractional part to be
$$
\{z\} = z - \max_{ \substack{y \leq z \\ y \in \mathbb{Z}}} y.
$$
Given $\mathbf{c} = (c_1, \ldots, c_{h}) \in
\mathbb{Z}^{h}$ with $0 \leq c_{i} < V$ $(1 \leq i \leq h)$,
we let $\mathcal{R}(\widehat{\mathbf{u}}; \widehat{\mathbf{v}}; \mathbf{c})$ be the set of $\mathbf{u}_{d} \in \mathcal{U}(\mathbf{u}_1, \ldots, \mathbf{u}_{d - 1})$
satisfying
$$
\frac{c_{i}}{ V } \leq \{  \mathcal{L}(\underline{\mathbf{u}}; \widehat{\mathbf{v}}, \mathbf{e}_i )   \} <
\frac{\min ( c_{i} + 1, V)  }{  V  }
\quad (1 \leq i \leq h).
$$
Then we obtain that the right hand side of (\ref{sum 1+}) is bounded by
\begin{eqnarray}
\label{sum 1++}
\\
\notag
\ll \sum_{0 \leq  c_1, \ldots, c_h < V} \# \mathcal{R}(\widehat{\mathbf{u}}; \widehat{\mathbf{v}}; \mathbf{c})  \cdot
\prod_{  1 \leq i \leq h} \min \left(  V, \ \max \left(  \frac{ V }{c_{i} }, \ \frac{ V }{ V  - \min ( c_{i} + 1, V)}  \right)   \right).
\end{eqnarray}
Next we obtain a bound for $ \# \mathcal{R}(\widehat{\mathbf{u}}; \widehat{\mathbf{v}}; \mathbf{c})$.
We define the multilinear form
\begin{eqnarray}
\label{defnMMMMMM}
\Gamma (\underline{\mathbf{u}} ; \underline{\mathbf{v}} ) = (d!)^2 \sum_{ 1 \leq  \mathbf{j} \leq h } \sum_{ 1 \leq   \mathbf{k}  \leq  h }
G_{ \mathbf{j}, \mathbf{k} } \tsp u_{1, j_1} \cdots u_{d, j_{d}} \tsp v_{1, k_1} \cdots v_{d, k_{d}}.
\end{eqnarray}
When $\mathbf{u} = \mathbf{u}_1 = \cdots = \mathbf{u}_{d}$ and $\mathbf{v} = \mathbf{v}_1 = \cdots = \mathbf{v}_{d-1}$,
we have
\begin{eqnarray}
\label{derivcond}
\Gamma( ( \mathbf{u}, \ldots , \mathbf{u}) ; ( \mathbf{v}, \ldots , \mathbf{v},  \mathbf{e}_{i})   ) = \frac{(d!)^2}{d} \cdot  \frac{\partial g^{[2d]}}{ \partial v_{i}}
(\mathbf{u};\mathbf{v})  \quad (1 \leq i \leq h).
\end{eqnarray}
If $\#\mathcal{R}(\widehat{\mathbf{u}}; \widehat{\mathbf{v}}; \mathbf{c}) = 0$, then there is nothing to prove. Thus we suppose
$\#\mathcal{R}(\widehat{\mathbf{u}}; \widehat{\mathbf{v}}; \mathbf{c}) > 0$ and let $\mathbf{z} \in \mathcal{R}(\widehat{\mathbf{u}}; \widehat{\mathbf{v}}; \mathbf{c})$. Then
for any $\mathbf{z}' \in \mathcal{R}(\widehat{\mathbf{u}}; \widehat{\mathbf{v}}; \mathbf{c})$,  we have
$$
\mathbf{z} - \mathbf{z}' \in \mathcal{U}(\mathbf{u}_1, \ldots, \mathbf{u}_{d - 1})^D \subseteq [- U, U]^h
$$
and
\begin{eqnarray}
\label{counting 1''}
\|  \mathcal{L} (\widehat{\mathbf{u}}, \mathbf{z} ; \widehat{\mathbf{v}}, \mathbf{e}_i  ) - \mathcal{L} (\widehat{\mathbf{u}}, \mathbf{z}' ; \widehat{\mathbf{v}},
\mathbf{e}_{i}  )   \| < V^{-1} \quad (1 \leq i \leq h).
\end{eqnarray}
Let $M( \widehat{\mathbf{u}} ; \widehat{\mathbf{v}} )$ be the number of integral vectors $\mathbf{u} \in [-U, U]^h$ such that
$$
\| \alpha \Gamma (  \widehat{\mathbf{u}}, \mathbf{u} ;  \widehat{\mathbf{v}}, \mathbf{e}_{i})  \| < V^{-1} \quad (1 \leq i \leq h).
$$
Given any $\v_d \in \RR^h$, we have
\begin{eqnarray}
&&\mathcal{L}(  \widehat{\mathbf{u}}, \mathbf{z}; \widehat{\mathbf{v}}, \mathbf{v}_d  ) -
\mathcal{L} (\widehat{\mathbf{u}}, \mathbf{z}'; \widehat{\mathbf{v}}, \mathbf{v}_d  )
\notag
\\
\notag
&=&
(-1)^d  d! \alpha  \sum_{ 1 \leq  \mathbf{k} \leq h}
(\mathcal{H}_{\mathbf{k}}(\widehat{\mathbf{u}}, \mathbf{z}) - \mathcal{H}_{\mathbf{k}}(\widehat{\mathbf{u}}, \mathbf{z}')) \tsp   v_{1, k_1} \cdots v_{d, k_{d}}
\\
\notag
&=&
(d!)^2 \alpha  \sum_{ 1 \leq  \mathbf{k} \leq h}
\sum_{ 1 \leq  \mathbf{j} \leq h } G_{\mathbf{j}, \mathbf{k} }  \tsp
u_{1, j_1} \cdots u_{d-1, j_{d-1}}  (z_{j_d} - z'_{j_d} )
\tsp  v_{1, k_1} \cdots v_{d, k_{d}}
\\
&=&
\notag
\alpha
\Gamma(  \widehat{\mathbf{u}}, \mathbf{z} - \mathbf{z}'; \widehat{\mathbf{v}}, \mathbf{v}_{d} ),
\end{eqnarray}
where we recall (\ref{defHj}) for the second equality. Thus (\ref{counting 1''}) becomes
$$
\| \alpha
\Gamma(  \widehat{\mathbf{u}}, \mathbf{z} - \mathbf{z}'; \widehat{\mathbf{v}}, \mathbf{e}_{i} )   \| < V^{-1} \quad (1 \leq i \leq h),
$$
and it follows that the vector $\mathbf{z} - \mathbf{z}'$ is counted by $M( \widehat{\mathbf{u}} ; \widehat{\mathbf{v}} )$
for all $\mathbf{z}' \in \mathcal{R}(\widehat{\mathbf{u}}; \widehat{\mathbf{v}}; \mathbf{c})$; therefore, we have
$$
\# \mathcal{R}(\widehat{\mathbf{u}}; \widehat{\mathbf{v}}; \mathbf{c}) \leq M( \widehat{\mathbf{u}} ; \widehat{\mathbf{v}} )
$$
for any $\mathbf{c}$ under consideration. Therefore, by
combining (\ref{sum 1+}) and (\ref{sum 1++}), we obtain
\begin{eqnarray}
\label{ineq 1-3}
&& \sum_{\mathbf{u}_{d} \in \mathcal{U}(\mathbf{u}_1, \ldots, \mathbf{u}_{d - 1})} \Big{|} \sum_{ \mathbf{v}_{d} \in \mathcal{V}(\mathbf{v}_1, \ldots, \mathbf{v}_{d-1}) }
\bpsi_{\widehat{\mathbf{u}} ; \widehat{\mathbf{v}} } (\mathbf{u}_d ; \mathbf{v}_d)
e (  \mathcal{L} (\underline{\mathbf{u}}; \underline{\mathbf{v}})   )   \Big{|}
\\
\notag
&\ll& \sum_{0 \leq  c_1, \ldots, c_h < V} M( \widehat{\mathbf{u}} ; \widehat{\mathbf{v}} )
\prod_{  1 \leq i \leq h} \min \left(  V, \ \max \left(  \frac{ V }{c_{i} }, \ \frac{ V }{ V  -  \min ( c_{i} + 1, V) }  \right)   \right)
\\
\notag
&\ll& M( \widehat{\mathbf{u}} ; \widehat{\mathbf{v}} ) \prod_{1 \leq i \leq h} V \log V.
\end{eqnarray}
Let us define $\mathcal{M}(\alpha; U'; V'; P)$ to be the number of integral vectors
$$
\underline{\mathbf{u}}\in [-U', U']^{h d}
\quad
\textnormal{ and }
\quad
\widehat{\mathbf{v}} \in [-V', V']^{h(d-1)}
$$
satisfying
$$
\| \alpha \Gamma (\underline{\mathbf{u}}; \widehat{\mathbf{v}}, \mathbf{e}_{i}  )  \| \leq P \quad (1 \leq i \leq h).
$$
By substituting (\ref{ineq 1-3}) into (\ref{ineq 2'}), it then follows that
\begin{eqnarray}
\label{exp bound with count 1}
| \mathcal{E} ({\alpha})|^{2^{2 d  -2}}
\ll
X^{\varepsilon}
U^{h(2^{2 d  -2} - d)} V^{ h( 2^{2 d  -2} - d + 1)}
\mathcal{M}({\alpha}; U ; V; V^{-1}).
\end{eqnarray}

We need the following lemmas on geometry of numbers.
\begin{lem} \cite[Lemma 14.1]{S}
\label{shrink lemma}
Let $\mathfrak{L}_1, \ldots, \mathfrak{L}_h$ be symmetric linear forms given by $\mathfrak{L}_i (y_1, \ldots, y_h) = c_{i,1}y_1 + \cdots + c_{i,h} y_h$
$(1 \leq i \leq h)$, i.e. $c_{i,j} = c_{j,i}$ $(1 \leq i, j \leq h)$. Let $\gamma \in \RR_{>1}$ and $Z > 0$. We denote by $\mathfrak{Y}(Z)$ the number of integer solutions $y_1, \ldots, y_{h}$ to the system of inequalities
$$
|y_i| \leq \gamma Z \quad (1 \leq i \leq h) \quad  \textnormal{ and }  \quad \|\mathfrak{L}_i (y_1, \ldots, y_h) \| \leq \gamma^{-1} Z \quad (1 \leq i \leq h).
$$
Then for $0 < Z_1 \leq Z_2 \leq 1$ we have
$$
\frac{\mathfrak{Y}(Z_2)}{\mathfrak{Y}(Z_1)} \ll \left( \frac{Z_2}{Z_1} \right)^{h},
$$
where the implicit constant depends only on $h$.
\end{lem}

\begin{lem} \cite[Lemma 3.1]{DS}
\label{shrink lemma2}
Let $\gamma \in \RR_{>1}$ and $Z > 0$.
Let $\mathfrak{L}_1, \ldots, \mathfrak{L}_h$ be linear forms given by $\mathfrak{L}_i (y_1, \ldots, y_h) = c_{i,1}y_1 + \cdots + c_{i,h} y_h$
$(1 \leq i \leq h)$ and $\mathfrak{Y}(Z)$ be as in Lemma \ref{shrink lemma}. Let us define
$$
\mathfrak{L}_j^T (y_1, \ldots, y_h) = c_{1,j}y_1 + \cdots + c_{h, j} y_h \quad (1 \leq j \leq h).
$$
We denote by $\mathfrak{Y}^T(Z)$ the number of integer solutions $y_1, \ldots, y_{h}$ to the system of inequalities
$$
|y_i| \leq \gamma Z \quad (1 \leq i \leq h) \quad  \textnormal{ and }  \quad \|\mathfrak{L}^T_j (y_1, \ldots, y_h) \| \leq \gamma^{-1} Z \quad (1 \leq j \leq h).
$$
Then for $0 < Z_1 \leq Z_2 \leq 1$ we have
$$
\mathfrak{Y}(Z_2) \ll \left( \frac{Z_2}{Z_1} \right)^{h} (\mathfrak{Y}(Z_1) + \mathfrak{Y}^T(Z_1)) ,
$$
where the implicit constant depends only on $h$.
\end{lem}
In the statement of \cite[Lemma 3.1]{DS}, strict inequalities are used to define $\mathfrak{Y}^T(Z)$, but it can be seen from the proof that
this does not matter.

Let $0 < Q_1, Q_2 \leq 1$ to be set in due course.
For each $0 \leq t \leq d$, let
\begin{eqnarray}
N_2^{(t)}(Q_2; \alpha)
\notag
=
\# \left\{
 \underline{\u} \in \ZZ^{d h},  \widehat{\v} \in \ZZ^{(d-1)h}:
\begin{array}{l}
 \u_1,\ldots, \u_t \in  [- Q_2 U, Q_2 U]^h \\
 \u_{t + 1},\ldots, \u_{d} \in [- U, U]^h \\
 \v_1, \ldots, \v_{d-1} \in [- Q_1V, Q_1 V]^h  \\
 \| \alpha \Gamma(  \underline{\u} ; \widehat{\v},  \mathbf{e}_{i} ) \| \leq  Q_2^t Q_1^{d-1} V^{-1} \quad (1 \leq i \leq h)
\end{array}
\right\}.
\end{eqnarray}
For each $1 \leq t \leq d$, let
\begin{eqnarray}
\widetilde{N}_2^{(t)}(Q_2; \alpha)
\notag
\notag
=
\# \left\{
 \underline{\u} \in \ZZ^{d h},  \widehat{\v} \in \ZZ^{(d-1)h}:
\begin{array}{l}
 \u_1,\ldots, \u_t \in  [- Q_2 U, Q_2 U]^h \\
 \u_{t + 1},\ldots, \u_{d} \in [- U, U]^h \\
 \v_1, \ldots, \v_{d-1} \in [- Q_1V, Q_1 V]^h  \\
 \| \alpha \Gamma( \mathbf{e}_j, \u_2,  \ldots,  \u_d ; \widehat{\v},  \u_1 ) \|
 \leq  Q_2^t Q_1^{d-1} V^{-1} \\
  (1 \leq j \leq h)
\end{array}
\right\}.
\end{eqnarray}

By a repeated application of  Lemmas \ref{shrink lemma} and \ref{shrink lemma2}, we obtain the following estimate.
We remark that there is an oversight in the derivation of \cite[Lemma 6.3]{Y2} resulting from the fact that an analogous estimate was not used
in place of \cite[(6.27)]{Y2}. We correct this oversight in Appendix \ref{B}; we note that the statement of \cite[Lemma 6.3]{Y2} remains unaffected.
\begin{lem}
We have
\begin{eqnarray}
\label{ineq 1-4}
\mathcal{M}(\alpha; U; V; V^{-1}) \ll
Q_1^{- h (d-1) } Q_2^{-h  d}
( N_2^{(d)}(Q_2; \alpha) +  \widetilde{N}_2^{(d)}(Q_2; \alpha)  ).
\end{eqnarray}
\end{lem}
\begin{proof}
For each $0 \leq t \leq d-1$, let
\begin{eqnarray}
N_1^{(t)}(Q_1; \alpha)
\notag
=
\# \left\{
 \underline{\u} \in \ZZ^{d h},  \widehat{\v} \in \ZZ^{(d-1)h}:
\begin{array}{l}
 \u_1, \ldots, \u_d \in [-U, U]^h  \\
 \v_1,\ldots, \v_t \in  [- Q_1 V, Q_1 V]^h \\
  \v_{t + 1},\ldots, \v_{d-1} \in [- V, V]^h \\
 \| \alpha \Gamma(  \underline{\u} ; \widehat{\v},  \mathbf{e}_{i} ) \| \leq Q_1^{t} V^{-1} \quad (1 \leq i \leq h)
\end{array}
\right\}.
\end{eqnarray}
First we prove
\begin{equation}\label{N0 bound}
\mathcal{M}(\alpha; U; V; V^{-1}) = N_1^{(0)}(Q_1; \alpha) \ll Q_1^{ - (d-1) h  } N_1^{(d-1)}(Q_1; \alpha),
\end{equation}
which in turn follows by showing that
$$
N_1^{(t-1)}(Q_1; \alpha) \ll Q_1^{- h}  N_1^{(t)}(Q_1; \alpha),
$$
for each $1 \leq t \leq d-1$. This will be a straightforward consequence of Lemma \ref{shrink lemma}. To see this,
fix $1 \leq t \leq d-1$,  $\underline{\u} \in [-U, U]^{d h}$ and
$\v_i \in \ZZ^h$, for each $i \neq t$, satisfying
\begin{eqnarray}
\label{cond++++++}
\v_1,\ldots, \v_{t-1} \in  [- Q_1 V, Q_1 V]^h
\quad
\textnormal{and}
\quad
 \v_{t + 1},\ldots, \v_{d-1} \in [- V, V]^h.
\end{eqnarray}
Let
\begin{eqnarray}
\mathfrak{L}_{i} (\v)
&=&
\alpha  \Gamma ( \underline{\u};  \v_{1}, \ldots, \v_{t-1}, \v , \v_{t + 1},  \ldots, \v_{d-1}, \mathbf{e}_{i})
\notag
\\
\notag
&=&
\sum_{j = 1}^h \alpha  \Gamma ( \underline{\u};  \v_{1}, \ldots, \v_{t-1}, \mathbf{e}_j , \v_{t + 1},  \ldots, \v_{d-1}, \mathbf{e}_{i}) v_j
\end{eqnarray}
for each $1 \leq i \leq h$.
We write
$$
M_1(Z)  =
\# \{
\v \in [-  Q_1^{ -  \frac{t - 1 }{2}  } V Z  , Q_1^{ -  \frac{t - 1 }{2}  } V Z]^h:  \|\mathfrak{L}_{i}(\v) \| \leq   Q_1^{ \frac{t - 1 }{2}  } V^{-1} Z  \quad (1 \leq i \leq h)  \}
$$
for any  $0 < Z \leq 1$. By applying Lemma \ref{shrink lemma} with
$$
Z_1 = Q_1^{ \frac{ t+1 }{2} }, \quad Z_2 = Q_1^{ \frac{t - 1 }{2} }
\quad \textnormal{and}
\quad \gamma =  Q_1^{ -  \frac{t - 1 }{2}  } V,
$$
we obtain
$$
\frac{M_1(  Z_2  )  }{M_1(  Z_1  )} \ll Q_1^{-h}.
$$
Summing the inequality over all
$\underline{\u} \in [-U, U]^{d h}$ and $\v_{i}$ with $i \neq t$ such that (\ref{cond++++++}) holds,  we arrive at the inequality
$N_1^{(t-1)}(Q_1; \alpha) \ll Q_1^{-h}  N^{(t)}(Q_1; \alpha),$
which is what we were supposed to prove.

Next we prove
\begin{equation}\label{N0 bound+}
 N_1^{(d-1)}(Q_1; \alpha) = N_2^{(0)}(Q_2; \alpha) \ll Q_2^{ - d h  } ( N_2^{(d)}(Q_2; \alpha) +  \widetilde{N}_2^{(d)}(Q_2; \alpha)  ),
\end{equation}
which in turn follows by showing that
$$
N_2^{(t-1)}(Q_2; \alpha) \ll  Q_2^{- h} (  N_2^{(t)}(Q_2; \alpha)   +  \widetilde{N}_2^{(t)}(Q_2; \alpha)  ),
$$
for each $1 \leq t \leq d$, and
$$
\widetilde{N}_2^{(t-1)}(Q_2; \alpha) \ll  Q_2^{- h}  \widetilde{N}_2^{(t)}(Q_2; \alpha),
$$
for each $2 \leq t \leq d$.

We begin with the first claim. Let us fix $1 \leq t \leq d$, $\widehat{\v} \in [- Q_1V, Q_1 V]^{(d - 1) h}$ and
$\u_i \in \ZZ^h$, for each $i \neq t$, satisfying
\begin{eqnarray}
\label{cond+++000}
\u_1,\ldots, \u_{t-1} \in  [- Q_2 U, Q_2 U]^h
\quad
\textnormal{and}
\quad
 \u_{t + 1},\ldots, \u_{d} \in [- U, U]^h.
\end{eqnarray}
Let
\begin{eqnarray}
\mathfrak{L}_{i} (\u)
&=&
\alpha  \Gamma (  \u_{1}, \ldots, \u_{t-1}, \u , \u_{t + 1},  \ldots, \u_{d}; \widehat{\v},   \mathbf{e}_{i})
\notag
\\
\notag
&=& \sum_{j=1}^h  \alpha  \Gamma (  \u_{1}, \ldots, \u_{t-1}, \mathbf{e}_j , \u_{t + 1}, \ldots, \u_{d}; \widehat{\v},   \mathbf{e}_{i}) u_j
\end{eqnarray}
for each $1 \leq i \leq h$, and
\begin{eqnarray}
\mathfrak{L}_{j}^T (\u)
&=&
\sum_{i=1}^h  \alpha  \Gamma (  \u_{1}, \ldots, \u_{t-1}, \mathbf{e}_j , \u_{t + 1},  \ldots, \u_{d}; \widehat{\v},   \mathbf{e}_{i}) u_i
\notag
\\
\notag
&=& \alpha  \Gamma (  \u_{1}, \ldots, \u_{t-1},  \mathbf{e}_j , \u_{t + 1},  \ldots, \u_{d}; \widehat{\v}, \u )
\end{eqnarray}
for each $1 \leq j \leq h$.
We write
\begin{eqnarray}
M_2(Z)
\notag
&=&
\# \{
\u \in [- Q_2^{ - \frac{t - 1 }{2} } Q_1^{- \frac{d-1}{2}} (U V)^{ \frac{1}{2}} Z , Q_2^{ - \frac{t - 1 }{2} } Q_1^{- \frac{d-1}{2}} (U V)^{ \frac{1}{2}} Z ]^h:
\\
\notag
&& \|\mathfrak{L}_{i}(\u) \| \leq   Q_2^{ \frac{t - 1 }{2} } Q_1^{\frac{d-1}{2}} (U V)^{- \frac{1}{2}}  Z \quad (1 \leq i \leq h) \}
\end{eqnarray}
and
\begin{eqnarray}
M_2^T(Z)
\notag
&=&
\# \{
\u \in [- Q_2^{ - \frac{t - 1 }{2} } Q_1^{- \frac{d-1}{2}} (U V)^{ \frac{1}{2}} Z , Q_2^{ - \frac{t - 1 }{2} } Q_1^{- \frac{d-1}{2}} (U V)^{ \frac{1}{2}} Z ]^h:
\\
\notag
&& \|\mathfrak{L}^T_{j}(\u) \| \leq   Q_2^{ \frac{t - 1 }{2} } Q_1^{\frac{d-1}{2}} (U V)^{- \frac{1}{2}}  Z \quad (1 \leq j \leq h) \}
\end{eqnarray}
for any  $0 < Z \leq 1$. By applying Lemma \ref{shrink lemma2} with
$$
Z_1 = Q_2^{ \frac{t + 1 }{2} } Q_1^{\frac{d-1}{2}} (U V^{-1})^{\frac{1}{2}} , \quad Z_2 = Q_2^{ \frac{t - 1 }{2} } Q_1^{\frac{d-1}{2}} (U V^{-1})^{\frac{1}{2}}
\quad
\textnormal{and}
\quad \gamma =  Q_2^{ - \frac{t - 1 }{2} } Q_1^{- \frac{d-1}{2}} (U V)^{ \frac{1}{2}},
$$
we obtain
$$
M_2( Z_2 )   \ll Q_2^{-h}  (M_2( Z_1  ) + M_2^T( Z_1)).
$$
Summing the inequality over all
$\widehat{\v} \in [-Q_1 V, Q_1 V]^{(d - 1) h}$ and
$\u_{i}$ with $i \neq t$ such that (\ref{cond+++000}) holds,
we arrive at the inequality
\begin{eqnarray}
&&N_2^{(t-1)}(Q_2; \alpha)
\notag
\\
\notag
&\ll&  Q_2^{-h}  N^{(t)}(Q_2; \alpha) +  Q_2^{-h} \cdot
\\
\notag
&& \# \left\{
 \underline{\u} \in \ZZ^{d h},  \widehat{\v} \in \ZZ^{(d-1)h}:
\begin{array}{l}
 \u_1,\ldots, \u_t \in  [- Q_2 U, Q_2 U]^h \\
 \u_{t + 1},\ldots, \u_{d} \in [- U, U]^h \\
 \v_1, \ldots, \v_{d-1} \in [- Q_1V, Q_1 V]^h  \\
 \| \alpha \Gamma( \u_1, \ldots, \u_{t-1},  \mathbf{e}_j, \u_{t+1}, \ldots, \u_d ; \widehat{\v},  \u_t ) \|
 \leq  Q_2^t Q_1^{d-1} V^{-1} \\ (1 \leq j \leq h)
\end{array}
\right\}
\\
\notag
&\ll& Q_2^{-h}  N_2^{(t)}(Q_2; \alpha) +  Q_2^{-h} \widetilde{N}_2^{(t)}(Q_2; \alpha).
\end{eqnarray}

Next we prove the second claim. This again will be a straightforward consequence of Lemma \ref{shrink lemma}. To see this,
fix $2 \leq t \leq d$,  $\widehat{\v} \in [- Q_1V, Q_1 V]^{(d - 1) h}$ and
$\u_i \in \ZZ^h$, for each $i \neq t$, satisfying
\begin{eqnarray}
\label{cond+++}
\u_1,\ldots, \u_{t-1} \in  [- Q_2 U, Q_2 U]^h
\quad
\textnormal{and}
\quad
\u_{t + 1},\ldots, \u_{d} \in [- U, U]^h.
\end{eqnarray}
Let
\begin{eqnarray}
\mathfrak{L}_{j} (\u)
&=&
\alpha  \Gamma (  \mathbf{e}_{j}, \u_{2}, \ldots, \u_{t-1}, \u , \u_{t + 1},  \ldots, \u_{d}; \widehat{\v},    \u_1)
\notag
\\
\notag
&=& \sum_{i=1}^h  \alpha  \Gamma ( \mathbf{e}_{j}, \u_{2}, \ldots, \u_{t-1}, \mathbf{e}_i , \u_{t + 1},  \ldots, \u_{d}; \widehat{\v},   \u_1) u_i
\end{eqnarray}
for each $1 \leq j \leq n$. We write
\begin{eqnarray}
M_3(Z)
\notag
&=&
\# \{
\u \in [- Q_2^{ - \frac{t - 1 }{2} } Q_1^{- \frac{d-1}{2}} (U V)^{ \frac{1}{2}} Z , Q_2^{ - \frac{t - 1 }{2} } Q_1^{- \frac{d-1}{2}} (U V)^{ \frac{1}{2}} Z ]^h:
\\
\notag
&& \|\mathfrak{L}_{j}(\u) \| \leq   Q_2^{ \frac{t - 1 }{2} } Q_1^{\frac{d-1}{2}} (U V)^{- \frac{1}{2}}  Z \quad (1 \leq j \leq h) \}
\end{eqnarray}
for any  $0 < Z \leq 1$. By applying Lemma \ref{shrink lemma} with
$$
Z_1 = Q_2^{ \frac{t + 1 }{2} } Q_1^{\frac{d-1}{2}} (U V^{-1})^{\frac{1}{2}} , \quad Z_2 = Q_2^{ \frac{t - 1 }{2} } Q_1^{\frac{d-1}{2}} (U V^{-1})^{\frac{1}{2}}
\quad \textnormal{and}
\quad \gamma =  Q_2^{ - \frac{t - 1 }{2} } Q_1^{- \frac{d-1}{2}} (U V)^{ \frac{1}{2}},
$$
we obtain
$$
\frac{M_3( Z_2 )  }{M_3( Z_1  )} \ll Q_2^{-h}.
$$
Summing the inequality over all
$\widehat{\v} \in [-Q_1 V, Q_1 V]^{(d - 1) h}$ and
$\u_{i}$ with $i \neq t$ such that (\ref{cond+++}) holds,  we arrive at the inequality
$\widetilde{N}_2^{(t-1)}(Q_2; \alpha) \ll Q_2^{-h}  \widetilde{N}_2^{(t)}(Q_2; \alpha),$
which is what we were supposed to prove. Finally, the result follows on combining (\ref{N0 bound}) and (\ref{N0 bound+}).
\end{proof}
With this estimate we obtain the following lemma. For simplicity let us denote
\begin{equation}
\label{V notation}
V^*_2 = V^{*}_{g^{[2d]}, 2}.
\end{equation}
\begin{lem}
\label{6.2}
Let $0 < \vartheta < 1$.
Let $\varepsilon > 0$ be sufficiently small.
Then for $X$ sufficiently large, at least one of the following alternatives holds:

\textnormal{i)} One has the upper bound
\begin{eqnarray}
| \mathcal{E} (\alpha) | \ll
\notag
X^{\varepsilon} U^h V^h U^{- \vartheta  \frac{\textnormal{codim} \tsp V^*_{2}}{2^{2 d -2}} }.
\end{eqnarray}

\textnormal{ii)}
There exist $1 \leq q \leq U^{  \vartheta  (2d-1) }$ and $a \in \mathbb{Z}$ with
$\gcd( {a}, q) = 1$ such that
\begin{eqnarray}
| q \alpha - a | \leq X^{-d + d \sigma} U^{ \vartheta (2d-1) }.
\notag
\end{eqnarray}
\end{lem}
\begin{proof}
If
$$
X^{ \frac{\varepsilon}{2} } \geq  U^{ \vartheta  \frac{\textnormal{codim} \tsp V^*_{2}}{2^{2 d -2}} },
$$
then the result is trivial. Therefore, let us suppose otherwise and set
$$
\vartheta' = \vartheta -   \frac{\varepsilon}{2} \cdot   \frac{ \log X  }{ \log U } \cdot  \frac{ 2^{2 d -2}  }{  \textnormal{codim} \tsp V^*_{2} }.
$$
Note from (\ref{VV}) and (\ref{UV}) we have
\begin{equation}
\label{VVV}
X^{\frac12 - \sigma_0}  \ll  U \ll  X^{\frac12}.
\end{equation}
Consider the affine varieties $\mathcal{Z}$ and $\widetilde{ \mathcal{Z} }$ defined by
$$
\mathcal{Z} = \{   (\underline{\mathbf{u}}, \widehat{\mathbf{v}} ) \in \mathbb{A}^{h(2 d - 1)}_{\mathbb{C}}  :
\Gamma(  \underline{\mathbf{u}};  \widehat{\mathbf{v}}, \mathbf{e}_{i} ) = 0 \quad (1 \leq i \leq h)  \}
$$
and
$$
\widetilde{ \mathcal{Z} } = \{  (\underline{\mathbf{u}}, \widehat{\mathbf{v}} ) \in \mathbb{A}^{h(2 d - 1)}_{\mathbb{C}}  :
\Gamma( \mathbf{e}_j, \u_2, \ldots, \u_d ;  \widehat{\mathbf{v}},  \u_1 ) = 0 \quad (1 \leq j \leq h) \}.
$$
Let us define
\begin{eqnarray}
\notag
\mathcal{N}(\mathcal{Z})
=
\notag
\left\{
(\underline{\mathbf{u}}, \widehat{\mathbf{v}}) \in \mathbb{Z}^{h(2 d  - 1)} \cap \mathcal{Z} :
\begin{array}{l}
\u_{1}, \ldots, \u_d \in [- Q_2 U,  Q_2  U]^h   \\
\v_{1}, \ldots, \v_{d-1} \in [- Q_1 V,  Q_1  V]^h
\end{array}{}
\right\}
\end{eqnarray}
and similarly $\mathcal{N}(\widetilde{\mathcal{Z}})$.
In this proof, we set $Q_1 = U^{\vartheta'}V^{- 1}$ and $Q_2 = U^{\vartheta' - 1}$ so that $0 < Q_1, Q_2 \leq 1$.

Suppose we have that every point counted by
$N_2^{(d)}(Q_2; \alpha)$ is contained in $\mathcal{N}(\mathcal{Z})$ and every point counted by $ \widetilde{N}_2^{(d)}(Q_2; \alpha)$
is contained in $\mathcal{N}(\widetilde{\mathcal{Z}})$.  Then we apply \cite[(3.1)]{Bro};
this bound is independent of the coefficients of the polynomials defining the affine variety (depending only on the dimension and the degree).
As a result, we obtain
\begin{eqnarray}
\label{ineq 4 1}
N_2^{(d)}(Q_2; \alpha) + \widetilde{N}_2^{(d)}(Q_2; \alpha)
\ll
\# \mathcal{N}(\mathcal{Z}) + \# \mathcal{N}(\widetilde{\mathcal{Z}}) \ll U ^{\vartheta' \dim \mathcal{Z}} +  U ^{\vartheta' \dim \widetilde{\mathcal{Z}}}.
\end{eqnarray}
Therefore, it follows from (\ref{exp bound with count 1}), (\ref{ineq 1-4}) and (\ref{ineq 4 1}) that
\begin{eqnarray}
\label{6.23}
\\
\notag
| \mathcal{E} (\alpha)|^{2^{2 d -2}}
&\ll& X^{ \varepsilon } U^{h( 2^{2 d -2} - d)} V^{ h( 2^{2 d -2} - d + 1)}
V^{h(d-1)}  U^{- h(d-1) \vartheta' }
U^{hd} U^{-h d \vartheta' }
( U^{\vartheta' \dim \mathcal{Z} } +  U^{\vartheta' \dim  \widetilde{\mathcal{Z}} }  )
\\
\notag
&=&
X^{\varepsilon } U^{h 2^{2 d - 2}} V^{h 2^{2 d -2}}
( U^{\vartheta'( - h (d - 1) -  h d +  \dim \mathcal{Z} )}  +  U^{\vartheta'( - h (d - 1) -  h d +  \dim  \widetilde{ \mathcal{Z} } )}  ).
\end{eqnarray}
Let
\begin{eqnarray}
\mathcal{D}_1 = \{  (\underline{\mathbf{u}}, \widehat{\mathbf{v}}) \in \mathbb{A}_{\mathbb{C}}^{h(2 d  -1 )}:  \mathbf{u}_1 = \cdots = \mathbf{u}_{d},
\mathbf{v}_1 = \cdots = \mathbf{v}_{d-1} \}
\end{eqnarray}
and
\begin{eqnarray}
\mathcal{D}_2 = \{  (\underline{\mathbf{u}}, \widehat{\mathbf{v}}) \in \mathbb{A}_{\mathbb{C}}^{h(2 d  -1 )}:  \mathbf{u}_2 = \cdots = \mathbf{u}_{d},
\mathbf{v}_1 = \cdots = \mathbf{v}_{d-1} = \u_1 \}.
\end{eqnarray}
Then from  (\ref{derivcond}) and (\ref{singlocbhmg}) we have
\begin{equation}
\notag
\dim V^{*}_{2} = \dim ( \mathcal{Z} \cap \mathcal{D}_1 ) \geq \dim \mathcal{Z} - h (d - 1) -  h (d - 2)  =   - h (d - 1)  - hd + \dim \mathcal{Z} + 2 h.
\end{equation}
Similarly, since
\begin{eqnarray}
\notag
\widetilde{\mathcal{Z}} \cap \mathcal{D}_2  &=& \{  (\u, \v) \in \mathbb{A}^{2 h}_{\mathbb{C}}  :
\Gamma( \mathbf{e}_j, \u, \ldots, \u ;  \v, \ldots, \v ) = 0 \quad (1 \leq j \leq h) \}
\\
\notag
&=&
\left\{  (\u, \v) \in \mathbb{A}^{2 h}_{\mathbb{C}}  :
\frac{ \partial g^{[2d]} }{\partial u_j} ( \u ; \v ) = 0 \quad (1 \leq j \leq h) \right\}
\\
\notag
&\cong&
\left\{  (\u, \v) \in \mathbb{A}^{2 h}_{\mathbb{C}}  :
\frac{ \partial g^{[2d]} }{\partial v_i} ( \u ; \v ) = 0 \quad (1 \leq i \leq h) \right\},
\end{eqnarray}
where the isomorphism follows from the fact that $g^{[2d]}(\u; \v)$ is symmetric in $\u$ and $\v$ (recall \ref{ggggggggg}),
we also obtain
\begin{equation}
\notag
\dim V^{*}_{2} = \dim ( \widetilde{\mathcal{Z}} \cap \mathcal{D}_2 ) \geq \dim \widetilde{\mathcal{Z}} - h (d - 1) -  h (d - 2)  =   - h (d - 1)  - hd + \dim \widetilde{\mathcal{Z}} + 2 h.
\end{equation}
With these inequalities, (\ref{6.23}) becomes
\begin{eqnarray}
| \mathcal{E} (\alpha)|^{2^{2 d -2}}
\notag
\ll
\notag
X^{\varepsilon  }
U^{h 2^{2 d - 2}} V^{h 2^{2 d -2}}
U^{- \vartheta' \textnormal{codim} \tsp V^{*}_{2} },
\end{eqnarray}
and the estimate in i) follows immediately from the definition of $\vartheta'$.

On the other hand, suppose there exists $(\underline{\mathbf{u}}, \widehat{\mathbf{v}})$ counted by
$N_2^{(d)}(Q_2; \alpha)$ which is not contained in $\mathcal{N}(\mathcal{Z})$, or  counted by $ \widetilde{N}_2^{(d)}(Q_2; \alpha)$
which is not contained in $\mathcal{N}(\widetilde{\mathcal{Z}})$. Let us begin with the former case, i.e.
there exists $1 \leq i_0 \leq h$ such that
$$
\Gamma(\underline{\mathbf{u}};  \widehat{\mathbf{v}}, \mathbf{e}_{i_0})  \not = 0.
$$
Let us write
$$
\alpha \Gamma(\underline{\mathbf{u}};  \widehat{\mathbf{v}}, \mathbf{e}_{i_0}) = a + \xi,
$$
where $a \in \mathbb{Z}$ and $|\xi | \leq  Q_1^{d-1} Q_2^{d} V^{-1}$.
Let $q$ be the absolute value of $\Gamma(\underline{\mathbf{u}};  \widehat{\mathbf{v}}, \mathbf{e}_{i_0})$.
Then
\begin{eqnarray}
\notag
1 \leq q \ll Q_1^{d-1} Q_2^{d} U^d V^{d-1}
= U^{\vartheta' (2d -1)}.
\end{eqnarray}
It follows from (\ref{UV}) that $X^{1 - \sigma} \ll UV$.
Therefore, we obtain
\begin{eqnarray}
| \xi | \leq Q_1^{d-1} Q_2^{d} V^{-1} = \frac{U^{\vartheta' (2 d  - 1)}}{ U^d V^d } \ll X^{-d + d \sigma}  U^{\vartheta' (2 d  - 1)}.
\notag
\end{eqnarray}
Finally, since $\vartheta' < \vartheta$, it follows that $1 \leq q  \leq U^{ \vartheta (2d-1) }$
and $| \xi | \leq X^{-d + d \sigma}  U^{\vartheta (2 d  - 1)}$
for $X$ sufficiently large; we have obtained the statement in ii). On the other hand, for the latter case
we may obtain the same conclusion in a similar manner.
\end{proof}
By combining (\ref{case 4 ineq 1}) and Lemma \ref{6.2}, we obtain the following.
\begin{prop}
\label{6.2 for S}
Let $0 < \vartheta < 1$.
Let $\varepsilon > 0$ be sufficiently small.
Then for $X$ sufficiently large, at least one of the following alternatives holds:

\textnormal{i)} One has the upper bound
\begin{eqnarray}
| S_{\ell;  \underline{\mathbf{M}}, \underline{\mathbf{N} }} (\alpha) |
\ll
\notag
X^{n + \varepsilon}  U^{- \vartheta  \frac{\textnormal{codim} \tsp V^*_{2}}{2^{2 d -2}} }.
\end{eqnarray}

\textnormal{ii)}
There exist $1 \leq q \leq U^{  \vartheta  (2d-1) }$ and $a \in \mathbb{Z}$ with
$\gcd( {a}, q) = 1$ such that
\begin{eqnarray}
| q \alpha - a | \leq X^{-d + d \sigma} U^{ \vartheta (2d-1) }.
\notag
\end{eqnarray}
\end{prop}
We let $\nu = \frac{\log U}{\log X}$. Then from (\ref{VVV}) it follows that
\begin{eqnarray}
\label{VVVV}
\frac{1 - 2 \sigma_0}{2 + \varepsilon} < \nu < \frac12 + \varepsilon.
\end{eqnarray}
We define $\mathfrak{M}^{(\textnormal{II})}(\vartheta)$ to be the set of $\alpha \in [0,1)$
satisfying ii) in Proposition \ref{6.2 for S}, i.e. there exist $1 \leq q \leq X^{  (2d-1) \nu \vartheta }$ and $a \in \mathbb{Z}$ with
$\gcd( {a}, q) = 1$ such that
\begin{eqnarray}
| q \alpha - a | \leq X^{(2d-1) \nu \vartheta +  d \sigma - d}.
\notag
\end{eqnarray}

\subsection{Sliding scale argument}
We begin by setting
$$
\mathcal{K} = \frac{\textnormal{codim} \tsp V^{*}_{2}}{2^{2d - 2} }.
$$
For any $0 < \vartheta < 1$,
the Lebesgue measure of $\mathfrak{M}^{(\textnormal{II})}( \vartheta)$ is bounded by the following quantity
\begin{eqnarray}
\textnormal{meas}(\mathfrak{M}^{(\textnormal{II})}( \vartheta) )
&\ll& \sum_{1 \leq  q \leq X^{  (2d-1)    \nu \vartheta  } }
\sum_{ \substack{  0 \leq a \leq q  \\  \gcd(a, q) = 1 } } q^{-1} X^{-d + d \sigma  +    (2d-1)   \nu \vartheta  }
\label{Lm1}
\\
\notag
&\ll& X^{- d  + d \sigma  +   2 (2d - 1)    \nu \vartheta  }.
\end{eqnarray}
We then define a sequence
$$
0 < \vartheta < \vartheta_1 < \cdots < \vartheta_J < 1
$$
satisfying
$$
\frac{\tau}{2} > 2 (2d-1) \nu (\vartheta_{1} - \vartheta),
\quad \frac{\tau}{2} > 2 (2d-1) \nu (\vartheta_{j+1} - \vartheta_{j})
\quad (1 \leq j \leq J-1)
$$
and
$$
1 - \tau < \vartheta_J < 1,
$$
where $\tau > 0$ is sufficiently small. In particular, $J \ll 1$.
Then, under the assumption that $\mathcal{K} > 2(2d-1)$ holds, by Proposition \ref{6.2 for S} and (\ref{Lm1}) we have
\begin{eqnarray}
\label{minor arc 1 bdd1}
&&\int_{[0,1) \setminus  \mathfrak{M}^{(\textnormal{II})}( \vartheta) }  | S_{\ell;  \underline{\mathbf{M}}, \underline{\mathbf{N} }} (\alpha )| \tsp {d} {\alpha}
\\
\notag
&=&
\int_{[0,1) \setminus  \mathfrak{M}^{(\textnormal{II})} ( \vartheta_J) }  |S_{\ell;  \underline{\mathbf{M}}, \underline{\mathbf{N} }} (\alpha )| \tsp {d} {\alpha}
+
\sum_{j=1}^{J-1}  \int_{\mathfrak{M}^{(\textnormal{II})} ( \vartheta_{j+1}) \setminus  \mathfrak{M}^{(\textnormal{II})}( \vartheta_{j}) }
 |S_{\ell;  \underline{\mathbf{M}}, \underline{\mathbf{N} }} (\alpha )| \tsp {d} {\alpha}
\\
\notag
&+&
\int_{\mathfrak{M}^{(\textnormal{II})} ( \vartheta_1) \setminus  \mathfrak{M}^{(\textnormal{II})}( \vartheta) }
|S_{\ell;  \underline{\mathbf{M}}, \underline{\mathbf{N} }} (\alpha )| \tsp {d} {\alpha}
\\
\notag
&\ll&
X^{n -  \mathcal{K} \nu + \varepsilon }
+
X^{n   - d + d \sigma  -  ( \mathcal{K} -  2 (2d-1)  ) \nu \vartheta    + \varepsilon }
\\
\notag
&\ll&
X^{n   - d  - \varepsilon }
+
X^{n   - d + d \sigma  -  ( \mathcal{K} -  2 (2d-1)  ) \nu \vartheta    + \varepsilon },
\end{eqnarray}
where the final inequality follows provided
\begin{eqnarray}
\label{17}
\mathcal{K} \geq 33 d^2 + 1 = \frac{d}{  \Delta } + 1 >  \frac{(2 + \varepsilon ) d }{ 1 - 2\sigma_0 }  > \frac{d}{\nu}.
\end{eqnarray}

Now we define $0 < \vartheta_0' < 1$ by
$$
\vartheta_0' = \frac{ (2 +  \varepsilon) d \sigma}     {  (1 - 2 \sigma_0)  (\mathcal{K} -  2 (2d-1) ) }.
$$
On recalling (\ref{VVVV}) we see that
\begin{equation}
\label{theta0}
\frac{d \sigma}     { \mathcal{K} -  2 (2d-1)  } <  \nu \vartheta_0'
\end{equation}
is satisfied.
We further need $\vartheta_0'$ to satisfy
\begin{equation}
\label{theta0 1}
(2d-1) \vartheta_0'+  \frac{1}{12}  < \frac{5}{24},
\end{equation}
which follows provided
\begin{eqnarray}
\label{cond KK}
\\
\notag
\mathcal{K} >
\frac{ 2 d \sigma (2d-1)  }{  (\frac{5}{24} -  \frac{1}{12} ) (1 - 2 \sigma_0) } + 2 (2d-1) =
\frac{  d \sigma (2d-1)  }{  (\frac{5}{24} -  \frac{1}{12}) \Delta } + 2 (2d-1)  > 2 (2d-1).
\end{eqnarray}
Recall the definition of $\mathcal{C}_0$ given in (\ref{defn C0}).
Let us set $H = \mathfrak{t}d$ with $\mathfrak{t} \in \NN$ so that  (recall (\ref{defn delta}))
$$
\sigma = \frac{ \frac12 - \Delta}{\mathfrak{t} d}.
$$
Then from (\ref{codim dj0+}), which states
\begin{eqnarray}
\notag
\textnormal{codim} \tsp V^{*}_{2} \geq
\frac{\textnormal{codim} \tsp V^*_{F } - (2H + 1) \mathcal{C}_0 }{H + 1},
\end{eqnarray}
we see that (\ref{cond KK}) is satisfied provided
\begin{eqnarray}
\label{codim cond+}
\textnormal{codim} \tsp V^*_{F}  &>&  (H+1)  4^{d-1} \left( \frac{  d \sigma (2d-1)  }{ \frac{1}{8}  \Delta } + 2 (2d-1) \right) + (2H + 1) \mathcal{C}_0
\\
&=&
\notag
d (2d-1) 4^{d}
\left(  \frac{ (H + 1)  \sigma   }{ \frac{1}{2} \Delta } +  \frac{H + 1}{2 d}  \right)  + (2 \mathfrak{t} d + 1 ) \mathcal{C}_0
\\
&=&
\notag
d (2d-1) 4^{d}
\left(  \frac{ (\mathfrak{t} d + 1)  \frac{ \frac12 - \Delta}{\mathfrak{t} d }   }{ \frac{1}{2} \Delta } +  \frac{\mathfrak{t} d + 1}{2 d}  \right)  + (2 \mathfrak{t} d + 1 ) \mathcal{C}_0
\\
&=&
\notag
d^2 (2d-1) 4^{d}
\left(  \frac{ ( \mathfrak{t} +  \frac{1}{d})   (\frac12 - \Delta )    }{  \frac{1}{2}  \mathfrak{t} d \Delta } +  \frac{\mathfrak{t} + \frac{1}{d} }{2 d}  \right)  + (2 \mathfrak{t} d + 1 ) \mathcal{C}_0.
\end{eqnarray}
Recall that $\Delta = 1/(33d)$.
Let us choose $\mathfrak{t} = 6$, in which case (\ref{codim cond+}) becomes
\begin{eqnarray}
\notag
\textnormal{codim} \tsp V^*_{F}
&>&
d^2 (2d-1) 4^{d}
\left(  \frac{ ( 6 +  \frac{1}{d})   (\frac12 -  \frac{1}{33 d} )    }{  \frac{3}{33} } +  \frac{6 + \frac{1}{d} }{ 2 d}  \right)  + (12 d + 1 ) \mathcal{C}_0
\\
\notag
&=&
d^2 (2d-1) 4^{d}
\left( 33 \frac{ ( 6 +  \frac{1}{d})   (\frac12 -  \frac{1}{33 d} )    }{  3 }
+ \frac{3}{d} + \frac{1}{2 d^2 } \right)  + ( 12 d + 1 ) \mathcal{C}_0.
\end{eqnarray}
Since
\begin{eqnarray}
33 \frac{ ( 6 +  \frac{1}{d})   (\frac12 -  \frac{1}{33 d} )    }{  3 }
+ \frac{3}{d} + \frac{1}{2 d^2 }
\notag
&<&
33 \frac{ ( 6 +  \frac{1}{d})   \frac12     }{  3 }
+ \frac{3}{d} + \frac{1}{2 d^2 }
\\
\notag
&=&
33 \frac{ ( 6 +  \frac{1}{d})    }{  6 }
+ \frac{3}{d} + \frac{1}{2 d^2 }
\\
\notag
&<&
\notag
33 + \frac{10}{d}
\\
\notag
&<& 33 + \frac{37}{2d-1 },
\end{eqnarray}
we see that (\ref{codim cond+}) holds if we have (\ref{codim of F}). It can also be verified in a similar manner that (\ref{17}) holds provided
\begin{eqnarray}
\notag
\textnormal{codim} \tsp V^*_{F}  &\geq&  (6 d + 1)  4^{d-1} ( 33 d^2 + 1) + (12 d + 1) \mathcal{C}_0
\\
\notag
&=&
d^2 (2d-1) 4^d    \left(    \frac{(6 d + 1) (33 d^2 + 1)   }{4 (2d-1) d^2 }   \right)  + (12 d + 1) \mathcal{C}_0.
\end{eqnarray}
Since
$$
\frac{(6 d + 1) (33 d^2 + 1)   }{4 (2d-1) d^2 }  =
\frac{99}{4} + \frac{37}{2d-1 } - \frac{2}{ d} - \frac{1}{4d^2}
< \frac{99}{4} + \frac{37}{2d-1 } < 33 + \frac{37}{2d-1 },
$$
we see that (\ref{17}) is also satisfied if we have (\ref{codim of F}). Therefore, it follows that (\ref{theta0 1}) and (\ref{cond KK}) hold.

Let us set
$$
\lambda = d \sigma =   \frac{ \frac{1}{2} - \Delta }{6}  <  \frac{1}{12}
$$
in the definition of $\mathfrak{M}^+ ( \vartheta_0; \lambda )$ given in (\ref{defn major}).
Let $\vartheta_0 = (2d-1)\vartheta_0'(\frac12 + \varepsilon)$.
Then from (\ref{theta0 1}), which we now know to hold assuming (\ref{codim of F}), it follows that
$$
2 \vartheta_0 +  \left( 2 \vartheta_0 +  \frac{2}{12}  + \varepsilon \right) < \frac{5}{12},
$$
i.e. (\ref{theta0 major}) is satisfied with $\vartheta_0$ and $\lambda$ as in this section and $\gamma = 2 \vartheta_0 + 2 \lambda + \varepsilon$.
In particular, we have that $\vartheta_0 < 1/16$. It is then easy to see from the definition of $\mathfrak{M}^{(\textnormal{II})} (\vartheta)$ and (\ref{VVVV}) that
$$
\mathfrak{M}^{(\textnormal{II})} (\vartheta_0') \subseteq  \mathfrak{M}^+ (1/16- \varepsilon; 1/12),
$$
hence
$$
\mathfrak{m}^+
=  [0,1) \setminus  \mathfrak{M}^+(1/16- \varepsilon; 1/12) \subseteq  [0,1) \setminus  \mathfrak{M}^{(\textnormal{II})}(\vartheta_0').
$$
Therefore, by (\ref{minor arc 1 bdd1}), (\ref{theta0}) and (\ref{cond KK}) we obtain
\begin{eqnarray}
\label{57}
\int_{ \mathfrak{m}^+ } |S_{ { \boldsymbol{\ell}; \underline{ \mathbf{M}}}, \underline{\mathbf{N}}}  ( \alpha)| \tsp d \alpha
\leq
\int_{[0,1) \setminus  \mathfrak{M}^{(\textnormal{II})}( \vartheta_0') }  |S_{ { \boldsymbol{\ell}; \underline{ \mathbf{M}}}, \underline{\mathbf{N}}}  ( \alpha)| \tsp {d} {\alpha}
\ll X^{n - d - \varepsilon}
\end{eqnarray}
as desired.

\section{Conclusion}
\label{conc}
By combining Propositions \ref{prop major} (with $\vartheta_0 = 1/16 - \varepsilon$ and $\lambda = 1/12$) and \ref{prop minor}, we obtain Theorem \ref{mainthm} when $F$ satisfies (II) of Definition \ref{dichotomy}.
We now deal with the remaining case. Suppose $F$ satisfies (I) of Definition \ref{dichotomy}. Let us denote $\mathbf{w} = (w_1, \ldots, w_{\ell})$, $\mathbf{s} = (s_1, \ldots, s_{m})$
and $\mathbf{t} = (t_1, \ldots, t_{n - m - \ell})$. We also let $\mathbf{s}' = (s'_1, \ldots, s'_{m})$ and $\mathbf{t}' = (t'_1, \ldots, t'_{n - m - \ell})$. Let
\begin{eqnarray}
\mathfrak{g}_{\mathbf{w}}(\mathbf{s}, \mathbf{s}', \mathbf{t}, \mathbf{t}' ) = F(\mathbf{s}, \mathbf{t}, \mathbf{w}) - F(\mathbf{s}, {\mathbf{t}'}, \mathbf{w}) - F(
{\mathbf{s}'}, \mathbf{t} , \mathbf{w}) + F({\mathbf{s}'}, {\mathbf{t}'} , \mathbf{w}).
\notag
\end{eqnarray}
Then
$$
\mathfrak{g}^{[d]}_{\mathbf{w}}(\mathbf{s}, {\mathbf{s}'}, \mathbf{t}, {\mathbf{t}'} ) =
\mathfrak{G}(\mathbf{s}, \mathbf{t}) - \mathfrak{G}(\mathbf{s}, {\mathbf{t}'}) - \mathfrak{G}( {\mathbf{s}'}, \mathbf{t}) + \mathfrak{G}({\mathbf{s}'}, {\mathbf{t}'})
$$
for each fixed $\mathbf{w}$, where
$\mathfrak{G}$ is defined in (\ref{den G 1}).
By applying the Cauchy-Schwarz inequality twice, we obtain
\begin{eqnarray}
\label{CS twice ineq}
|S(\alpha)|^4 &\ll& X^{2n + \ell + \varepsilon} \sum_{\mathbf{w} \in [0,X]^{\ell}} \Big{|} \sum_{\mathbf{s}, {\mathbf{s}'} \in [0, X]^{m}} \tsp \sum_{\mathbf{t}, {\mathbf{t}'} \in [0, X]^{n -
m - \ell}} e(\alpha \mathfrak{g}_{\mathbf{w}}(\mathbf{s}, {\mathbf{s}'}, \mathbf{t}, {\mathbf{t}'} )) \Big{|}.
\end{eqnarray}
Since $\mathfrak{g}^{[d]}_{\mathbf{w}}(\mathbf{s}, \mathbf{0}, \mathbf{t}, \mathbf{0} ) = \mathfrak{G}(\mathbf{s}, \mathbf{t})$, it follows from \cite[Lemma 3.1]{Y2} that
\begin{equation}
\label{ineq codim 1'}
\textnormal{codim} \tsp V_{ \mathfrak{g}^{[d]}_{\mathbf{w}} }^* \geq \textnormal{codim} \tsp V_{\mathfrak{G}}^* > \mathcal{C}_0
\end{equation}
for each fixed $\mathbf{w}$. Then we obtain the following as an immediate consequence of \cite[Lemma 4.3]{Bir}
(additional explanation can be found  in the proof of \cite[Lemma 4.2]{Y2}).
\begin{lem}
\label{last lem}
Let $F \in \mathbb{Z}[x_1, \ldots, x_n]$ be a homogeneous form of degree $d \geq 2$ satisfying \textnormal{(I)} of Definition \ref{dichotomy}.
Let $0 < \varsigma < 1$.
Let $\varepsilon > 0$ be sufficiently small.
Then for $X$ sufficiently large, at least one of the following alternatives holds:

\textnormal{i)} One has the upper bound
$$
|S(\alpha)| \ll X^{n - \varsigma 2^{- 1 - d} \mathcal{C}_0 + \varepsilon}.
$$

\textnormal{ii)} There exist $1 \leq q \leq X^{(d-1) \varsigma }$ and $a \in \mathbb{Z}$ with $\gcd (a, q) = 1$
such that
$$
| q \alpha - a | \leq X^{-d + (d-1)\varsigma}.
$$
\end{lem}

We define $\mathfrak{M}^{(\textnormal{I})}(\varsigma)$ to be the set of $\alpha \in [0,1)$
satisfying ii) in Lemma \ref{last lem}, i.e. there exist $1 \leq q \leq X^{ (d-1) \varsigma }$ and $a \in \mathbb{Z}$ with
$\gcd( {a}, q) = 1$ such that
\begin{eqnarray}
| q \alpha - a | \leq X^{-d + (d-1) \varsigma }.
\notag
\end{eqnarray}

\subsection{Sliding scale argument}
For any $0 < \varsigma < 1$,
the Lebesgue measure of $\mathfrak{M}^{(\textnormal{I})}(\varsigma)$ is bounded by the following quantity
\begin{eqnarray}
\textnormal{meas}(\mathfrak{M}^{(\textnormal{I})}( \varsigma) )
\ll \sum_{1 \leq  q \leq X^{  (d-1) \varsigma  } }
\sum_{ \substack{  0 \leq a \leq q  \\  \gcd(a, q) = 1 } } q^{-1} X^{-d  +   (d-1)  \varsigma  }
\label{Lm2}
\ll X^{- d  + 2 (d - 1)  \varsigma  }.
\end{eqnarray}
We then define a sequence
$$
0 < \varsigma = \varsigma_0 < \varsigma_1 < \cdots < \varsigma_J < 1
$$
satisfying
$$
\frac{\tau}{2} > 2 (d-1) (\varsigma_{j+1} - \varsigma_j)
\quad
(0 \leq j \leq J-1)
\quad
\textnormal{ and }
\quad
1 - \tau < \varsigma_J < 1,
$$
where $\tau > 0$ is sufficiently small. In particular, $J \ll 1$.
From the definition of $\mathcal{C}_0$ given in (\ref{defn C0}), it is clear that
$$
2^{- 1 - d} \mathcal{C}_0 > d \quad \textnormal{ and } \quad  2^{- 1 - d} \mathcal{C}_0 -  2 (d-1)   > 0.
$$
Then by Lemma \ref{last lem} and (\ref{Lm2}) we have
\begin{eqnarray}
\label{minor arc 1 bdd1+}
\int_{[0,1) \setminus  \mathfrak{M}^{(\textnormal{I})}( \varsigma) }  |S (\alpha)| \tsp {d} {\alpha}
&=&
\int_{[0,1) \setminus  \mathfrak{M}^{(\textnormal{I})} ( \varsigma_J) }  |S (\alpha)| \tsp {d} {\alpha}
+
\sum_{j=0}^{J-1} \int_{\mathfrak{M}^{(\textnormal{I})} ( \varsigma_{j+1}) \setminus  \mathfrak{M}^{(\textnormal{I})}( \varsigma_{j}) }  |S (\alpha)| \tsp {d} {\alpha}
\\
\notag
&\ll&
X^{n - 2^{- 1 - d} \mathcal{C}_0 + \varepsilon }
+
\sum_{j=0}^{J-1}  X^{- d  + 2 (d - 1)  \varsigma_{j+1}  }  X^{n - \varsigma_j 2^{- 1 - d} \mathcal{C}_0 + \varepsilon }
\\
\notag
&\ll&
X^{n - 2^{- 1 - d} \mathcal{C}_0 + \varepsilon }
+
X^{n   - d   -  ( 2^{- 1 - d} \mathcal{C}_0 -  2 (d-1)  ) \varsigma  + \varepsilon }
\\
\notag
&\ll&
X^{n   - d  - \varepsilon }.
\end{eqnarray}
Let $\vartheta_0 = 1/16 - \varepsilon$ and  $\lambda  = 1/12$ so that (\ref{theta0 major}) is satisfied.
Let us set $0 <  \varsigma < \vartheta_0/(d-1) $ so that
$$
\mathfrak{M}^{(\textnormal{I})} (\varsigma) \subseteq  \mathfrak{M}^+ ( 1/16 - \varepsilon; 1/12),
$$
hence
$$
\mathfrak{m}^+
=  [0,1) \setminus  \mathfrak{M}^{+}( 1/16 - \varepsilon; 1/12) \subseteq  [0,1) \setminus  \mathfrak{M}^{(\textnormal{I})}(\varsigma).
$$
Therefore, we obtain from (\ref{minor arc 1 bdd1+}) that
\begin{eqnarray}
\int_{ \mathfrak{m}^+ } |S (\alpha)| \tsp d \alpha
\leq
\int_{[0,1) \setminus  \mathfrak{M}^{(\textnormal{I})} ( \varsigma) }  |S (\alpha )| \tsp {d} {\alpha}
\ll X^{n - d - \varepsilon}.
\end{eqnarray}
By combining this estimate with Proposition \ref{prop major}, we obtain Theorem \ref{mainthm}
when $F$ satisfies (I) of Definition \ref{dichotomy} as well; this completes the proof of Theorem \ref{mainthm}.

\appendix
\section{}
\label{A}

Let $F \in \mathbb{Z}[x_1, \ldots, x_n]$ be a homogeneous form of degree $d \geq 2$ and $q \in \mathbb{N}$. We define
\begin{equation}
\label{defn Stilde}
\mathcal{S}_{ a, q } = \sum_{\mathbf{k} \in \mathbb{U}_q^{n}  } e \left(
\frac{a}{q} F (\mathbf{k}) \right)
\end{equation}
and
\begin{equation}
\label{def sigular series partial}
A( q ) = \sum_{ \substack{ 1 \leq a \leq q \\   \gcd (a, q) = 1 } } \frac{1}{\phi(q)^{n} }  \mathcal{S}_{ a, q }.
\end{equation}
We know that
$$
A(q_1 q_2) = A(q_1) A(q_2)
$$
for any coprime $q_1, q_2 \in \mathbb{N}$. 
We now bound $\mathcal{S}_{ a, q }$ when $q$ is a prime power.
\begin{lem}
\label{to bound local factor}
Let $p$ be a prime, $t \in \mathbb{N}$ and $q = p^t$. Let $1 \leq a \leq q$ with $\gcd(a, q)=1$.
Let $\varepsilon > 0$ be sufficiently small.
Then 
we have the following bounds
\begin{eqnarray}
\notag
\mathcal{S}_{ a, q } \ll
\left\{
    \begin{array}{ll}
         q^{n  - \frac{\textnormal{codim} \tsp V^*_{F } }{  (d-1) 2^{d-1}  + 2  (2 d - 1)    } + \varepsilon }
         &\mbox{if } t \leq 2d - 1,\\
          q^{n -  \frac{  (2d-1) \textnormal{codim} \tsp V^*_{F}   }{ (d-1) d 2^{d+1} } + \varepsilon}
         &\mbox{if } t > 2d - 1,
    \end{array}
\right.
\end{eqnarray}
where the implicit constants are independent of $p$ and $t$.
\end{lem}

\begin{proof}
We follow the proof of \cite[Lemma 10]{CM}. Let us  consider the two cases $t \leq 2 d - 1$ and $t > 2 d - 1$ separately.
We begin with the case $t \leq 2 d - 1$. In this case,
we apply the inclusion-exclusion principle (see \cite[(7.3)]{CM}) 
and obtain
\begin{eqnarray}
\label{bound on S tilde}
\mathcal{S}_{ a, q } =
\sum_{ I \subseteq \{ 1, 2, \ldots, n \}   } (-1)^{|I|}
\sum_{ \mathbf{v} \in  (\mathbb{Z} / p^{t-1} \mathbb{Z} )^{|I|} } \,
\sum_{\mathbf{k} \in (\mathbb{Z} / q \mathbb{Z})^{n} }
\mathfrak{H}_{I, \mathbf{v}} (\mathbf{k})  \,   e \left( \frac{a}{q}  F (\mathbf{k})  \right),
\end{eqnarray}
where $\mathfrak{H}_{I, \mathbf{v}} (\mathbf{k})$ is the characteristic function of the set
$$
\{ \mathbf{k} \in (\mathbb{Z} / q \mathbb{Z})^{n} : k_{i} = p v_{i} \quad  (i \in I)\}.
$$

We now bound the summand in the expression (\ref{bound on S tilde}) by further considering two cases,
$|I| >  \frac{ (2d -1)  \textnormal{codim} \tsp V^*_{F}   }{ (d-1) 2^{d-1}  + 2  (2 d - 1)  }$ and $|I| \leq  \frac{ (2d -1)  \textnormal{codim} \tsp V^*_{F}   }{ (d-1) 2^{d-1}  + 2  (2 d - 1)  }$.
In the first case $|I| >   \frac{ (2d -1)  \textnormal{codim} \tsp V^*_{F}   }{ (d-1) 2^{d-1}  + 2  (2 d - 1)  }$, we use the following trivial estimate
\begin{eqnarray}
&&\Big{|} \sum_{ \mathbf{v} \in  (\mathbb{Z} / p^{t-1} \mathbb{Z})^{|I|} }  \,
\sum_{\mathbf{k} \in (\mathbb{Z} / q \mathbb{Z})^{n} }
\mathfrak{H}_{I, \mathbf{v}} (\mathbf{k}) \,   e \left( \frac{a}{q}  F (\mathbf{k})  \right) \Big{|}
\label{est 1}
\\
&\leq&
p^{(t-1) |I| } (p^t)^{n  - |I|}
\notag
\\
&=&
\notag
q^{n  -  \frac{|I|}{t} }
\\
&<&
\notag
q^{n  -   \frac{\textnormal{codim} \tsp V^*_{F}   }{ (d-1) 2^{d-1}  + 2  (2 d - 1)  }}.
\end{eqnarray}

On the other hand, suppose  $|I| \leq \frac{ (2d -1)  \textnormal{codim} \tsp V^*_{F}   }{ (d-1) 2^{d-1}  + 2  (2 d - 1)  }$.
Let us label $\mathbf{s} = (s_{1}, \ldots , s_{n - |I|})$
to be the remaining variables of $\mathbf{x}$ after setting $x_i = 0$ for each $i \in I$.
Let $\mathfrak{f} (\mathbf{s})$ be the polynomial obtained by substituting $x_i = p v_{i}$ $(i \in I)$ into the polynomial $F(\mathbf{x})$.
Thus $\mathfrak{f} (\mathbf{s})$ is a polynomial in $\mathbf{s}$ whose coefficients may depend on $p$ and $\mathbf{v}$.
With these notation we have
\begin{eqnarray}
\label{ineq 3}
\\
\sum_{\mathbf{k} \in (\mathbb{Z} / q \mathbb{Z})^{n}} \mathfrak{H}_{I, \mathbf{v}} (\mathbf{k}) \, e \left(  \frac{a}{q}  F(\mathbf{k})   \right)
\notag
&=&
\sum_{  \mathbf{s} \in [0, q-1]^{n - |I|}  } e \left( \frac{a}{q}   \mathfrak{f} (\mathbf{s}) \right).
\end{eqnarray}
We can also deduce easily that the degree $d$ homogeneous portion of the polynomial $\mathfrak{f}(\mathbf{s})$ is obtained by substituting $x_i = 0$ $(i \in I)$ into $F(\mathbf{x})$.
In particular, it is independent of $p$ and $\mathbf{v}$.
It then follows from \cite[Lemma 3.1]{SS} that
\begin{eqnarray}
\textnormal{codim} \tsp V^*_{\mathfrak{f}^{[d]}}
\geq \textnormal{codim} \tsp V^*_{F} - 2 |I|
\notag
\geq
\left( 1 -  \frac{ 2 (2d -1)    }{ (d-1) 2^{d-1}  + 2  (2 d - 1)  }   \right) \textnormal{codim} \tsp V^*_{F}.
\end{eqnarray}
Let $\varepsilon > 0$ be sufficiently small. Thus by \cite[Lemma 5.4]{Bir} we obtain
\begin{eqnarray}
\notag
\sum_{  \mathbf{s} \in [0, q-1]^{n - |I|}  } e \left(  \frac{a}{q}  \mathfrak{f} ( \mathbf{s} )  \right)
\ll q^{n - |I| - \frac{\textnormal{codim} \tsp V^*_{ \mathfrak{f}^{[d]}  }}{ (d-1) 2^{d-1}  } + \varepsilon }
\notag
\leq  q^{n - |I| - \frac{\textnormal{codim} \tsp V^*_{F } }{  (d-1) 2^{d-1}  + 2  (2 d - 1)    } + \varepsilon }.
\end{eqnarray}
Consequently, we have  from (\ref{ineq 3}) that
\begin{eqnarray}
&&\Big{|} \sum_{ \mathbf{v} \in  (\mathbb{Z} / p^{t-1} \mathbb{Z})^{|I|} } \
\sum_{\mathbf{k} \in (\mathbb{Z} / q \mathbb{Z})^{n} }
\mathfrak{H}_{I, \mathbf{v}} (\mathbf{k}) \, e \left( \frac{a}{q}  F (\mathbf{k})   \right) \Big{|}
\label{est 2}
\\
&\leq&
p^{(t-1) |I| }  q^{n - |I| - \frac{\textnormal{codim} \tsp V^*_{F } }{  (d-1) 2^{d-1}  + 2  (2 d - 1)    } + \varepsilon }
\notag
\\
&\leq&
\notag
q^{n  - \frac{\textnormal{codim} \tsp V^*_{F } }{  (d-1) 2^{d-1}  + 2  (2 d - 1)    } + \varepsilon }
\end{eqnarray}
holds in this case as well.
By applying the estimates (\ref{est 1}) and (\ref{est 2}) in (\ref{bound on S tilde}), we obtain
the desired bound for the case $t \leq 2 d - 1$.

We now consider the case $t > 2d - 1$. By the definition of $\mathcal{S}_{ a, q }$, we have
\begin{eqnarray}
\label{widetile S part 3-1}
\mathcal{S}_{ a, q }
=
\sum_{\mathbf{k} \in \mathbb{U}_p^{n}}  \sum_{  \mathbf{b} \in [0,p^{t-1}-1]^{n}  } e \left( \frac{a}{q}  F (\mathbf{k} + p \mathbf{b} )  \right).
\end{eqnarray}
For each fixed $\mathbf{k} \in \mathbb{U}_p^{n}$, we have
$$
F (\mathbf{k} + p \mathbf{b})  = p^{d}  F ( \mathbf{b})
+ g_{p, \mathbf{k}}(\mathbf{b}),
$$
where $g_{p, \mathbf{k}}(\mathbf{b})$ is a polynomial in $\mathbf{b}$ of degree at most $d - 1$ whose coefficients are integers which may depend on $p$ and $\mathbf{k}$.
We let
$$
G(\mathbf{b}) = F ( \mathbf{b}) + \frac{1}{p^{d}}  g_{p, \mathbf{k}} (\mathbf{b}).
$$
We can then express the inner sum on the right hand side of (\ref{widetile S part 3-1}) as
\begin{equation}
\label{**}
\sum_{ \mathbf{b} \in [0,p^{t-1}-1]^{n}   } e \left(  \frac{ a }{q/p^{d}}  G (\mathbf{b} )    \right).
\end{equation}
We have that $G$ has coefficients in $\mathbb{Q}$, while its degree $d$ homogeneous portion $F$ has coefficients in $\mathbb{Z}$.
We apply \cite[Lemma 4.3]{Bir} with
$\mathcal{B} = [0,1)^{n}$, $P = p^{t-1}$ and  $\alpha = a /p^{t - d}$; we recall that the
estimate of this lemma is independent of the lower degree terms of $G(\b)$. Let
$$
\theta = \frac{1}{2(d - 1)}.
$$
Suppose there exist $\widetilde{a}$ and
$1 \leq \widetilde{q} \leq P^{(d - 1) \theta}$
such that $\gcd(\widetilde{a}, \widetilde{q}) = 1$ and
$$
2 |\widetilde{q} \alpha - \widetilde{a}| \leq  P^{- d + (d - 1) \theta}.
$$
Note from $t > 2 d - 1$ it follows that $t - d > \frac{t - 1}{2}$.
Then it is not possible that $p^{t - d}$ divides $\widetilde{q}$, because
$$
1 \leq \widetilde{q} \leq P^{(d - 1) \theta} = p^{ \frac{t-1}{2}} < p^{t - d}.
$$
It follows that $\widetilde{q} \alpha$ is not an integer. Thus we have
$$
\frac{1}{p^{t- d}} \leq |\widetilde{q} \alpha - \widetilde{a}| <   P^{- d +  (d-1) \theta }  =  p^{- d (t-1) + \frac{t-1}{2}},
$$
which is a contradiction.
Therefore, way may apply  \cite[Lemma 4.3]{Bir}, and the expression (\ref{**}) is bounded by
\begin{eqnarray}
\label{***}
\ll
P^{n - \theta  \frac{ \textnormal{codim} \tsp V^*_{F}   }{ 2^{d - 1} }  + \varepsilon }
= (p^{t-1})^{n -  \frac{ \textnormal{codim} \tsp V^*_{F}   }{ (d-1) 2^{d} } + \varepsilon}.
\end{eqnarray}
Thus we can bound ~(\ref{widetile S part 3-1}) by (\ref{**}) and (\ref{***}) as follows
\begin{eqnarray}
| \mathcal{S}_{ a, q } |
\ll
p^{n} (p^{t-1})^{n -  \frac{ \textnormal{codim} \tsp V^*_{F}   }{ (d-1) 2^{d} } + \varepsilon}
\leq
p^{ \frac{ \textnormal{codim} \tsp V^*_{F}   }{ (d-1) 2^{d} } } q^{n -  \frac{ \textnormal{codim} \tsp V^*_{F}   }{ (d-1) 2^{d} } + \varepsilon}
\leq
q^{ \frac{ \textnormal{codim} \tsp V^*_{F}   }{ (d-1) d 2^{d + 1} } } q^{n -  \frac{ \textnormal{codim} \tsp V^*_{F}   }{ (d-1) 2^{d} } + \varepsilon}.
\notag
\end{eqnarray}

\end{proof}

Let
$$
\mathcal{Q} = \min \left(  \frac{\textnormal{codim} \tsp V^*_{F } }{  (d-1) 2^{d-1}  + 2  (2 d - 1)    },   \frac{  (2d-1) \textnormal{codim} \tsp V^*_{F}   }{ (d-1) d 2^{d+1} }    \right).
$$
We now obtain a bound for $A(q)$ by applying Lemma \ref{to bound local factor} in the following manner.
Let $q = p_1^{t_1} \cdots p_{v}^{t_{v}}$ be the prime factorization of $q \in \mathbb{N}$.
Without loss of generality, let us suppose  $t_j \leq 2 d - 1$ $(1 \leq j \leq v_0)$ and
$t_j > 2 d - 1$ $(v_0 < j \leq v)$. Then by applying Lemma \ref{to bound local factor}, with $C > 0$ being the implicit constant, we obtain
\begin{eqnarray}
\label{bound on Aq}
|A(q)|
&=&
|A(p_1^{t_1})|  \cdots   |A(p_{v}^{t_{v}})|
\\
\notag
&\leq&
C^v q^{\varepsilon}
\left( \prod_{j=1}^{v_0} p_j^{ t_j  (1  - \frac{\textnormal{codim} \tsp V^*_{F } }{  (d-1) 2^{d-1}  + 2  (2 d - 1)    }) } \right) \cdot
\left(  \prod_{j=v_0 + 1}^{v}
 p_j^{ t_j(  1 -  \frac{  (2d-1) \textnormal{codim} \tsp V^*_{F}   }{ (d-1) d 2^{d+1} } )}
 \right)
\\
&\leq&
\notag
C^v q^{1  - \mathcal{Q}   + \varepsilon}
\\
\notag
&\ll&
q^{1 - \mathcal{Q}   + 2 \varepsilon},
\end{eqnarray}
where the implicit constant is independent of $q$.

\section{}
\label{B}

In this appendix, we correct the oversight in the derivation of \cite[Lemma 6.3]{Y2}. We note that the statement of \cite[Lemma 6.3]{Y2} remains unaffected.
We will not introduce all the relevant notation here; we refer the reader to see \cite[Section 6]{Y2}.
We begin with \cite[(6.26)]{Y2} which states
\begin{eqnarray}
\label{exp bound with count 1BBB}
|\mathfrak{T}_{\mathbf{U}}({\alpha})|^{2^{2 d  -2}}
\ll
N^{\varepsilon(2 + 2^{d-1})}
\widetilde{X}^{2^{2 d  -2} - d + 1}  \widetilde{Y}^{2^{2 d  -2} - d}
M({\alpha}; \mathbf{X}; \mathbf{Y}; X_1^{-1}, \ldots, X_{m}^{-1}),
\end{eqnarray}
where $M(\alpha; \mathbf{X}; \mathbf{Y}; P_1, \ldots, P_{m})$ is number of integral vectors
$$
\widehat{\mathbf{x}} \in \left( (-X_1, X_1) \times  \cdots \times (-X_{m}, X_{m}) \right)^{d-1}
\ \ \text{  and  }  \  \  \
\underline{\mathbf{y}}\in \left( (-Y_1, Y_1) \times  \cdots \times (-Y_{m}, Y_{m}) \right)^{d}
$$
such that
$$
\| \alpha \Gamma (\widehat{\mathbf{x}}, \mathbf{e}_{\ell} ; \underline{ \mathbf{y} }  )  \| < P_{\ell} \ \ (1 \leq \ell \leq m).
$$

The following lemma on geometry of numbers was obtained in \cite{SS}; this is a slightly more general version of \cite[Lemma 12.6]{D1}.
\begin{lem} \cite[Lemma 2.4]{SS}
\label{shrink lemmaBBB}
Let $L_1, \ldots, L_s$ be symmetric linear forms given by $L_i = \gamma_{i,1}u_1 + \cdots + \gamma_{i,s} u_s$
$(1 \leq i \leq s)$, i.e. such that $\gamma_{i,j} = \gamma_{j,i}$ $(1 \leq i, j \leq s)$. Let $a_1, \ldots, a_s >1$
be real numbers. We denote by $\mathfrak{U}(Z)$ the number of integer solutions $u_1, \ldots, u_{2s}$ to the system of inequalities
$$
|u_i| < a_i Z \ \ (1 \leq i \leq s) \ \ \text{  and  } \ \ |L_i - u_{s + i}| < a_i^{-1} Z \ \ (1 \leq i \leq s).
$$
Then for $0 < Z_1 \leq Z_2 \leq 1$ we have
$$
\frac{\mathfrak{U}(Z_2)}{\mathfrak{U}(Z_1)} \ll \left( \frac{Z_2}{Z_1} \right)^{s},
$$
where the implicit constant depends only on $s$.
\end{lem}

We also need the following lemma analogous to Lemma \ref{shrink lemma2}.

\begin{lem}
\label{shrink lemma2BBB}
Let $a_1, \ldots, a_s > 1$ and $Z > 0$.
Let $L_1, \ldots, L_s$ be linear forms given by $L_i (u_1, \ldots, u_s) = \gamma_{i,1} u_1 + \cdots + \gamma_{i,s} u_s$
$(1 \leq i \leq s)$ and $\mathfrak{U}(Z)$ be as in Lemma \ref{shrink lemmaBBB}. Let us define
$$
L_j^T (u_1, \ldots, u_s) = \gamma_{1,j}u_1 + \cdots + \gamma_{s, j} u_s \quad (1 \leq j \leq s).
$$
We denote by $\mathfrak{U}^T(Z)$ the number of integer solutions $u_1, \ldots, u_s$ to the system of inequalities
$$
|u_i| < a_i Z \quad (1 \leq i \leq s) \quad  \textnormal{ and }  \quad \| L^T_j (u_1, \ldots, u_s) \| < a_j^{-1} Z \quad (1 \leq j \leq s).
$$
Then for $0 < Z_1 \leq Z_2 \leq 1$ we have
$$
\mathfrak{U}(Z_2) \ll \left( \frac{Z_2}{Z_1} \right)^{s} (\mathfrak{U}(Z_1) + \mathfrak{U}^T(Z_1)) ,
$$
where the implicit constant depends only on $s$.
\end{lem}
\begin{proof}
Let $\Lambda$ be the $2s$-dimensional lattice defined by
$$
x_i = {a_i}^{-1} u_i
\quad
\textnormal{and}
\quad
x_{s + i} = {a_i} ( \gamma_{i,1} u_1 + \dots + \gamma_{i,s} u_s + u_{s+i}  ),
$$
for each $1 \leq i \leq s$.
Let us define
$$
\mathfrak{U}_0(Z) = \# \{ \x = (x_1, \dots, x_{2s}) \in \Lambda: |\x|_2 < Z  \}.
$$
Then we have
$$
\mathfrak{U}_0(Z) \leq \mathfrak{U}(Z) \leq \mathfrak{U}_0( \sqrt{2 s} Z).
$$
We identify the lattice $\Lambda$ with its matrix
$$
\Lambda = \begin{bmatrix}
            a_1^{-1} & \dots  & 0 & 0 & \dots  & 0 \\
            \vdots  & \ddots  &  \vdots   & \vdots  &  & \vdots  \\
            0  & \dots  & a_s^{-1} & 0 & \dots  & 0 \\
            {a_1} \gamma_{1,1} & \dots  & {a_1} \gamma_{1, s} & {a_1} & \dots  & 0 \\
            \vdots  &  & \vdots  & \vdots  & \ddots  & \vdots  \\
            {a_s} \gamma_{s,1} & \dots  & {a_s} \gamma_{s,s} & 0 & \dots  & {a_s}
          \end{bmatrix}.
$$
We find that the adjoint matrix is given by
$$
(\Lambda^T)^{-1} = \begin{bmatrix}
              {a_1}  & \dots  & 0 &  - {a_1} \gamma_{1,1} & \dots  & - {a_1} \gamma_{s, 1} \\
            \vdots  & \ddots  & \vdots  & \vdots  &   & \vdots  \\
            0 & \dots  & {a_s}   & - {a_s} \gamma_{1,s} & \dots  & - {a_s} \gamma_{s,s} \\
           0 & \dots  & 0 &  {a_1}^{-1}  & \dots  & 0 \\
            \vdots  &   &  \vdots   & \vdots  & \ddots  & \vdots  \\
            0  & \dots  &0 & 0 & \dots  &  a_s^{-1} \\
          \end{bmatrix}.
$$
Let
$$
\widetilde{\Lambda} = \begin{bmatrix}
              {a_1}^{-1}  & \dots  & 0 &  0 & \dots  & 0 \\
            \vdots  & \ddots  & \vdots  & \vdots  &   & \vdots  \\
            0 & \dots  & {a_s}^{-1}   & 0 & \dots  & 0 \\
             {a_1} \gamma_{1,1} & \dots  &  {a_1} \gamma_{s, 1} &  {a_1}  & \dots  & 0 \\
            \vdots  &   &  \vdots   & \vdots  & \ddots  & \vdots  \\
             {a_s} \gamma_{1,s} & \dots  &  {a_s} \gamma_{s,s} & 0 & \dots  &  a_s \\
          \end{bmatrix}.
$$
Then we have
$$
(\Lambda^T)^{-1} =
\begin{bmatrix}
[0] & I_s \\
- I_s & [0]
\end{bmatrix}
\widetilde{\Lambda}
\begin{bmatrix}
                               [0] & I_s \\
                               - I_s & [0]
                             \end{bmatrix}^{-1}.
$$
Therefore, the lattices $(\Lambda^T)^{-1}$ and  $\widetilde{\Lambda}$ have the same successive minima $S_1, \ldots, S_{2s}$.
Let $R_1, \dots, R_{2s}$ denote the successive minima of the lattice $\Lambda$. Then by \cite[Lemma 12.5]{D1} we have
$$
R_i S_{2s - i + 1} \asymp 1,
$$
for each $1 \leq i \leq 2s$.
Let us define
$$
\mathfrak{U}^T_0(Z) = \# \{ \x = (x_1, \dots, x_{2s}) \in \widetilde{\Lambda}: |\x|_2 < Z  \}.
$$
Then we have
$$
\mathfrak{U}^T_0(Z) \leq \mathfrak{U}^T(Z) \leq \mathfrak{U}^T_0( \sqrt{2 s} Z).
$$
From this point on we may simply follow the proof of \cite[Lemma 3.1]{DS} and we omit the remaining details.
\end{proof}

Let
$$
\mathcal{T} = \max_{1 \leq j \leq m} \  \frac{X_{j}^{-1} }{  Y_{j}^{-1}} = \max_{1 \leq j \leq m} \ \frac{Y_{j} }{  X_{j}} \leq 1.
$$
Let $0 < Q_1, Q_2 \leq 1$ to be chosen later.
For each $0 \leq t \leq d$, let
\begin{eqnarray}
&&N_2^{(t)}(Q_2; \alpha)
\notag
\\
\notag
&=&
\# \left\{
 \widehat{\x} \in \ZZ^{(d-1)m},
 \underline{\y} \in \ZZ^{d m}:
\begin{array}{l}
 \x_1, \ldots, \x_{d-1} \in (- Q_1 X_1, Q_1 X_1) \times  \cdots \times (- Q_1 X_{m}, Q_1 X_{m})   \\
 \y_1,\ldots, \y_t \in   (- Q_2 Y_1, Q_2 Y_1) \times  \cdots \times (- Q_2 Y_{m}, Q_2 Y_{m}) \\
 \y_{t + 1},\ldots, \y_{d} \in   (- Y_1, Y_1) \times  \cdots \times (- Y_{m}, Y_{m}) \\
 \| \alpha \Gamma(\widehat{\x},   \mathbf{e}_\ell ; \underline{\y}  ) \|
 <  Q_2^t Q_1^{d-1}  \mathcal{T} Y_\ell^{-1} \\ (1 \leq \ell \leq m)
\end{array}
\right\}.
\end{eqnarray}
For each $1 \leq t \leq d$, let
\begin{eqnarray}
&&\widetilde{N}_2^{(t)}(Q_2; \alpha)
\notag
\\
\notag
&=&
\# \left\{
 \widehat{\x} \in \ZZ^{(d-1)m},
 \underline{\y} \in \ZZ^{d m}:
\begin{array}{l}
 \x_1, \ldots, \x_{d-1} \in (- Q_1 X_1, Q_1 X_1) \times  \cdots \times (- Q_1 X_{m}, Q_1 X_{m})   \\
 \y_1,\ldots, \y_t \in   (- Q_2 Y_1, Q_2 Y_1) \times  \cdots \times (- Q_2 Y_{m}, Q_2 Y_{m}) \\
 \y_{t + 1},\ldots, \y_{d} \in   (- Y_1, Y_1) \times  \cdots \times (- Y_{m}, Y_{m}) \\
 \| \alpha \Gamma(\widehat{\x},  \y_1;  \mathbf{e}_j,  \y_2, \ldots, \y_d  ) \|
 <  Q_2^t Q_1^{d-1}  \mathcal{T} Y_j^{-1} \\ (1 \leq j \leq m)
\end{array}
\right\}.
\end{eqnarray}
Instead of \cite[(6.27)]{Y2} we obtain the following.
\begin{lem}
We have
\begin{eqnarray}
\label{ineq 1-4BBB}
M({\alpha}; \mathbf{X}; \mathbf{Y}; X_1^{-1}, \ldots, X_{m}^{-1}) \ll
Q_1^{- m (d-1) } Q_2^{- m  d}
( N_2^{(d)}(Q_2; \alpha) +  \widetilde{N}_2^{(d)}(Q_2; \alpha)  ).
\end{eqnarray}
\end{lem}
\begin{proof}
For each $0 \leq t \leq d-1$, let
\begin{eqnarray}
N_1^{(t)}(Q_1; \alpha)
\notag
=
\# \left\{
\widehat{\x} \in \ZZ^{(d-1)m},  \underline{\y} \in \ZZ^{d m}:
\begin{array}{l}
 \x_1,\ldots, \x_t \in  (- Q_1 X_1, Q_1 X_1) \times  \cdots \times (- Q_1 X_{m}, Q_1 X_{m}) \\
  \x_{t + 1},\ldots, \x_{d-1} \in (-X_1, X_1) \times  \cdots \times (-X_{m}, X_{m}) \\
  \y_1, \ldots, \y_d \in  (-Y_1, Y_1) \times  \cdots \times (-Y_{m}, Y_{m})   \\
 \| \alpha \Gamma(   \widehat{\x},  \mathbf{e}_{\ell}; \underline{\y}  ) \|  <  Q_1^{t} X_\ell^{-1} \quad (1 \leq \ell \leq m)
\end{array}
\right\}.
\end{eqnarray}
First we prove
\begin{equation}\label{N0 boundBBB}
M({\alpha}; \mathbf{X}; \mathbf{Y}; X_1^{-1}, \ldots, X_{m}^{-1}) = N_1^{(0)}(Q_1; \alpha) \ll Q_1^{ - (d-1) m  } N_1^{(d-1)}(Q_1; \alpha),
\end{equation}
which in turn follows by showing that
$$
N_1^{(t-1)}(Q_1; \alpha) \ll Q_1^{- m}  N_1^{(t)}(Q_1; \alpha),
$$
for each $1 \leq t \leq d-1$. This will be a straightforward consequence of Lemma \ref{shrink lemmaBBB}. To see this,
fix $1 \leq t \leq d-1$,  $\underline{\y} \in (-Y_1, Y_1) \times  \cdots \times (-Y_{m}, Y_{m})$ and
$\x_i \in \ZZ^h$, for each $i \neq t$, satisfying
\begin{eqnarray}
\label{cond++++++BBB}
\x_1,\ldots, \x_{t-1} \in  (- Q_1 X_1, Q_1 X_1) \times  \cdots \times (- Q_1 X_{m}, Q_1 X_{m})
\end{eqnarray}
and
\begin{eqnarray}
\notag
 \x_{t + 1},\ldots, \x_{d-1} \in (-X_1, X_1) \times  \cdots \times (-X_{m}, X_{m}).
\end{eqnarray}
Let
\begin{eqnarray}
L_{\ell} (\x)
&=&
\alpha  \Gamma (\x_{1}, \ldots, \x_{t-1}, \x , \x_{t + 1},  \ldots, \x_{d-1}, \mathbf{e}_{\ell};  \underline{\y} )
\notag
\\
\notag
&=&
\sum_{j = 1}^m \alpha  \Gamma (  \x_{1}, \ldots, \x_{t-1}, \mathbf{e}_j , \x_{t + 1},  \ldots, \x_{d-1}, \mathbf{e}_{\ell} ;  \underline{\y}) x_j
\end{eqnarray}
for each $1 \leq \ell \leq m$.
We write
\begin{eqnarray}
M_1(Z)  &=&
\# \{
\x \in (-  Q_1^{ -  \frac{t - 1 }{2}  } X_1 Z  , Q_1^{ -  \frac{t - 1 }{2}  } X_1 Z) \times \cdots \times (-  Q_1^{ -  \frac{t - 1 }{2}  } X_m Z  , Q_1^{ -  \frac{t - 1 }{2}  } X_m Z)  :
\notag
\\
\notag
&& \|L_{\ell}(\x) \| <   Q_1^{ \frac{t - 1 }{2}  } X_\ell^{-1} Z  \quad (1 \leq \ell \leq m)  \}
\end{eqnarray}
for any  $0 < Z \leq 1$. By applying Lemma \ref{shrink lemmaBBB} with
$$
Z_1 = Q_1^{ \frac{ t+1 }{2} }, \quad Z_2 = Q_1^{ \frac{t - 1 }{2} }
\quad \textnormal{and}
\quad a_\ell =  Q_1^{ -  \frac{t - 1 }{2}  } X_\ell \quad (1 \leq \ell \leq m),
$$
we obtain
$$
\frac{M_1(  Z_2  )  }{M_1(  Z_1  )} \ll Q_1^{-m}.
$$
Summing the inequality over all
$\underline{\y} \in (-Y_1, Y_1) \times  \cdots \times (-Y_{m}, Y_{m})$ and $\x_{i}$ with $i \neq t$ such that (\ref{cond++++++BBB}) holds,  we arrive at the inequality
$N_1^{(t-1)}(Q_1; \alpha) \ll Q_1^{-m}  N^{(t)}(Q_1; \alpha),$
which is what we were supposed to prove.

Since $X_\ell^{-1} \leq \mathcal{T} Y_\ell^{-1}$ $(1 \leq \ell \leq m)$, we have
$$
N_1^{(d-1)}(Q_1; \alpha) \leq N_2^{(0)}(Q_2; \alpha).
$$
Next we prove
\begin{equation}\label{N0 bound+BBB}
N_2^{(0)}(Q_2; \alpha) \ll Q_2^{ - d m  } ( N_2^{(d)}(Q_2; \alpha) +  \widetilde{N}_2^{(d)}(Q_2; \alpha)  ),
\end{equation}
which in turn follows by showing that
$$
N_2^{(t-1)}(Q_2; \alpha) \ll  Q_2^{- m} (  N_2^{(t)}(Q_2; \alpha)   +  \widetilde{N}_2^{(t)}(Q_2; \alpha)  ),
$$
for each $1 \leq t \leq d$, and
$$
\widetilde{N}_2^{(t-1)}(Q_2; \alpha) \ll  Q_2^{- m}  \widetilde{N}_2^{(t)}(Q_2; \alpha),
$$
for each $2 \leq t \leq d$.

We begin with the first claim. Let us fix $1 \leq t \leq d$, $\widehat{\x} \in (  (- Q_1 X_1, Q_1 X_1) \times  \cdots \times (- Q_1 X_{m}, Q_1 X_{m}) )^{d-1} $ and
$\y_i \in \ZZ^m$, for each $i \neq t$, satisfying
\begin{eqnarray}
\label{cond+++000BBB}
\y_1,\ldots, \y_{t-1} \in  (- Q_2 Y_1, Q_2 Y_1) \times  \cdots \times (- Q_2 Y_{m}, Q_2 Y_{m})
\end{eqnarray}
and
\begin{eqnarray}
\notag
 \y_{t + 1},\ldots, \y_{d} \in   (-  Y_1,  Y_1) \times  \cdots \times (-  Y_{m},  Y_{m}).
\end{eqnarray}
Let
\begin{eqnarray}
L_\ell (\y)
&=&
\alpha  \Gamma (  \widehat{\x},   \mathbf{e}_{\ell};   \y_{1}, \ldots, \y_{t-1}, \y , \y_{t + 1},  \ldots, \y_{d})
\notag
\\
\notag
&=& \sum_{j=1}^m  \alpha  \Gamma (  \widehat{\x},   \mathbf{e}_{\ell};    \y_{1}, \ldots, \y_{t-1}, \mathbf{e}_j , \y_{t + 1}, \ldots, \y_{d}) y_j
\end{eqnarray}
for each $1 \leq \ell \leq m$, and
\begin{eqnarray}
L_{j}^T (\y)
&=&
\sum_{\ell=1}^m  \alpha  \Gamma (  \widehat{\x},   \mathbf{e}_{\ell}; \y_{1}, \ldots, \y_{t-1}, \mathbf{e}_j , \y_{t + 1},  \ldots, \y_{d}) y_\ell
\notag
\\
\notag
&=& \alpha   \Gamma (  \widehat{\x},   \y; \y_{1}, \ldots, \y_{t-1}, \mathbf{e}_j , \y_{t + 1},  \ldots, \y_{d})
\end{eqnarray}
for each $1 \leq j \leq m$.
We write
\begin{eqnarray}
M_2(Z)
\notag
&=&
\# \{
\y \in  \prod_{\ell = 1}^m (- Q_2^{ - \frac{t - 1 }{2} } Q_1^{- \frac{d-1}{2}} \mathcal{T}^{ - \frac{1}{2}}  Y_\ell Z,   Q_2^{ - \frac{t - 1 }{2} } Q_1^{- \frac{d-1}{2}} \mathcal{T}^{ - \frac{1}{2}}  Y_\ell Z)  :
\\
\notag
&& \| L_{\ell}(\y) \| <   Q_2^{ \frac{t - 1 }{2} } Q_1^{\frac{d-1}{2}} \mathcal{T}^{ \frac{1}{2}} Y_{\ell}^{-1} Z \quad (1 \leq \ell \leq m) \}
\end{eqnarray}
and
\begin{eqnarray}
M_2^T(Z)
\notag
&=&
\# \{
\y \in  \prod_{\ell = 1}^m (- Q_2^{ - \frac{t - 1 }{2} } Q_1^{- \frac{d-1}{2}} \mathcal{T}^{ - \frac{1}{2}}  Y_\ell Z,   Q_2^{ - \frac{t - 1 }{2} } Q_1^{- \frac{d-1}{2}} \mathcal{T}^{ - \frac{1}{2}}  Y_\ell Z) :
\\
\notag
&& \| L^T_{j}(\y) \| <   Q_2^{ \frac{t - 1 }{2} } Q_1^{\frac{d-1}{2}} \mathcal{T}^{ \frac{1}{2}} Y_j^{-1}  Z \quad (1 \leq j \leq m) \}
\end{eqnarray}
for any  $0 < Z \leq 1$. By applying Lemma \ref{shrink lemma2BBB} with
$$
Z_1 = Q_2^{ \frac{t + 1 }{2} } Q_1^{\frac{d-1}{2}} \mathcal{T}^{\frac{1}{2}} , \quad Z_2 = Q_2^{ \frac{t - 1 }{2} } Q_1^{\frac{d-1}{2}} \mathcal{T}^{\frac{1}{2}}
$$
and
$$
a_\ell =  Q_2^{ - \frac{t - 1 }{2} } Q_1^{- \frac{d-1}{2}} \mathcal{T}^{ - \frac{1}{2}}  Y_\ell \quad (1 \leq \ell \leq m),
$$
we obtain
$$
M_2( Z_2 )   \ll Q_2^{-m}  (M_2( Z_1  ) + M_2^T( Z_1)).
$$
Summing the inequality over all
$\widehat{\x} \in (  (- Q_1 X_1, Q_1 X_1) \times  \cdots \times (- Q_1 X_{m}, Q_1 X_{m}) )^{d-1} $ and
$\y_{i}$ with $i \neq t$ such that (\ref{cond+++000BBB}) holds,
we arrive at the inequality
\begin{eqnarray}
&&N_2^{(t-1)}(Q_2; \alpha)
\notag
\\
\notag
&\ll&  Q_2^{-m}  N^{(t)}(Q_2; \alpha) +  Q_2^{-m} \cdot
\\
\notag
&& \# \left\{
 \widehat{\x} \in \ZZ^{(d-1)m},
 \underline{\y} \in \ZZ^{d m}:
\begin{array}{l}
 \x_1, \ldots, \x_{d-1} \in (- Q_1 X_1, Q_1 X_1) \times  \cdots \times (- Q_1 X_{m}, Q_1 X_{m})   \\
 \y_1,\ldots, \y_t \in   (- Q_2 Y_1, Q_2 Y_1) \times  \cdots \times (- Q_2 Y_{m}, Q_2 Y_{m}) \\
 \y_{t + 1},\ldots, \y_{d} \in   (- Y_1, Y_1) \times  \cdots \times (- Y_{m}, Y_{m}) \\
 \| \alpha \Gamma(\widehat{\x},  \y_t;  \y_1, \ldots, \y_{t-1},  \mathbf{e}_j, \y_{t+1}, \ldots, \y_d  ) \|
 <  Q_2^t Q_1^{d-1}  \mathcal{T} Y_j^{-1} \\ (1 \leq j \leq m)
\end{array}
\right\}
\\
\notag
&\ll& Q_2^{-m}  N_2^{(t)}(Q_2; \alpha) +  Q_2^{-m} \widetilde{N}_2^{(t)}(Q_2; \alpha).
\end{eqnarray}

Next we prove the second claim. This again will be a straightforward consequence of Lemma \ref{shrink lemmaBBB}. To see this,
fix $2 \leq t \leq d$,  $\widehat{\x} \in (  (- Q_1 X_1, Q_1 X_1) \times  \cdots \times (- Q_1 X_{m}, Q_1 X_{m}) )^{d-1} $ and
$\y_i \in \ZZ^m$, for each $i \neq t$, satisfying
\begin{eqnarray}
\label{cond+++BBB}
\y_1,\ldots, \y_{t-1} \in  (- Q_2 Y_1, Q_2 Y_1) \times  \cdots \times (- Q_2 Y_{m}, Q_2 Y_{m})
\end{eqnarray}
and
\begin{eqnarray}
\notag
 \y_{t + 1},\ldots, \y_{d} \in   (-  Y_1,  Y_1) \times  \cdots \times (-  Y_{m},  Y_{m}).
\end{eqnarray}
Let
\begin{eqnarray}
L_{j} (\y)
&=&
\alpha  \Gamma ( \widehat{\x},  \y_1;  \mathbf{e}_{j}, \y_{2}, \ldots, \y_{t-1}, \y , \y_{t + 1},  \ldots, \y_{d})
\notag
\\
\notag
&=& \sum_{i=1}^m  \alpha  \Gamma (\widehat{\x},   \y_1;  \mathbf{e}_{j}, \y_{2}, \ldots, \y_{t-1}, \mathbf{e}_i , \y_{t + 1},  \ldots, \y_{d}) y_i
\end{eqnarray}
for each $1 \leq j \leq m$. We write
\begin{eqnarray}
M_3(Z)
\notag
&=&
\# \{
\y \in  \prod_{\ell = 1}^m (- Q_2^{ - \frac{t - 1 }{2} } Q_1^{- \frac{d-1}{2}} \mathcal{T}^{ - \frac{1}{2}}  Y_\ell Z,   Q_2^{ - \frac{t - 1 }{2} } Q_1^{- \frac{d-1}{2}} \mathcal{T}^{ - \frac{1}{2}}  Y_\ell Z) :
\\
\notag
&& \|L_j(\y) \| <   Q_2^{ \frac{t - 1 }{2} } Q_1^{\frac{d-1}{2}} \mathcal{T}^{ \frac{1}{2}} Y_j^{-1}  Z \quad (1 \leq j \leq m) \}
\end{eqnarray}
for any  $0 < Z \leq 1$. By applying Lemma \ref{shrink lemmaBBB} with
$$
Z_1 = Q_2^{ \frac{t + 1 }{2} } Q_1^{\frac{d-1}{2}} \mathcal{T}^{\frac{1}{2}} , \quad Z_2 = Q_2^{ \frac{t - 1 }{2} } Q_1^{\frac{d-1}{2}} \mathcal{T}^{\frac{1}{2}}
$$
and
$$
a_\ell =  Q_2^{ - \frac{t - 1 }{2} } Q_1^{- \frac{d-1}{2}} \mathcal{T}^{ - \frac{1}{2}}  Y_\ell \quad (1 \leq \ell \leq m),
$$
we obtain
$$
\frac{M_3( Z_2 )  }{M_3( Z_1  )} \ll Q_2^{- m}.
$$
Summing the inequality over all
$\widehat{\x} \in (  (- Q_1 X_1, Q_1 X_1) \times  \cdots \times (- Q_1 X_{m}, Q_1 X_{m}) )^{d-1} $ and
$\y_{i}$ with $i \neq t$ such that (\ref{cond+++BBB}) holds,  we arrive at the inequality
$\widetilde{N}_2^{(t-1)}(Q_2; \alpha) \ll Q_2^{-m}  \widetilde{N}_2^{(t)}(Q_2; \alpha),$
which is what we were supposed to prove. Finally, the result follows on combining (\ref{N0 boundBBB}) and (\ref{N0 bound+BBB}).
\end{proof}

Finally, by making use of (\ref{ineq 1-4BBB}) instead of \cite[(6.27)]{Y2} in the proof of \cite[Lemma 6.3]{Y2},
we may recover the statement of \cite[Lemma 6.3]{Y2} in a similar manner as in the proof of Lemma \ref{6.2}.


\textbf{Data availability.} No data were used in the preparation of this paper.

\end{document}